\documentclass[a4paper,10pt]{article}
\usepackage[letterpaper,margin=0.80in,bottom=0.90in,top=0.85in]{geometry}
\usepackage[english]{babel}
\usepackage[utf8]{inputenc}
\usepackage{amsmath,mathrsfs}
\usepackage{indentfirst}
\usepackage{fancyhdr}

\usepackage{amssymb}
\usepackage{amsthm}
\usepackage[T1]{fontenc}
\usepackage{lmodern}
\usepackage[dvips]{graphicx}
\usepackage{color}
\usepackage{cancel}
\usepackage{subcaption}
\usepackage[dvipsnames]{xcolor}
\usepackage{eucal}
\usepackage{latexsym}
\usepackage{chngcntr}
\usepackage{faktor}
\usepackage{microtype}
\usepackage{enumitem}
\usepackage{hyperref,geometry,mathtools} 
\usepackage{todonotes}

\usepackage{amsmath, amsfonts, amsthm,amssymb,bbm}
\usepackage{tikz-cd}
\usepackage{multicol}
\usepackage{cancel}
\usepackage{mathtools}
\usepackage{graphicx}

\newcommand{\kB}{\mathfrak{B}}
\newcommand{\Eta}{{\mathcal E}}
\newcommand{\tth}{\mathtt{h}}
\newcommand{\ttf}{\mathtt{f}}

\newcommand{\te}{\mathtt{e}}

\newcommand{\teWB}{\mathtt{e}_{\scriptscriptstyle{\textsc{WB}}}}
\newcommand{\etaWB}{\eta_{\scriptscriptstyle{\textsc{WB}}}}

\newcommand{\DeltaBF}{\Delta_{\scriptscriptstyle{\textsc{BF}}}}
\newcommand{\tthWB}{\mathtt{h}_{\scriptscriptstyle{\textsc{WB}}}}

\DeclareMathOperator*{\sgn}{sgn}

\DeclareMathOperator*{\eq}{=}

\newcommand{\ta}{{\mathtt a}}

\newcommand{\e}{\epsilon}
\newcommand{\im}{\mathrm{i}\,}

\newcommand{\molt}[2]{\left(#1\,,\,#2\right)}
\newcommand{\lie}[2]{\left[#1\,,\, #2\right]}

\newcommand{\tg}{{\mathtt g}}
\newcommand\restr[2]{{
  \left.\kern-\nulldelimiterspace 
  #1 
  \vphantom{\big|} 
  \right|_{#2} 
  }}

\allowdisplaybreaks
\theoremstyle{plain}
\newtheorem{lem}{Lemma}
\newtheorem{teo}[lem]{Theorem}
\newtheorem{prop}[lem]{Proposition}

\theoremstyle{definition}

\newtheorem{rmk}[lem]{Remark}

\renewcommand{\bar}{\overline}

\newcommand{\vet}[2]{\begin{bmatrix}#1 \\ #2 \end{bmatrix}}
\newcommand{\sepvet}[2]{\begin{bmatrix}#1 \\[4mm] #2 \end{bmatrix}}
\newcommand{\uno}{\mathrm{Id}}

\newcommand{\bR}{\mathbb{R}}
\newcommand{\bT}{\mathbb{T}}
\newcommand{\bZ}{\mathbb{Z}}
\newcommand{\bN}{\mathbb{N}}

\newcommand{\bC}{\mathbb{C}}

\newcommand{\tf}{\mathtt{f}}

\newcommand{\sL}{\mathscr{L}}
\newcommand{\cO}{\mathcal{O}}

\newcommand{\cJ}{\mathcal{J}}
\newcommand{\cB}{{\cal B}}
\newcommand{\cL}{\mathcal{L}}
\newcommand{\cV}{\mathcal{V}}
\newcommand{\tB}{\mathtt{B}}
\newcommand{\tJ}{\mathtt{J}}
\newcommand{\tL}{\mathtt{L}}
\newcommand{\de}{\mathrm{d}}
\newcommand{\pa}{\partial}

\newcommand{\cH}{\mathcal{H}}

\newcommand{\cF}{\mathcal{F}}

\newcommand{\cU}{\mathcal{U}}
\newcommand{\cW}{\mathcal{W}}

\newcommand{\off}{\varnothing}

\newcommand{\bro}{\bar\rho}
\newcommand{\Gc}{\mathfrak{c}}
\newcommand{\inv}{\rotatebox[origin=c]{90}{$\perp$}}

\newcommand{\tb}{{\mathtt{b}}}

\newcommand{\ch}{{\mathtt c}_{\mathtt h}}
\newcommand{\wt}[1]{\widetilde #1}

\newcommand{\ka}{\mathfrak{a}}
\newcommand{\kb}{\mathfrak{b}}

\newcommand{\kn}{\mathfrak{n}}
\newcommand{\km}{\mathfrak{m}}
\newcommand{\ku}{\mathfrak{u}}
\newcommand{\kh}{\mathfrak{h}}
\newcommand{\kp}{\mathfrak{p}}

\numberwithin{equation}{section}
\counterwithin{lem}{section}

\title{\bf Stokes waves at the critical depth \\
are modulational unstable}

\begin{document}
\author{Massimiliano Berti, Alberto Maspero, Paolo Ventura\footnote{
International School for Advanced Studies (SISSA), Via Bonomea 265, 34136, Trieste, Italy. 
 \textit{Emails: } \texttt{berti@sissa.it},  \texttt{alberto.maspero@sissa.it}, \texttt{paolo.ventura@sissa.it}
 }}

\date{}

\maketitle

\begin{abstract}
This paper fully answers a long standing open question concerning the stability/instability 
of pure gravity periodic traveling 
water  waves --called Stokes waves--
at the critical Whitham-Benjamin depth $ \tthWB = 1.363... $
and nearby values.  
We prove that 
Stokes waves of small amplitude  $ \cO ( \epsilon ) $
are, at the critical depth $ \tthWB $,  linearly unstable under long wave perturbations. 
This is also true for 
slightly smaller values of the depth $ \tth >  \tthWB - c \e^2 $, $ c > 0 $, depending on the amplitude of the wave. 
This problem 
was not rigorously solved   in previous literature 
because the  
expansions 
degenerate at the critical depth. 
In  order to resolve this degenerate case, and describe in a mathematically 
exhaustive 
way  
how the 
eigenvalues change their stable-to-unstable nature 
along this shallow-to-deep water transient, 
we Taylor  expand the 
computations of \cite{BMV3} at a higher degree of accuracy, 
derived by the fourth order 
expansion of the Stokes waves. We prove that also in this transient regime 
a pair of unstable eigenvalues depict a closed figure 8, of smaller size than for 
$ \tth >  \tthWB $, as the Floquet exponent varies.
\end{abstract}

\tableofcontents

\section{Introduction  and main result}

In the last years substantial mathematical progresses have been obtained in the classical problem of determining the stability/instability of the Stokes waves, i.e. periodic traveling waves of the gravity water waves equations in any depth, 
subject to long wave perturbations.

Let us briefly summarize the state of the art. 
The existence of small amplitude
 Stokes waves, pioneered by the famous work of Stokes \cite{stokes} in 1847,  
 was first rigorously proved by Struik \cite{Struik}, Levi-Civita \cite{LC}, 
and Nekrasov \cite{Nek} one century ago, and then extended 
to  global branches containing extreme 
waves in \cite{KN,To,ML,CS,AFT,Plo}.
 In the sixties Benjamin and Feir \cite{BF,Benjamin}, Whitham \cite{Whitham},
 Lighthill   \cite{Li} and Zakharov \cite{Z0,ZK}
discovered, through experiments and formal arguments, that small amplitude 
Stokes waves in 
sufficiently deep water are  unstable, proposing a heuristic mechanism which  leads 
to the disintegration of wave trains.
More precisely, these works predicted the existence of a critical depth 
--that we shall call the  {\it Whitham-Benjamin depth}--
$$ 
\tthWB := 1.363... \, , 
$$
such that  $ 2\pi\kappa $-space periodic Stokes waves in an ocean of  depth
$   \mathtt h > \tthWB \kappa^{-1} $ are unstable: in correspondence with 
small Floquet exponents $ \mu $  
(i.e. long-wave perturbations) 
the linearized equations 
 at the Stokes wave possess 
a pair of 
eigenvalues with non-zero real part close to zero.
This phenomenon is nowadays called 
``Benjamin-Feir'' --or modulational-- instability, and it is supported by an enormous  amount of  physical observations and numerical simulations, see e.g. \cite{DO,KDZ,DDLS}. 
 We refer to \cite{ZO} for a historical survey of the modulational theory of wave packets
for several dispersive and fluid PDE models.
We  remark that modulational instability has indeed been observed 
also in a  variety of approximate water waves models, such as KdV, gKdV, NLS and the Whitham equation, see \cite{Wh,SHCH,GH,HK,BJ,J,HJ,BHJ,HP,LBJM}.

\smallskip
 
For the water waves equations, the first mathematically 
 rigorous proof of a local branch of unstable Benjamin-Feir eigenvalues close to zero 
for  $ \kappa \mathtt h >  
  \tthWB $
was obtained   by 
Bridges-Mielke \cite{BrM}  in finite depth (see also   Hur-Yang 
\cite{HY}) 
via 
a center manifold reduction,
and recently  by Nguyen-Strauss \cite{NS} 
via a Lyapunov-Schmidt decomposition. In deep water we mention the nonlinear modulational result by Chen-Su \cite{ChenSu}.

Very recently  Berti-Maspero-Ventura \cite{BMV1,BMV2} developed a 
completely different rigorous spectral approach,
based on a symplectic version of Kato's theory of similarity transformations
and a block diagonalization technique   inspired by KAM theory, 
which provided  the full topological 
splitting of all the four eigenvalues close to zero 
as the Floquet exponent $ \mu $ is turned on. More precisely the works 
\cite{BMV1,BMV2},  
that, with no loss of generality, are formulated for 
$2\pi$-periodic Stokes waves, i.e. with wave number $\kappa=1$, 
rigorously prove that: 

\begin{itemize}
\item{{\it Shallow water case:}} for any $ 0 < \mathtt h < \tthWB $  
the eigenvalues close to zero
are purely imaginary for Stokes waves
of sufficiently small amplitude $ \epsilon $, see Figure \ref{figure-eigth}-left; 
\item{\it Sufficiently deep water case:} 
for any $ \tthWB  < \mathtt h \leq \infty $, 
there exists a pair of 
eigenvalues with non-zero real part, which 
traces a complete closed figure ``8'' (as shown in  Figure \ref{figure-eigth}-right) parameterized by the Floquet exponent $ \mu $.
As $ {\mathtt h}  \to \tthWB^{\, +} $ the set of unstable Floquet exponents shrinks to zero and the Benjamin-Feir unstable eigenvalues 
collapse to the origin  
(we remark that 
the case 
$ \tth =  + \infty $ 
 is  not a consequence of the  finite depth case). 
The figure 8 had been numerically observed in Deconinck-Oliveras  \cite{DO}.
\end{itemize}
 \begin{figure}[h]
 \centering
 \subcaptionbox*{}[.4\textwidth]{\includegraphics[width=4.5cm]{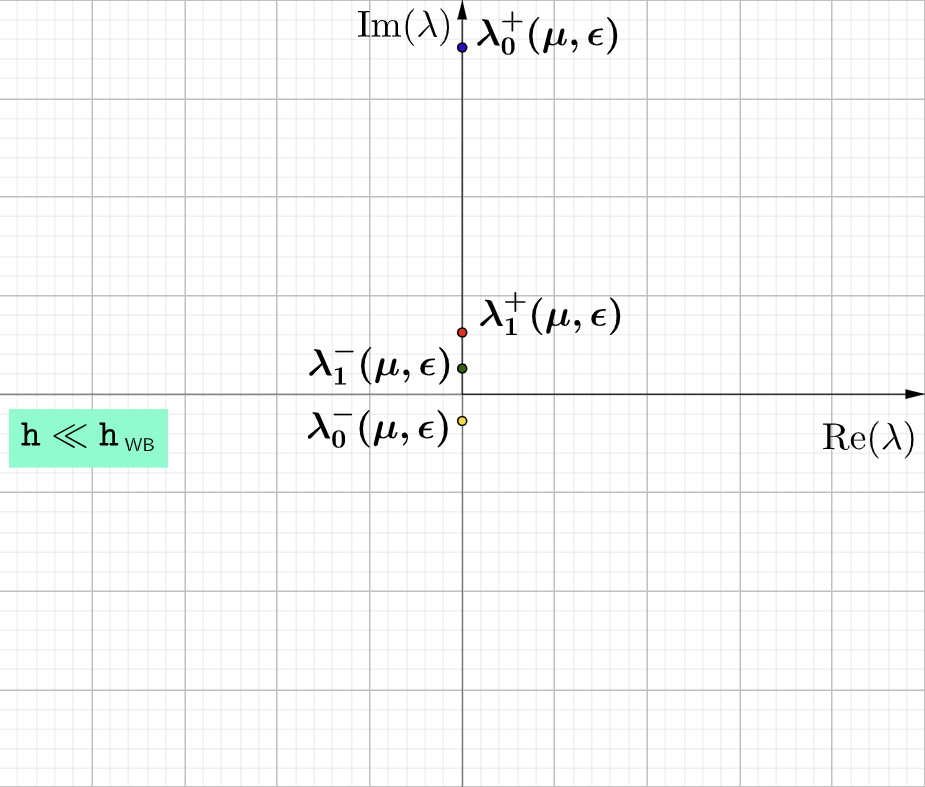}}\hspace{1cm}
\subcaptionbox*{}[.4\textwidth]{
\includegraphics[width=4.5cm]{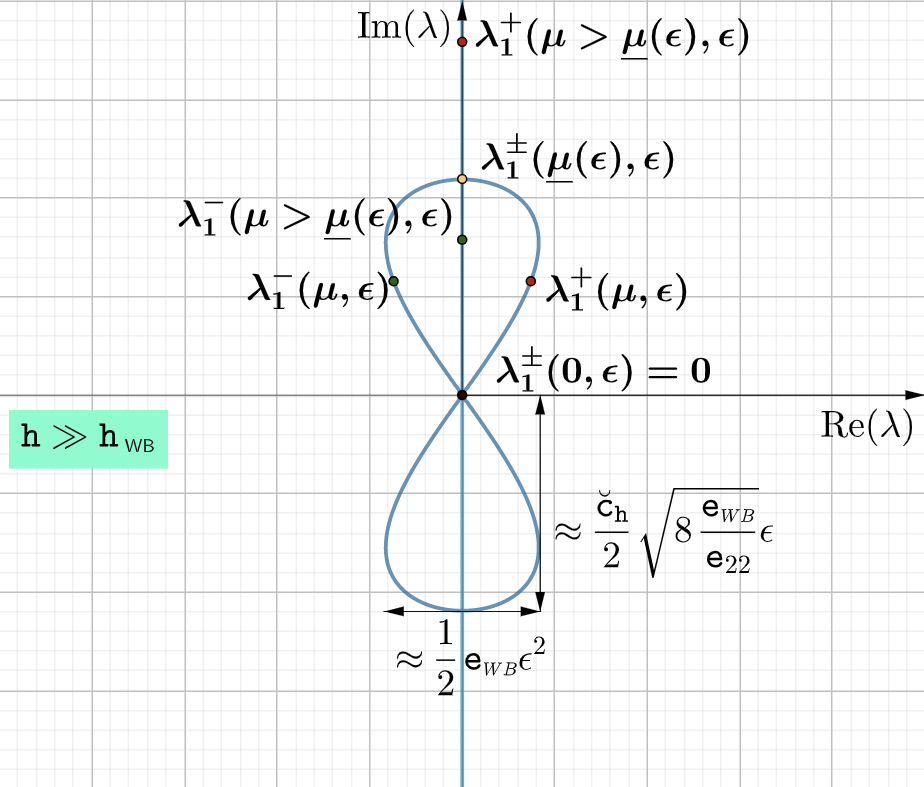} } \vspace{-0.8cm}
\caption{The left picture  shows that for $\tth< \tthWB $ the  eigenvalues $\lambda^\pm_{1} (\mu,\epsilon )$ and  $\lambda^\pm_{0} (\mu,\epsilon )$ are purely imaginary.
 The picture on the right shows that  for  $\tth> \tthWB $
  the  eigenvalues $\lambda^\pm_1 (\mu,\epsilon )$ 
  at fixed $|\e| \ll 1 $ as $\mu$ varies. This figure ``8 '' depends on $\tth$ and shrinks to $0$ as $\tth\to \tthWB^+$. Some higher order formal expansions have been 
  computed in Creedon-Deconinck \cite{CD}.  \label{figure-eigth}}
\end{figure}

The question which remains open is to determine the stability or instability 
of the Stokes waves at the critical Whitham-Benjamin depth $ \tthWB $ and to analyze 
in detail the change of stable-vs-unstable behavior of the eigenvalues along this shallow-to-deep water transient.

Some formal answers have been given so far \cite{J, Kakutani, slunyaev, sedletsky}.  Water waves solutions  
in the modulational instability ansatz 
are formally approximated by an equation for the wave envelope which is, 
if  $ \tth  < \tthWB $, a defocusing cubic nonlinear Sch\"rodinger equation (NLS)   
 whereas,  
if $ \tth  > \tthWB $,  it is a focusing cubic NLS (this is in agreement with the 
rigorous stability/instability results in shallow/deep water stated above).
The transient behaviour at the critical depth $ \tth  = \tthWB $ corresponds to the vanishing of the cubic coefficients, and, in this case, it is required to determine a  
higher order effective NLS. 
 In the seventies,    formal computations by Johnson \cite{John} 
 suggested the stability of the Stokes waves for nearby 
larger values of $ \tth >  \tthWB $. 
Some years later Kakutani-Michihiro \cite{Kakutani} derived a different quintic NLS equation, actually claiming the modulational instability of Stokes waves. The  instability  was further confirmed 
by 
Slunyaev \cite{slunyaev} who computed how the coefficients of the quintic effective 
NLS depend on $\tth$. 

In this paper we  mathematically rigorously prove that 
 this last scenario is what happens:  Stokes waves  
 of the pure gravity water waves equations at the critical depth are linearly unstable under long wave perturbations.
Unlike all the previous mentioned works, we do not use any formal 
approximation argument with some quintic NLS equation, but prove directly the existence of unstable 
eigenvalues of  the linearized water waves equations at the Stokes wave.
 Informally speaking the main result 
we prove  is  the following:  
\begin{teo}\label{teoin} {\bf (Modulational instability of the Stokes wave at} 
$ \tth = \tthWB ${\bf )}
If $ \tth = \tthWB $ then  
small 
amplitude  Stokes waves of amplitude $ \cO (\e) $ 
are linearly unstable subject to long wave perturbations.  
Actually
 Stokes waves are modulational unstable also at nearby depths $ \tth < \tthWB $:  
there is an analytic function defined for $ \epsilon $ small,  
of the form 
$$
 \underline{ \tth}(\e)  = - c \e^2   + \cO(\e^3) \, , \quad c > 0 \, , 
$$
such that, for any  $(\tth, \epsilon) $ satisfying   
\begin{equation}\label{tunpo}
\tth >  \tthWB +    \underline{ \tth}(\e) \, , 
\end{equation}
then the linearized equations 
at the Stokes wave have two eigenvalues with nontrivial real part for any 
Floquet exponent $ \mu $  small enough, see Figure \ref{fig:zone}.
In particular for $ \tth = \tthWB $  the unstable eigenvalues depict a 
closed figure ``8'' as  $ \mu $ varies 
in an interval of size $[0, c_1 \e^2)$, 
see Figure \ref{fig:hWB}, whose height and width 
are much smaller 
than for $ \tth > \tthWB  $, see  Figure \ref{figure-eigth}. 
\end{teo}

\begin{figure}
 \centering
\includegraphics[scale=0.37]{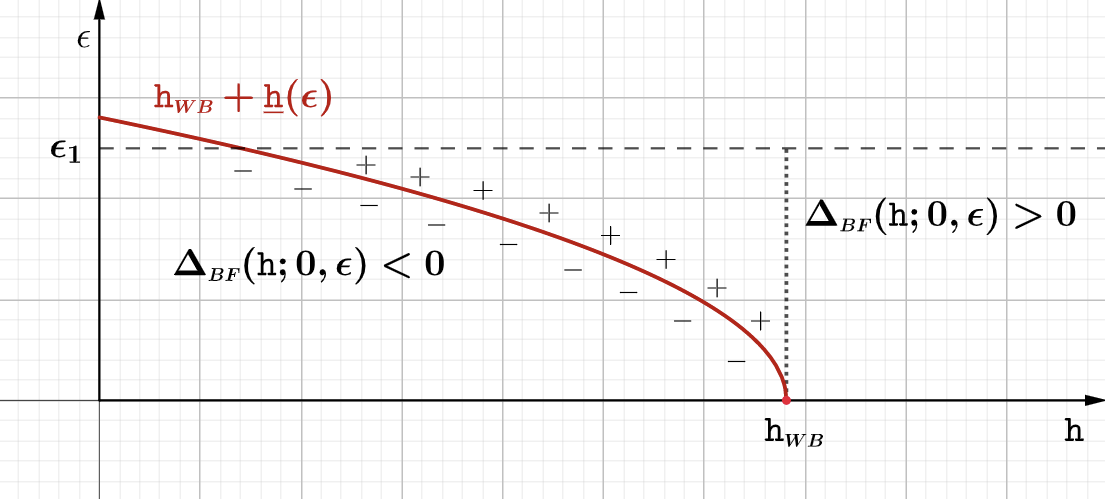}
\caption{
The values of ($\tth, \epsilon$) in $ (0,\infty) \times (0, \epsilon_1) $ 
for which there are Benjamin-Feir unstable eigenvalues fill the zone above the red curve, where $\DeltaBF(\tth;0,\e)>0$.
 }\label{fig:zone}
\end{figure}

\begin{figure}
 \centering
\includegraphics[scale=0.35]{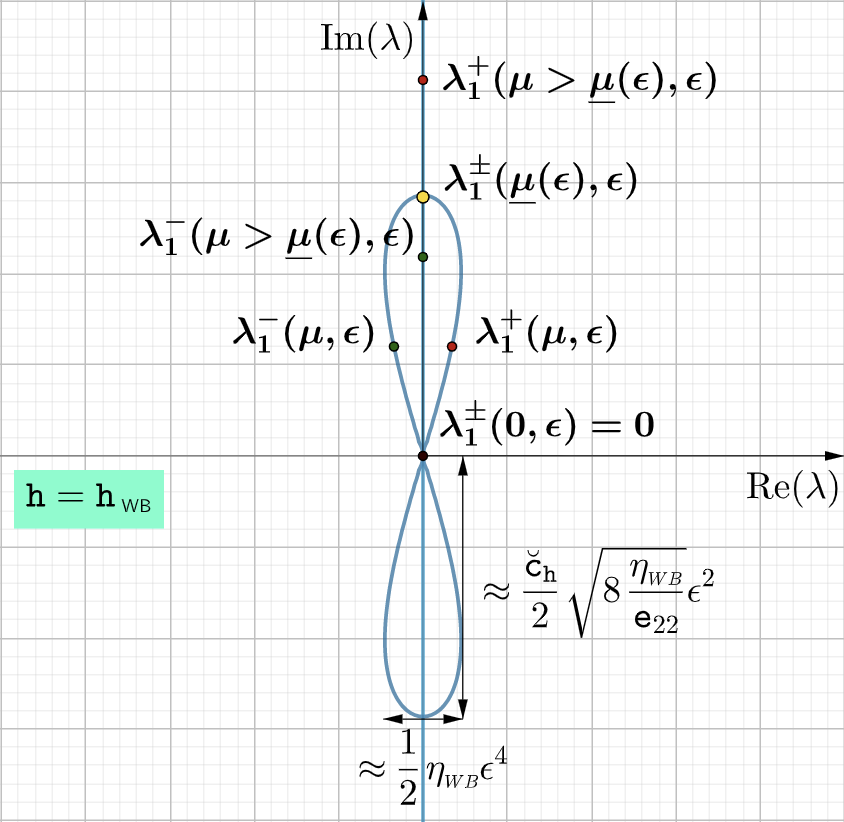}
\caption{
The figure "8" at the critical depth, which has a smaller size with respect to the one in Figure
\ref{figure-eigth}.
 }\label{fig:hWB}
\end{figure}

For a more rigorous statement we refer to Theorems \ref{abstractdec} and \ref{corollario}.
Actually Theorems \ref{abstractdec} and \ref{corollario} 
provide a necessary and sufficient condition for the existence of  unstable 
eigenvalues: the Benjamin-Feir discriminant 
function $\DeltaBF(\tth;\mu,\e)$ that appears in the  
matrix entry $ [\mathtt U]_{21} $   in \eqref{Udav}  has to be positive. 
We prove in  Theorem \ref{abstractdec} that 
the Benjamin-Feir discriminant 
function 
admits the expansion  
\begin{equation}\label{DWB}
\DeltaBF(\tth;\mu,\e) := 8\teWB (\tth) \e^2+ 8\etaWB (\tth) \e^4 + r_1(\e^5,\mu\e^3)-\te_{22} (\tth) \mu^2\big(1+r_1''(\e,\mu)\big) 
\end{equation}
where
\begin{equation}\label{funzioneWB}
\teWB  (\tth)  := \frac{1}{\ch}
\Big[ \frac{9\ch^8-10\ch^4+9}{8\ch^6} - 
\frac{1}{\tth- \frac14\te_{12}^2} \Big(1 + \frac{1-\ch^4}{2} + \frac34 \frac{(1-\ch^4)^2}{\ch^2}\tth \Big) \Big]  
\end{equation}
is called the Benjamin-Feir  function, 
 the coefficient $\te_{22} (\tth) > 0 $ is defined in \eqref{te22}, and the coefficient 
 $\etaWB (\tth) $ is computed in \eqref{etaWB.c}. 
The graph of the Benjamin-Feir   function $ \teWB (\tth )$ is  as in Figure \ref{graficoe112}.
\begin{figure}[h]
\centering
\includegraphics[width=4.5cm]{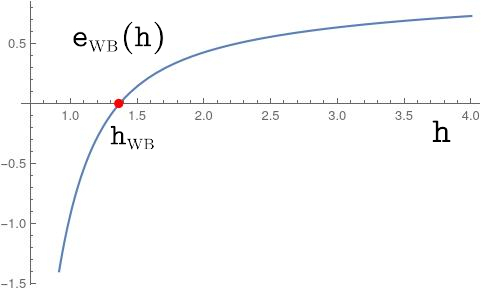}
\caption{The Whitham-Benjamin function  $ \teWB  (\tth) $
has a unique root $\tthWB=1.363\dots $.
 }
\label{graficoe112}
\end{figure}
Thus, for any $  \tth > \tthWB $, resp. $  \tth < \tthWB $, it results $ \teWB (\tth) > 0 $, 
resp. $  \teWB (\tth) < 0 $, 
and therefore for $ \mu $ and $ \epsilon $ 
small enough $ \DeltaBF(\tth;\mu,\e) > 0 $, resp. $ \DeltaBF(\tth;\mu,\e) < 0 $, proving the existence of
unstable, resp. stable, eigenvalues. 
This is the result proved in \cite{BMV3}.

On the other hand 
at $\tth=\tthWB $ the coefficient $ \teWB ( \tthWB ) = 0 $ vanishes and the sign of the Benjamin-Feir  function  
$ \DeltaBF(\tth;\mu,\e) $ in \eqref{DWB} is determined by the sign of the coefficient 
$ \etaWB (\tthWB) $
(note that in \eqref{DWB} no pure term of order $ \e^3 $  appears; 
such a degeneracy is a consequence  of symmetries of the problem). 
The constant $ \etaWB (\tthWB) $ is computed in \eqref{etaWB.c} and it turns out to be strictly {\it positive}. 
\begin{figure}
 \centering
\includegraphics[width=5cm]{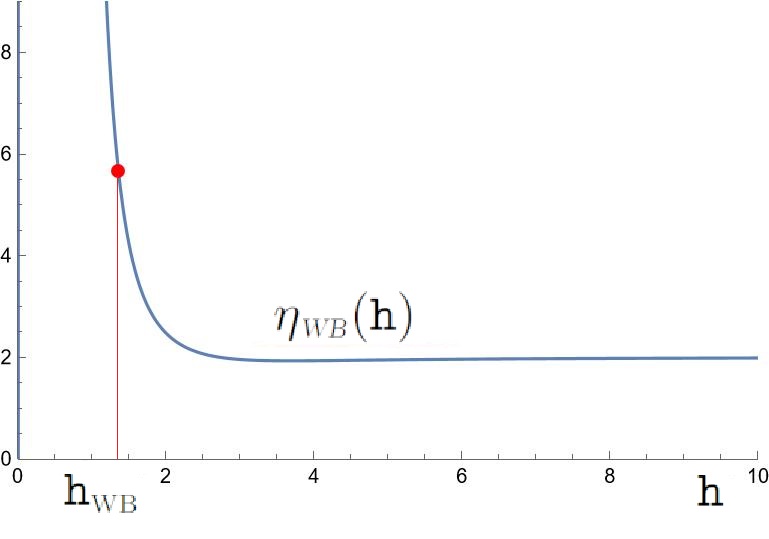}
\caption{
The  plot of the function $\etaWB (\tth) $  looks positive for every depth.
 }\label{fig:plotetaWB}
\end{figure}
This proves the linear instability of the Stokes wave at $ \tth = \tthWB $, stated in Theorem \ref{teoin}. 
In addition the regions 
where 
$\DeltaBF(\tth;0,\e)> 0 $, respectively  $\DeltaBF(\tth;0,\e) < 0 $,  are delimited by 
the graph of an analytic  curve of the form (obtained by the analytic implicit function theorem)
$$
\e \mapsto \tthWB+\underline{\tth}(\e)  \, , \quad 
 \underline{ \tth}(\e)  = - \frac{\etaWB(\tthWB)}{\teWB'(\tthWB)} \e^2   + \cO(\e^3) \, ,\qquad \text{solving}\quad\DeltaBF(\tthWB+\underline{\tth}(\e);0,\e)= 0\, ,
$$
see Figure \ref{fig:zone}. It turns out that 
 $ \DeltaBF(\tth;\mu,\e) > 0 $ 
if the condition \eqref{tunpo} holds. 

\smallskip

In order to prove \eqref{DWB} --which is the 
the major achievement of Sections \ref{sec:Bmue} and \ref{sec:block}--
we need to expand the Stokes waves up to order $ \e^4 $, as provided in
Appendix \ref{sec:App3}, and  to  explicitly compute the  
Taylor expansion of \cite{BMV3} at the fourth order of accuracy. 
We implement an effective algorithm to compute \eqref{DWB}
observing several analytical cancellations in the symplectic 
Kato reduction process. 
We now present rigorously the main results.

\paragraph{Main result.} 
We first shortly introduce  the pure  gravity water waves equations,  
their Hamiltonian formulation, and the linearized water waves equations at
 the Stokes waves. We 
refer to \cite{BMV3} for more details.   
\\[1mm]
{\bf The water waves equations.}
We consider the Euler equations for a 2-dimensional incompressible, irrotational fluid under the action of  gravity. The fluid fills the
region 
$$
{ \mathcal D}_\eta := \left\{ (x,y)\in \bT\times \bR\;:\; -\tth\leq  y< \eta(t,x)\right\} \, , 
\quad \bT :=\bR/2\pi\bZ \,, 
$$  
with finite depth
and  space periodic boundary conditions. 
The irrotational velocity field is the gradient  
of a harmonic scalar potential $\Phi=\Phi(t,x,y) $  
determined by its trace $ \psi(t,x)=\Phi(t,x,\eta(t,x)) $ at the free surface
$ y = \eta (t, x ) $.
Actually $\Phi$ is the unique solution of the elliptic equation
$    \Delta \Phi = 0 $ in $ {\mathcal D}_\eta $ with Dirichlet datum $
     \Phi(t,x,\eta(t,x)) = \psi(t,x)$ and $    \Phi_y(t,x,y)  =  0 $ at 
     $y =  - \tth $.

The time evolution of the fluid is determined by two boundary conditions at the free surface. 
The first is that the fluid particles  remain, along the evolution, on the free surface  
and the second one is that the pressure of the fluid  
is equal, at the free surface, to the constant atmospheric pressure. 
Then, as shown by Zakharov \cite{Zak1} and Craig-Sulem \cite{CS}, 
the time evolution of the fluid is determined by the 
following equations for the unknowns $ (\eta (t,x), \psi (t,x)) $,  
\begin{equation}\label{WWeq}
 \eta_t  = G(\eta)\psi \, , \quad 
  \psi_t  =  
- g \eta - \dfrac{\psi_x^2}{2} + \dfrac{1}{2(1+\eta_x^2)} \big( G(\eta) \psi + \eta_x \psi_x \big)^2 \, , 
\end{equation}
where $g > 0 $ is the gravity constant and $G(\eta):= G(\eta, \tth)$ denotes 
the Dirichlet-Neumann operator $
 [G(\eta)\psi](x) := \Phi_y(x,\eta(x)) -  \Phi_x(x,\eta(x)) \eta _x(x)$. 
 In the sequel, with no loss of generality, we set the gravity constant $ g = 1 $. 

The equations \eqref{WWeq} are the Hamiltonian system
\begin{equation}\label{PoissonTensor}
 \pa_t \vet{\eta}{\psi} = \cJ \vet{\nabla_\eta \mathcal{H}}{\nabla_\psi \mathcal{H}}, \quad \quad \cJ:= \begin{bmatrix} 0 & \uno \\ -\uno & 0 \end{bmatrix} ,
 \end{equation}
 where $ \nabla $ denote the $ L^2$-gradient, and the Hamiltonian
$  \mathcal{H}(\eta,\psi) :=  \frac12 \int_{\mathbb{T}} \left( \psi \,G(\eta)\psi +\eta^2 \right) \de x
$
is the sum of the kinetic and potential energy of the fluid. 
In addition of being Hamiltonian, the water waves 
system \eqref{WWeq} 
is reversible with respect to the involution 
\begin{equation}\label{revrho}
\rho\vet{\eta(x)}{\psi(x)} := \vet{\eta(-x)}{-\psi(-x)}, \quad \text{i.e. }
\mathcal{H} \circ \rho = \mathcal{H}  \, , 
\end{equation}
and it 
 is space invariant.
\\[1mm]{\bf Stokes waves.}
The Stokes waves are traveling
solutions of \eqref{WWeq}  of 
the form $\eta(t,x)=\breve \eta(x-ct)$ and $\psi(t,x)=\breve \psi(x-ct)$ for some real  $c$   and  $2\pi$-periodic functions  $(\breve \eta (x), \breve \psi (x)) $.
In a reference frame in translational motion with constant speed $c$,  the water waves equations \eqref{WWeq} become
\begin{equation}\label{travelingWW0}
\eta_t  = c\eta_x+G(\eta)\psi \, , \quad 
 \psi_t  = c\psi_x - \eta - \dfrac{\psi_x^2}{2} + \dfrac{1}{2(1+\eta_x^2)} \big( G(\eta) \psi + \eta_x \psi_x \big)^2  
\end{equation}
and the Stokes waves $(\breve \eta, \breve \psi)$ are  equilibrium 
steady solutions 
of \eqref{travelingWW0}. 

Small amplitude  Stokes waves were constructed 
by Struik \cite{Struik} in finite depth, and  Levi-Civita \cite{LC}, 
and Nekrasov \cite{Nek} in infinite depth.
 \begin{teo}\label{LeviCivita}
{\bf (Stokes waves)} For any $\tth \in (0,+\infty]$ there exist $\e_*=\e_*(\tth)  >0$ and a unique family  of real analytic solutions $(\eta_\e(x), \psi_\e(x), c_\e)$, parameterized by the amplitude $|\e|<\e_* $, of
\begin{equation}\label{travelingWW}
 c \, \eta_x+G(\eta)\psi = 0 \, , \quad
 c \, \psi_x -  \eta - \dfrac{\psi_x^2}{2} + \dfrac{1}{2(1+\eta_x^2)} \big( G(\eta) \psi + \eta_x \psi_x \big)^2  = 0 \, , 
\end{equation}
such that
 $\eta_\e (x) $ and $\psi_\e (x) $ are $2\pi$-periodic;  $\eta_\e (x) $ is even and $\psi_\e (x) $ is odd, of the form
   \begin{subequations}\label{exp:Sto}
   \begin{align}\label{etaexp}
 &\begin{aligned}
  & \eta_\e (x) =  \e \cos (x) + \e^2 \big(\eta_{2}^{[0]} + \eta_{2}^{[2]} \cos (2x)\big) +  \e^3 \big(\eta_{3}^{[1]} \cos(x) + \eta_{3}^{[3]} \cos (3x)\big)  \\
  & \qquad\qquad \qquad +\e^4 \big(\eta_{4}^{[0]} +\eta_{4}^{[2]} \cos(2x)+ \eta_{4}^{[4]} \cos (4x)\big) +
  \cO(\e^5)\, , 
  \end{aligned}\\
  &\begin{aligned}\label{psiexp}
  & \psi_\e (x)  = \e \ch^{-1} \sin (x) + \e^2 \psi_{2}^{[2]} \sin (2x)   + \e^3 \big(\psi_{3}^{[1]} \sin (x)+\psi_{3}^{[3]} \sin (3x)  \big) \\
  &\qquad\qquad \qquad + \e^4 \big(\psi_{4}^{[2]} \sin (2x)+\psi_{4}^{[4]} \sin (4x)\big)
  +\cO(\e^5) \, , 
 \end{aligned}   \\ \label{cexp}
  & c_\e = \ch + \e^2 c_2 +\e^4 c_4 + \cO(\e^5) \, , \qquad \ch = \sqrt{\tanh(\tth)} \, , 
   \end{align}
  \end{subequations}
with  coefficients given  in \eqref{allcoefStokes}. 
\end{teo}

The expansions \eqref{exp:Sto}
 are  derived in Proposition \ref{expstokes} 
 (they coincide with  \cite{Fenton} after some suitable rescaling, translation and choice of the moving frame, see Remark \ref{check-controllo}).
\\[1mm]
{\bf Remark.} More general time quasi-periodic traveling Stokes waves 
--which are nonlinear superpositions of multiple Stokes waves traveling with 
rationally independent speeds-- have been recently proved 
for \eqref{WWeq} in \cite{BFM2} in finite depth, in \cite{FG} in infinite depth,
and in \cite{BFM} for capillary-gravity water waves in any depth. 
\\[1mm]
\noindent
{\bf Linearization at the Stokes waves.}\label{goodunknown}
In order to determine the stability/instability of the Stokes  waves given by Theorem \ref{LeviCivita}, 
we linearize  the water waves equations \eqref{travelingWW0} with $ c = c_\e $ at  $(\eta_\epsilon(x), \psi_\epsilon(x))$. 
In \cite{BMV3}
we obtain the autonomous real linear Hamiltonian and reversible system
\begin{equation}\label{linearWW}
\!\!\! \vet{\hat \eta_t}{\hat \psi_t}
 = \begin{bmatrix} -G(\eta_\e)B-\pa_x \circ (V-c_\e) & G(\eta_\e) \\ -1+B(V-c_\e)\pa_x - B \pa_x \circ (V-c_\e) - BG(\eta_\e) B  & - (V-c_\e)\pa_x + BG(\eta_\e) \end{bmatrix}\vet{\hat \eta}{\hat \psi}
 \end{equation}
 where
 the functions $(V(x),B(x))$  are the horizontal and vertical components of the velocity field
$ (\Phi_x, \Phi_y) $ at the free surface. 
The real system \eqref{linearWW} is Hamiltonian and reversible, 
i.e. of the form $ \cJ \mathcal A $
with  $ \mathcal A = \mathcal A^\top $,
where $\mathcal A^\top$ is the transposed operator with respect the  scalar product of $L^2(\bT, \bR^2) $, and $ \cJ \mathcal A $
anti-commutes with the involution $ \rho $ in \eqref{revrho}.

The linear system \eqref{linearWW} assumes a simpler form by performing 
the time-independent  symplectic and reversibility preserving 
 ``good unknown of Alinhac''  
and  ``Levi-Civita" conformal  change of variables. 
As proved in \cite{BMV3,BBHM} there exists a diffeomorphism of $\mathbb{T}$,
 $ x\mapsto x+\mathfrak{p}(x)$, with a small $2\pi$-periodic odd function $\mathfrak{p}(x)$, 
 such that,  defining the associated composition operator $ (\mathfrak{P}u)(x) := u(x+\mathfrak{p}(x))$, 
 system \eqref{linearWW} is conjugated
under the change of variable
\begin{equation}\label{LC}
h =  \mathcal{P} Z^{-1} \vet{\hat \eta}{\hat \psi} \, , \quad
 \mathcal{P} := \begin{bmatrix}(1+\mathfrak{p}_x)\mathfrak{P} & 0 \\ 0 & \mathfrak{P} \end{bmatrix} \circ \begin{bmatrix}  1 & 0 \\ -B & 1\end{bmatrix}  \, , 
\end{equation}
 into
the linear system $ h_t = \cL_\e h $
 where  $ \cL_\e $ is the Hamiltonian and reversible real operator
\begin{equation}
\label{cLepsilon}
\cL_\e   
 :=  
\begin{bmatrix} \pa_x \circ (\ch+p_\e(x)) &  |D|\tanh((\tth+\mathtt{f}_\e) |D|) \\ - (1+a_\e(x)) &   (\ch+p_\e(x))\pa_x \end{bmatrix} 
 = \cJ \begin{bmatrix}   1+a_\e(x) &   -(\ch+p_\e(x)) \pa_x \\ 
\pa_x \circ (\ch+p_\e(x)) &  |D|\tanh((\tth+\mathtt{f}_\e) |D|)  \end{bmatrix} 
\end{equation}
where  $p_\e (x) $, $a_\e (x) $ are even real functions
and $\ttf_\e $ is small real constant.
 The functions 
 $p_\e$ and $a_\e$   are analytic in $\e$ as maps $B(\e_0)\to H^{s} (\mathbb T)$
 and admit  a Taylor expansion as in Proposition \ref{propaepe}.
 The function  $ \e \mapsto \ttf_\e $ is  analytic as well 
with a Taylor expansion as in \eqref{expfe0}-\eqref{expfe}.
\\[1mm] 
\noindent
{\bf Bloch-Floquet operator.} 
Since the operator $\cL_\e$ in \eqref{cLepsilon} has $2\pi$-periodic coefficients, Bloch-Floquet theory guarantees  that the spectrum 
$$
\sigma_{L^2(\bR)} (\cL_\e ) = \bigcup_{\mu\in [- \frac12, \frac12)} \sigma_{L^2(\bT)}  (\cL_{\mu, \e}) \qquad \text{where} \quad 
\qquad \cL_{\mu,\e}:= e^{- \im \mu x} \, \cL_\e \, e^{\im \mu x}  \, ,
$$
and, if $\lambda$ is an eigenvalue of $\cL_{\mu,\e}$ on $L^2(\bT, \bC^2)$
with eigenvector $v(x)$, then  $h (t,x) = e^{\lambda t} e^{\im \mu x} v(x)$ 
is a solution of $ h_t = \cL_{\e} h$.

The Floquet operator  associated with the real operator $\cL_\e$ in \eqref{cLepsilon} 
turns out to be  the complex  \emph{Hamiltonian} and \emph{reversible} pseudo-differential 
operator
\begin{align}\label{WW}
 \cL_{\mu,\e} :&= \begin{bmatrix} (\pa_x+\im\mu)\circ (\ch+p_\e(x)) & |D+\mu| \tanh\big((\tth + \ttf_\e) |D+\mu| \big) \\ -(1+a_\e(x)) & (\ch+p_\e(x))(\pa_x+\im \mu) \end{bmatrix} \\ 
 &= \underbrace{\begin{bmatrix} 0 & \uno\\ -\uno & 0 \end{bmatrix}}_{\displaystyle{=\cJ}} \underbrace{\begin{bmatrix} 1+a_\e(x) & -(\ch+p_\e(x))(\pa_x+\im \mu) \\ (\pa_x+\im\mu)\circ (\ch+p_\e(x)) & |D+\mu| \tanh\big((\tth + \ttf_\e) |D+\mu| \big) \end{bmatrix}}_{\displaystyle{=:\mathcal{B}_{\mu,\e}}} \, , \notag 
\end{align}
meaning that 
$ \mathcal{B}_{\mu,\e} = \mathcal{B}_{\mu,\e}^*  $ and 
$  \cL_{\mu,\e} \circ \bro =- \bro \circ \cL_{\mu,\e} $, 
where $ \bro $ is the complex involution (cfr. \eqref{revrho})
$$
 \bro \vet{\eta(x)}{\psi(x)} := \vet{\bar\eta(-x)}{-\bar\psi(-x)} \, .
$$
Equivalently 
the self-adjoint operator $\mathcal{B}_{\mu,\e}$ is \emph{reversibility-preserving},  i.e. 
$ \mathcal{B}_{\mu,\e} \circ \bro = \bro \circ \mathcal{B}_{\mu,\e} $.

We regard $  \cL_{\mu,\e} $ as an operator with 
domain $H^1(\bT):= H^1(\mathbb{T},\bC^2)$ 
and range $L^2(\bT):=L^2(\mathbb{T},\bC^2)$, equipped with  
the complex scalar product 
\begin{equation}\label{scalar}
(f,g) := \frac{1}{2\pi} \int_{0}^{2\pi} \left( f_1 \bar{g_1} + f_2 \bar{g_2} \right) \, \text{d} x  \, , 
\quad
\forall f= \vet{f_1}{f_2}, \ \  g= \vet{g_1}{g_2} \in  L^2(\bT, \bC^2) \, .
\end{equation} We also  denote $ \| f \|^2 = (f,f) $.

In addition $(\mu, \e) \to \cL_{\mu,\e} \in \cL(H^1(\bT), L^2(\bT))$ is analytic, 
since  the functions $\e \mapsto a_\e$, $p_\e$  defined in \eqref{SN1} are analytic as maps $B(\e_0) \to H^1(\bT)$ 
and  ${\mathcal L}_{\mu,\e}$ is analytic with respect to $\mu$. 
\\[1mm]
{\bf Remark.} The spectrum  
$ \sigma ({\mathcal L}_{-\mu,\e}) = \overline{  \sigma ({\mathcal L}_{\mu,\e}) } $ and 
we can restrict to  $ \mu > 0 $.
Furthermore $\sigma({\mathcal L}_{\mu,\e} )$ is a 1-periodic set with respect to $\mu$, so   one can restrict  to  $\mu \in [0, \frac12)$.  
\\[1mm]
\noindent 
{\bf Krein criterion}. 
In view of the Hamiltonian structure of $\cL_{\mu,\e}$, 
eigenvalues with non zero real part may arise only from multiple
  eigenvalues of $\cL_{\mu,0}$  
  because   if $\lambda$ is an eigenvalue of $\cL_{\mu,\e}$ then  also $-\bar \lambda$ is, and the total algebraic multiplicity of the eigenvalues is conserved under small perturbation. 
The 
Fourier multiplier matrix real operator 
$$ 
\cL_{0,0}  = \begin{bmatrix} \ch  \pa_x & |D| \,  \tanh\big(\tth  |D| \big)
  \\ -1 & \ch \pa_x \end{bmatrix}
$$
possesses the  eigenvalue $0$ with algebraic multiplicity $4$, 
 and geometric multiplicity $3$. 
 A real  basis of the corresponding generalized eigenspace  is 
\begin{equation}\label{funperturbed} 
 f_1^+ :=  \vet{ \ch^{1/2} \cos(x)}{\ch^{-1/2} \sin(x)}, 
 \quad 
 f_1^{-} :=  \vet{- \ch^{1/2} \sin(x)}{\ch^{-1/2}\cos(x)},  
 \quad f_0^+:=\vet{1}{0}, \quad f_0^-:=\vet{0}{1},
\end{equation}
where $ f_1^+,  f_1^{-}, f_0^- $ are eigenvectors of $\cL_{0,0}$ and 
 $f_0^+ $ is a generalized eigenvector,  namely $ \cL_{0,0}f_0^+ =-f_0^-  $. 
Furthermore $0$ is an isolated eigenvalue for $\cL_{0,0}$, namely the 
 spectrum   $\sigma\left(\cL_{0,0}\right)  $ decomposes in two separated parts
\begin{equation}
\label{spettrodiviso0}
\sigma\left(\cL_{0,0}\right) = \sigma'\left(\cL_{0,0}\right) \cup \sigma''\left(\cL_{0,0}\right)
\quad \text{where} \quad  \sigma'(\cL_{0,0}):=\{0\} 
\end{equation}
and 
$ \sigma''(\cL_{0,0}) $ is formed by non zero eigenvalues at a positive distance from $ 0 $. 
 In this paper we study the spectrum of $ {\cal L}_{\mu, \e} $ near $ 0 $ (an interesting 
 problem concerns the splitting of the other non zero eigenvalues, see \cite{CDT}).  
 By Kato's perturbation theory (see  Lemma \ref{lem:Kato1} below)
for any $\mu, \e \neq 0$  sufficiently small, the perturbed spectrum
$\sigma\left(\cL_{\mu,\e}\right) $ admits a disjoint decomposition as 
$
\sigma\left(\cL_{\mu,\e}\right) = \sigma'\left(\cL_{\mu,\e}\right) \cup \sigma''\left(\cL_{\mu,\e}\right) $
where $ \sigma'\left(\cL_{\mu,\e}\right)$  consists of 4 eigenvalues close to 0. 
We denote by $\cV_{\mu, \e}$   the spectral subspace associated with  $\sigma'\left(\cL_{\mu,\e}\right) $, which   has  dimension 4 and it is  invariant by $\cL_{\mu, \e} 
: \mathcal{V}_{\mu,\e} \to  \mathcal{V}_{\mu,\e} $.
The next Theorem \ref{abstractdec} provides 
the complete splitting
of the eigenvalues of the $4\times 4$ matrix 
which represents
 the operator  $ \cL_{\mu,\e} : \mathcal{V}_{\mu,\e} \to  \mathcal{V}_{\mu,\e} $.
Before stating it, we first introduce a notation used through all the paper:
\\[1mm] {\bf Notation: }
we denote by  $\cO(\mu^{m_1}\e^{n_1},\dots,\mu^{m_p}\e^{n_p})$, $ m_j, n_j \in \bN  $ (for us $\bN:=\{1,2,\dots\} $),  
analytic functions of $(\mu,\e)$ with values in a Banach space $X$ which satisfy, for some $ C > 0 $ uniform for  $\tth$ in any compact set of $(0, + \infty)$, the bound
 $\|\cO(\mu^{m_j}\e^{n_j})\|_X \leq C \sum_{j = 1}^p |\mu|^{m_j}|\e|^{n_j}$ 
 for small values of $(\mu, \e)$. 
Similarly we denote $r_k (\mu^{m_1}\e^{n_1},\dots,\mu^{m_p}\e^{n_p}) $
scalar  functions  $\cO(\mu^{m_1}\e^{n_1},\dots,\mu^{m_p}\e^{n_p})$ which are  also {\it real} analytic.

\smallskip

Our complete  spectral result is the following: 

\begin{teo}\label{abstractdec} {\bf (Complete Benjamin-Feir spectrum)}
There exist $ \e_0, \mu_0>0 $, uniformly for the 
depth $ \tth  $ in any compact set of $ (0,+\infty )$, 
 such that, 
for any  $ 0\, <\, \mu < \mu_0 $ and $ 0\leq \e < \e_0 $, 
 the operator $ \cL_{\mu,\e} : \mathcal{V}_{\mu,\e} \to  \mathcal{V}_{\mu,\e} $ 
 can be represented by a $4\times 4$ matrix of the form 
 \begin{equation} \label{matricefinae}
  \begin{pmatrix} \mathtt{U}  & 0 \\  0 &  \mathtt{S} \end{pmatrix},
 \end{equation}
 where $ \mathtt{U} $ and $ \mathtt{S} $ are  $ 2 \times 2 $
matrices, with identical purely imaginary diagonal entries each, of the form 
 \begin{align}
\label{Udav}
& \mathtt{U} = {\begin{pmatrix} 
\im  \big((\ch- \tfrac12\te_{12})\mu+ r_2(\mu\e^2,\mu^2\e,\mu^3) \big) & -\te_{22}\frac{\mu}{8}(1+r_5(\e,\mu)) \\
-\frac{\mu}{8} \DeltaBF(\tth; \mu, \e)  & \im  \big( (\ch-\tfrac12\te_{12})\mu+ r_2(\mu\e^2,\mu^2\e,\mu^3) \big)  
 \end{pmatrix}}\, ,   \\
 & \label{S} \mathtt{S} = 
 \begin{pmatrix} \im\ch\mu+ \im { r_9(\mu\e^2, \mu^2\e)}  & \tanh(\tth \mu)+ {r_{10}(\mu\e)} \\ 
-\mu + {r_8(\mu\e^2, \mu^3 \e)} &    \im\ch\mu+\im {r_9(\mu\e^2,\mu^2\e) }
 \end{pmatrix}\, .
\end{align}
The Benjamin-Feir discriminant function $\DeltaBF(\tth;\mu,\e)$ in \eqref{Udav} 
has the form \eqref{DWB}, 
where $\teWB (\tth) $ is the Whitham-Benjamin function 
in \eqref{funzioneWB},  the coefficient 
$\te_{22} (\tth) > 0 $ is  in \eqref{te22}, and  
\begin{equation}
\label{etaWB.c}
\begin{aligned}
\etaWB(\tth)& := {\scriptsize \begin{matrix}
\frac1{256} \ch^{-19} (\ch^2+1)^{-1} \big(\ch^4-2  (\ch^4+1) \tth \ch^2+ (\ch^4-1)^2 \tth^2\big)^{-4} \end{matrix} } \cdot \\
&\cdot {\scriptsize \begin{matrix} \Big[\big(476 \ch^{26}+532 \ch^{24}-3973 \ch^{22}-4361 \ch^{20}+17173 \ch^{18}+17557 \ch^{16}-37778 \ch^{14}-37754 \ch^{12}\end{matrix} } \\
&  {\scriptsize \begin{matrix}-8898 \ch^{10}-8442 \ch^8+855 \ch^6+963 \ch^4+81 \ch^2+81\big) \ch^{16}-8 \big(432 \ch^{30}+480 \ch^{28}-3057 \ch^{26}\end{matrix} }\\
&  {\scriptsize \begin{matrix}-3361 \ch^{24}+11452 \ch^{22}+11544 \ch^{20}-25989 \ch^{18}-25749 \ch^{16}+3928 \ch^{14}+4384 \ch^{12}-555 \ch^{10}\end{matrix} }\\
&  {\scriptsize \begin{matrix}-171 \ch^8+396 \ch^6+504 \ch^4+81 \ch^2+81\big) \tth \ch^{14}
+ 4 \big(2612 \ch^{34}+2876 \ch^{32}-15531 \ch^{30}-17239 \ch^{28}\end{matrix} }\\
&  {\scriptsize \begin{matrix}+44053 \ch^{26}+44277 \ch^{24}-82191 \ch^{22}  -81283 \ch^{20}+5921 \ch^{18}\end{matrix} }\\
&  {\scriptsize \begin{matrix}+9353 \ch^{16}+6831 \ch^{14}+10203 \ch^{12}+23007 \ch^{10}+24975 \ch^8-117 \ch^6+639 \ch^4+567 \ch^2+567\big) \tth^2 \ch^{12}\end{matrix} }\\
&  {\scriptsize \begin{matrix}-8 \big(2128 \ch^{38}+2304 \ch^{36}-11055 \ch^{34}-12463 \ch^{32}+19370 \ch^{30}+20126 \ch^{28}\end{matrix} } \\
&  {\scriptsize \begin{matrix}-5794 \ch^{26}-5594 \ch^{24}-51646 \ch^{22}-49154 \ch^{20}+57448 \ch^{18}
 +59416 \ch^{16}-28802 \ch^{14}-26582 \ch^{12}+32754 \ch^{10}\end{matrix} }\\
 & {\scriptsize \begin{matrix} +33786 \ch^8-2682 \ch^6-1926 \ch^4+567 \ch^2+567\big) \tth^3 \ch^{10}\end{matrix} } \\
&  {\scriptsize \begin{matrix}+2 \big(8020 \ch^{42}+8380 \ch^{40}-41279 \ch^{38}-46795 \ch^{36}+49331 \ch^{34}+57267 \ch^{32}+86052 \ch^{30}+84516 \ch^{28}\end{matrix} } \\
&  {\scriptsize \begin{matrix}-274180 \ch^{26}-267924 \ch^{24}+176654 \ch^{22}+178806 \ch^{20}+104434 \ch^{18}+101746 \ch^{16}-211660 \ch^{14}-205420 \ch^{12}\end{matrix} }\\
&  {\scriptsize \begin{matrix}+181752 \ch^{10}+181152 \ch^8-24615 \ch^6-20835 \ch^4+2835 \ch^2+2835\big) \tth^4 \ch^8\end{matrix} } \\
&  {\scriptsize \begin{matrix}-8 \big(\ch^4-1\big)^2 \big(1072 \ch^{38}+1024 \ch^{36}-4711 \ch^{34}-5399 \ch^{32}+4546 \ch^{30}+5302 \ch^{28}-4162 \ch^{26}-3850 \ch^{24}\end{matrix} }\\
&  {\scriptsize \begin{matrix}+13442 \ch^{22}+15070 \ch^{20}-11088 \ch^{18}-10992 \ch^{16}-9066 \ch^{14}-7806 \ch^{12}+29442 \ch^{10}+29466 \ch^8\end{matrix} } \\
&  {\scriptsize \begin{matrix}-5706 \ch^6-4950 \ch^4+567 \ch^2+567\big)\ch^6 \tth^5\end{matrix} } \\
&  {\scriptsize \begin{matrix}+  4\left(\ch^4-1\right)^4 \big(564 \ch^{34}+380 \ch^{32}-3755 \ch^{30}-4087 \ch^{28}+8917 \ch^{26}+9557 \ch^{24}-15215 \ch^{22}\end{matrix} } \\
&  {\scriptsize \begin{matrix}-14499 \ch^{20}+9953 \ch^{18}+10313 \ch^{16}-8273 \ch^{14}-7109 \ch^{12}+24159 \ch^{10}+24111 \ch^8-6165 \ch^6 \end{matrix} }\\
&  {\scriptsize \begin{matrix}-5409 \ch^4+567 \ch^2+567\big) \tth^6 \ch^4\end{matrix} }  {\scriptsize \begin{matrix}-8 \left(\ch^4-1\right)^6 \big(16 \ch^{30}-32 \ch^{28}-521 \ch^{26}-489 \ch^{24}+1252 \ch^{22}+1344 \ch^{20}\end{matrix} } \\
& {\scriptsize \begin{matrix} -1853 \ch^{18}-1901 \ch^{16}-512 \ch^{14}-344 \ch^{12}+3333 \ch^{10}+3285 \ch^8-900 \ch^6-792 \ch^4+81 \ch^2+81\big) \tth^7 \ch^2\end{matrix} }\\
&  {\scriptsize \begin{matrix}-\left(\ch^4-1\right)^8 
\big(
       36 \ch^{26}+ 108 \ch^{24} + 261 \ch^{22}  +73 \ch^{20}-1429 \ch^{18}-1237 \ch^{16} +3666\ch^{14}\end{matrix} }\\
&  {\scriptsize \begin{matrix}       +3450 \ch^{12}-3774 \ch^{10}-3654 \ch^{8}+873 \ch^6
       +765 \ch^4-81 \ch^2-81\big) \tth^8 \Big]\end{matrix} } \, . 
\end{aligned}
\end{equation}
\end{teo}

A 
numerical calculus performed  by Mathematica reveals that 
\begin{equation}\label{etaWBfinale}
\etaWB (\tthWB) \approx 5.65555 > 0  
\end{equation}
and we deduce Theorem \ref{teoin} since   
eigenvalues with nonzero real part appear whenever the  Benjamin-Feir discriminant 
$ \DeltaBF(\tth;\mu,\e)  > 0 $.  In the 
following corollary of Theorem \ref{abstractdec}
we describe the unstable eigenvalues of $ {\cal L}_{\mu,\epsilon} $ 
at the critical  depth $ \tth = \tthWB $ 
 (for simplicity we avoid to state the result for {\it any}  
$ (\tth, \e) $ satisfying \eqref{tunpo}).

\begin{teo}\label{corollario}
 {\bf (Benjamin-Feir unstable eigenvalues at} $ \tth = \tthWB .$$	 \bf )$
There exist  $ \e_1, \mu_0 > 0 $  and 
an analytic function $\underline \mu(\cdot): [0,\e_1)\to [0,\mu_0)$ 
of the form 
 \begin{equation}\label{barmuepbis}
 \underline  \mu(\e)  = 
   \underline{c} \, \e^2 \big(1+r(\e)\big)   \, , 
   \quad \underline{c} := 	
   \sqrt{ \frac{8 \etaWB (\tthWB)}{\te_{22}(\tthWB) }} \, , 
 \end{equation}
such that, 
 for any  $  \e \in [0, \e_1)  $, the 
 operator  $\cL_{\mu,\e}$ has two eigenvalues
  $\lambda^\pm_1 (\mu,\e)$
 \begin{equation} \label{eigelemubis}
 \begin{cases} 
     \im \frac12 \breve{\mathtt c}_\tth \mu+\im r_2(\mu\e^2,\mu^2\e,\mu^3)  \pm \tfrac18 \mu \sqrt{\te_{22}(\tthWB)}  (1+r(\e,\mu))  \sqrt{\DeltaBF(\tthWB;\mu,\e) }
& 
0<\mu<\underline\mu(\e)\!\!\! \\[1.5mm]
    \im \frac12 \breve{\mathtt c}_\tth \mu+\im r_2(\mu\e^2,\mu^2\e,\mu^3)  \pm \im \tfrac18 \mu \sqrt{\te_{22}(\tthWB)}  (1+r(\e,\mu))  \sqrt{|\DeltaBF(\tthWB;\mu,\e)|} & 
  \underline\mu(\e)\leq \mu< \mu_0\, ,\!\!\!
\end{cases}\!\!\!
\end{equation}
with 
$\breve{\mathtt c}_\tth:=2 \ch- \te_{12}(\tth) >0$ 
and $ \DeltaBF(\tthWB;\mu,\e) = 
8\etaWB (\tthWB) \e^4 + r_1(\e^5,\mu\e^3)-\te_{22} (\tthWB) \mu^2\big(1+r_1''(\e,\mu)\big)  $. 
\end{teo}

\noindent{\it Proof of Theorem \ref{corollario}}. 
Since 
$ \DeltaBF(\tthWB;0,\e) = 
8\etaWB (\tthWB) \e^4 (1+r(\e)) $, it results that
$ \DeltaBF(\tthWB;\mu,\e) > 0 $, for any $ \mu \in (0, \underline\mu(\e)) $ as in \eqref{barmuepbis} and $ \e $ small enough. 
The unstable eigenvalues $\lambda^\pm_1 (\mu,\e)$ in \eqref{barmuepbis}
are  those of the matrix $\mathtt{U}$ in \eqref{Udav}.
In order to determine 
the value $ \mu =  \underline\mu(\e) $ such that 
$\lambda^\pm_1 (\mu,\e)$ touches the imaginary axis far from the origin, we set
$ \mu = c \e^2 $ so that 
 $ \DeltaBF(\tthWB;\mu,\e)  = 0 $ if and only if 
$$ 
0 = \e^{-4} \DeltaBF(\tthWB;c \e^2,\e)  = 
8\etaWB (\tthWB)  (1+r(\e))+ 
r_1(c \e )-\te_{22} (\tthWB) c^2 \big(1+r_1(\e)\big)   \, .
$$
This equation is 
solved by an analytic  function $ \e \mapsto c_\e = \underline{c} (1+r(\e)) $
with $ \underline{c} $ defined in \eqref{barmuepbis}. 
 \qed
\smallskip

\medskip

We conclude this section describing the main steps of the proof and 
the organization of the paper. 
\\[1mm]
{\bf Ideas and scheme of proof.} In Section \ref{Katoapp} we shortly report the results of 
\cite{BMV3} which reduce the problem to determine the eigenvalues of 
the $ 4 \times 4 $ Hamiltonian and reversible matrix 
$ \tL_{\mu,\e} =\tJ\tB_{\mu,\e} $ in \eqref{Lmuepsi}.
Then in Section \ref{sec:Bmue}  we provide the Taylor 
expansion of  the matrix $\tB_{\mu,\e} $  in \eqref{tocomputematrix0}
at an order of accuracy higher  than 
 in \cite[Proposition 4.3]{BMV3}.  
 In particular in Proposition 
\ref{BexpG} we compute the coefficients of the Taylor expansion up to order $ 4 $
in 
 the matrix entries \eqref{BinG1}-\eqref{BinG3} 
 which enter in the constant $\etaWB (\tth) $ (cfr. \eqref{etaWB}) 
appearing in the Benjamin-Feir discriminant function \eqref{DWB}. 
This explicit computation requires the knowledge of the 
Taylor expansions of the Kato spectral projectors $ P_{\mu, \epsilon} $ up to cubic order,
that we provide in Section \ref{secPmue} and prove in Appendix \ref{conticinoleggeroleggero}, relying on complex analysis. In order to perform effective computations 
we observe several analytical cancellations in Sections \ref{secPgotmue} and
\ref{secProof31}, which reduce considerably the number of explicit scalar products 
to compute. 
 The proof of Proposition 
\ref{BexpG} 
requires ultimately the 
knowledge of the Taylor expansion up to order four 
of the Levi-Civita and Alinach good unknown transformations \eqref{LC} 
 and of the functions $ a_\epsilon (x), p_\epsilon (x) $ in the operator $ {\cal L}_{\mu,\e} $
 in \eqref{WW}, which are derived in Appendix \ref{sec:exape}. 
 In turn, such expansions follow by those 
  of the Stokes waves 
that we prove 
in Appendix \ref{sec:App31}. Finally in Section \ref{sec:block} we 
implement the block-diagonalization procedure of \cite[Section 5]{BMV3} which provides the 
block-diagonal matrix \eqref{Udav}
and we analytically compute the   expansion  of the 
Benjamin-Feir discriminant function $\DeltaBF(\tth;\mu,\e)$, in particular of the 
constant $ \etaWB (\tth) $ in \eqref{etaWB} and thus \eqref{etaWB.c}.

We point out that  the constant 
$ \etaWB $ in  \eqref{etaWB} is {\it analytically} computed in terms of the
coefficients \eqref{expXentries} which, in turn,
 are expressed 
in  terms of the  coefficients
$ \phi_{21}, \phi_{22}, \gamma_{12}, \eta_{12}, \gamma_{11}, \phi_{11}, \gamma_{22}, \phi_{12}, \tf_{11} $, and ultimately 
$\eta_{2}^{[0]}, \ldots, \eta_{4}^{[4]}, \psi_{2}^{[2]}, \dots,  \psi_{4}^{[4]}  $, 
$ c_2, c_4 $ 
of the Stokes wave provided in Appendix \ref{sec:App31}. Then we 
used Mathematica to compute how the  coefficients of the Stokes wave, of the functions $a_\e(x)$, $p_\e(x)$ in \eqref{WW},  and  $\etaWB(\tth)$ in \eqref{etaWB.c} depend on $\tth$, starting from their algebraic formulas. 
The Mathematica code employed can be found at 
\url{https://git-scm.sissa.it/amaspero/benjamin-feir-instability}.


\bigskip

\noindent{\bf Acknowledgments.}
Research supported by PRIN 2020 (2020XB3EFL001) 
``Hamiltonian and dispersive PDEs''.
\\[1mm]
{\bf Data availability statement:} the Mathematica code used to perform some computations is available  at https://git-scm.sissa.it/amaspero/benjamin-feir-instability.

\section{Perturbative approach to  separated eigenvalues}\label{Katoapp}

In this section we shortly report  the spectral  procedure developed in \cite{BMV3,BMV1}
to study the splitting of the eigenvalues of 
$ \cL_{\mu,\e} $ close to $ 0 $ for small values of $ \mu $ and $ \e $. 
First of all we decompose the operator $ \cL_{\mu,\e}$ in \eqref{WW} as 
  \begin{equation}\label{calL}
 \cL_{\mu,\e}  = \im  \ch \mu + \sL_{\mu,\e} \, , \qquad \mu > 0 \, ,  
\end{equation}
where 
$\sL_{\mu,\e}$ is the Hamiltonian and reversible operator
\begin{equation}\label{calL.ham}
\sL_{\mu,\e} = 
 \cJ \, {\cal B}_{\mu, \e} \, , \quad
{\cal B}_{\mu, \e}
:= \begin{bmatrix} 1+a_\e(x) & -(\ch+p_\e(x))\pa_x-\im \mu\,  p_\e(x)  \\
 \pa_x\circ (\ch+p_\e(x)) + \im \mu \,  p_\e(x) & |D+ \mu| \, \tanh\big((\tth+\ttf_\e)|D+\mu|\big) \end{bmatrix}
\end{equation}
with ${\cal B}_{\mu, \e}$ selfadjoint. 
The operator
 $\sL_{\mu,\e}$ is analytic with respect to $ (\mu, \e) $ as $ \cL_{\mu,\e} $ is. 
The operator $ \sL_{\mu,\e}  : Y \subset X \to X $   
has domain $Y:=H^1(\mathbb{T}):=H^1(\mathbb{T},\bC^2)$ and range $X:=L^2(\mathbb{T}):=L^2(\mathbb{T},\bC^2)$.

In view of \eqref{calL}, the spectrum  
$$ 
\sigma ({\mathcal L}_{\mu,\e}) = \im \ch \mu + \sigma (\sL_{\mu,\e}) 
$$ 
and  we focus on studying the spectrum of  $ \sL_{\mu,\e}$.
The unperturbed  $\sL_{0,0} = \cL_{0,0} $ has zero as isolated eigenvalue with algebraic multiplicity 4, geometric multiplicity 3 and generalized kernel spanned by the vectors  $\{f^+_1, f^-_1, f^+_0, f^-_0\}$ in \eqref{funperturbed}.
The following lemma is  \cite[Lemmata 3.1  and 3.2]{BMV3}. 

\begin{lem}\label{lem:Kato1}
{\bf (Kato theory for separated eigenvalues of Hamiltonian operators)} 
 Let $\Gamma$ be a closed, counterclockwise-oriented curve around $0$ in the complex plane separating $\sigma'\left(\sL_{0,0}\right)=\{0\}$
  and the other part of the spectrum $\sigma''\left(\sL_{0,0}\right)$ in \eqref{spettrodiviso0}.
There exist $\e_0, \mu_0>0$  such that for any $(\mu, \e) \in B(\mu_0)\times B(\e_0)$  the following statements hold:
\\[1mm] 
1. The curve $\Gamma$ belongs to the resolvent set of 
the operator $\sL_{\mu,\e} : Y \subset X \to X $ defined in \eqref{calL.ham}.
\\[1mm] 
2.
The operators
\begin{equation}\label{Pproj}
 P_{\mu,\e} := -\frac{1}{2\pi\im}\oint_\Gamma (\sL_{\mu,\e}-\lambda)^{-1} \de\lambda : X \to Y 
\end{equation}  
are well defined projectors commuting  with $\sL_{\mu,\e}$,  i.e. 
$ P_{\mu,\e}^2 = P_{\mu,\e} $ and 
$ P_{\mu,\e}\sL_{\mu,\e} = \sL_{\mu,\e} P_{\mu,\e} $. 
The map $(\mu, \epsilon)\mapsto P_{\mu,\epsilon}$ is  analytic from 
$B({\mu_0})\times B({\epsilon_0})$
 to $ \cL(X, Y)$. 
 The projectors $P_{\mu,\e}$ are 
 skew-Hamiltonian and reversibility preserving, i.e. 
 \begin{equation}\label{propPU}
  \cJ P_{\mu,\e}=P_{\mu,\e}^*\cJ \, , \quad 
\bro P_{\mu,\e} = P_{\mu,\e}  \bro \, .
\end{equation}
 Finally 
$P_{0,\e}$ is a  real operator, i.e. $\bar{P_{0,\e}}=P_{0,\e} $. 
\\[1mm] 
3.
The domain $Y$  of the operator $\sL_{\mu,\e}$ decomposes as  the direct sum
$$
Y= \mathcal{V}_{\mu,\e} \oplus \text{Ker}(P_{\mu,\e}) \, , \quad \mathcal{V}_{\mu,\e}:=\text{Rg}(P_{\mu,\e})=\text{Ker}(\uno-P_{\mu,\e}) \, ,
$$
of   closed invariant  subspaces, namely 
$ \sL_{\mu,\e} : \mathcal{V}_{\mu,\e} \to \mathcal{V}_{\mu,\e} $, $
\sL_{\mu,\e} : \text{Ker}(P_{\mu,\e}) \to \text{Ker}(P_{\mu,\e}) $.  
Moreover 
$$
\begin{aligned}
&\sigma(\sL_{\mu,\e})\cap \{ z \in \bC \mbox{ inside } \Gamma \} = \sigma(\sL_{\mu,\e}\vert_{{\mathcal V}_{\mu,\e}} )  = \sigma'(\sL_{\mu, \e}) , \\
&\sigma(\sL_{\mu,\e})\cap \{ z \in \bC \mbox{ outside } \Gamma \} = \sigma(\sL_{\mu,\e}\vert_{Ker(P_{\mu,\e})} )  = \sigma''( \sL_{\mu, \e}) \, .
\end{aligned}
$$
4.  The projectors $P_{\mu,\e}$ 
are similar one to each other: the  transformation operators
\begin{equation} \label{OperatorU} 
U_{\mu,\e}   := 
\big( \uno-(P_{\mu,\e}-P_{0,0})^2 \big)^{-1/2} \big[ 
P_{\mu,\e}P_{0,0} + (\uno - P_{\mu,\e})(\uno-P_{0,0}) \big] 
\end{equation}
are bounded and  invertible in $ Y $ and in $ X $, with inverse
$$
U_{\mu,\e}^{-1}  = 
 \big[ 
P_{0,0} P_{\mu,\e}+(\uno-P_{0,0}) (\uno - P_{\mu,\e}) \big] \big( \uno-(P_{\mu,\e}-P_{0,0})^2 \big)^{-1/2} \, , 
$$
 and 
$ U_{\mu,\e} P_{0,0}U_{\mu,\e}^{-1} =  P_{\mu,\e}  
$ as well as $ U_{\mu,\e}^{-1} P_{\mu,\e}  U_{\mu,\e} = P_{0,0} $. 
The map $(\mu, \epsilon)\mapsto  U_{\mu,\e}$ is analytic from  $B(\mu_0)\times B(\epsilon_0)$ to $\cL(Y)$.
The transformation operators $U_{\mu,\e}$  are symplectic and reversibility preserving, namely
\begin{equation}\label{Usr}
   U_{\mu,\e}^* \cJ U_{\mu,\e}= \cJ \, , \qquad 
\bro U_{\mu,\e} = U_{\mu,\e}  \bro \, .
\end{equation}
 Finally $U_{0,\e}$ is a real operator, i.e. $\bar{U_{0,\e}}=U_{0,\e}$.
\\[1mm] 
5. The subspaces $\mathcal{V}_{\mu,\e}=\text{Rg}(P_{\mu,\e})$ are isomorphic one to each other: 
$
\mathcal{V}_{\mu,\e}=  U_{\mu,\e}\mathcal{V}_{0,0}.
$
 In particular $\dim \mathcal{V}_{\mu,\e} = \dim \mathcal{V}_{0,0}=4 $, for any 
 $(\mu, \e) \in B(\mu_0)\times B(\e_0)$.
\end{lem}

We consider the basis  of the subspace $ {\mathcal V}_{\mu,\e} = 
\mathrm{Rg} (P_{\mu,\e}) $, 
\begin{equation}\label{basisF}
\cF := \big\{ 
f_{1}^+(\mu,\e), \  f_{1}^- (\mu,\e), \   f_{0}^+(\mu,\e),\   f_{0}^-(\mu,\e) \big\} \, , \quad 
f_{k}^\sigma(\mu,\e) := U_{\mu,\e} f_{k}^\sigma \, , \ \sigma=\pm \, , \,k=0,1 \, , 
\end{equation}
obtained by applying  the transformation operators $ U_{\mu,\e} $ in \eqref{OperatorU} 
to the eigenvectors $ f_1^+, f_1^-, f_0^+, f_0^- $ in \eqref{funperturbed}
which form a basis of 
$ \mathcal{V}_{0,0} =\mathrm{Rg} (P_{0,0}) $.  
Exploiting the property \eqref{Usr}  that the transformation operators $ U_{\mu,\e}  $ 
are symplectic and reversibility preserving, it is proved in  \cite[Section 3]{BMV3}
 the following lemma.
  
\begin{lem} {\bf (Matrix representation of $ {\cal L}_{\mu,\e}$ on $ \mathcal{V}_{\mu,\e}$)}
The operator $ 	\sL_{\mu,\e}: \mathcal{V}_{\mu,\e}\to\mathcal{V}_{\mu,\e} $ in 
\eqref{calL.ham} defined for $(\mu, \e) \in B_{\mu_0}(0)\times B_{\e_0}(0)$ 
 is  represented on the basis $ \cF $ in \eqref{basisF} by the $4\times 4$ Hamiltonian and reversible matrix
\begin{equation}\label{Lmuepsi}
 \tL_{\mu,\e} =\tJ\tB_{\mu,\e} 
 \quad
\text{where} \quad 
\tJ := \tJ_4 := 
 \begin{pmatrix} 
 \tJ_2&  0 \\
0  & \tJ_2
\end{pmatrix}, \quad 
{\small \tJ_2 := \begin{pmatrix} 
 0 & 1 \\
-1  & 0
\end{pmatrix}} \, , 
\end{equation}
  and 
$  \tB_{\mu,\e}    = \tB_{\mu,\e}^*  $ 
is the $ 4 \times 4 $ self-adjoint matrix
\begin{equation} \label{tocomputematrix0}
\begin{aligned}
\tB_{\mu,\e}  & := 
\begin{pmatrix}
\big(\mathfrak{B}_{\mu,\e} f_1^+, f_1^+ \big) 
& \big(\mathfrak{B}_{\mu,\e} f_1^-, f_1^+ \big)  &
\big(\mathfrak{B}_{\mu,\e} f_0^+, f_1^+ \big)  & 
\big(\mathfrak{B}_{\mu,\e} f_0^-, f_1^+ \big)  \\
\big(\mathfrak{B}_{\mu,\e} f_1^+, f_1^- \big) & 
\big(\mathfrak{B}_{\mu,\e} f_1^-, f_1^- \big) & 
\big(\mathfrak{B}_{\mu,\e} f_0^+, f_1^- \big) &
\big(\mathfrak{B}_{\mu,\e} f_0^-, f_1^- \big)  \\
\big(\mathfrak{B}_{\mu,\e} f_1^+, f_0^+ \big)  & 
\big(\mathfrak{B}_{\mu,\e} f_1^-, f_0^+ \big) & 
\big(\mathfrak{B}_{\mu,\e} f_0^+, f_0^+ \big) & 
\big(\mathfrak{B}_{\mu,\e} f_0^-, f_0^+ \big)  \\
\big(\mathfrak{B}_{\mu,\e} f_1^+, f_0^- \big)  & 
\big(\mathfrak{B}_{\mu,\e} f_1^-, f_0^- \big)  & 
\big(\mathfrak{B}_{\mu,\e} f_0^+, f_0^- \big) &
\big(\mathfrak{B}_{\mu,\e} f_0^-, f_0^- \big) \\
	\end{pmatrix}
\end{aligned}
\end{equation}
where 
\begin{equation}\label{Bgotico}
\mathfrak{B}_{\mu,\e} := P_{0,0}^* \, U_{\mu,\e}^* \, {\cal B}_{\mu,\e} \, U_{\mu,\e} \, P_{0,0}
\, . 
\end{equation} 
The entries of the matrix $ \tB_{\mu,\e} $ 
are alternatively  real or purely imaginary: for any $ \sigma = \pm $, $ k = 0, 1 $, 
the scalar product  $  \molt{ \mathfrak{B}_{\mu,\e}  \, f^{\sigma}_{k}}{f^{\sigma}_{k'}}  $ is real 
and $  \molt{ \mathfrak{B}_{\mu,\e}  \, f^{\sigma}_{k}}{f^{-\sigma}_{k'}} $ 
is purely imaginary. 
The matrix $ \tB_{\mu,\e} $ is analytic in $(\mu,\e)$ close to $ (0,0) $.
\end{lem}

We conclude this section recalling some notation. 
A $ 2n \times 2n $, $ n = 1,2, $ matrix of the form 
$\tL=\tJ_{2n} \tB$ is  \emph{Hamiltonian} if  $ \tB $ is a self-adjoint matrix, i.e.   $\tB=\tB^*$.
It is \emph{reversible} if $\tB$ is reversibility-preserving,   i.e. $\rho_{2n}\circ \tB = \tB \circ \rho_{2n} $, where 
$ 
 \rho_4 := \begin{pmatrix}\rho_2 & 0 \\ 0 & \rho_2\end{pmatrix} $, 
$\rho_2 := \begin{pmatrix} \mathfrak{c}  & 0 \\ 0 & - \mathfrak{c} \end{pmatrix} $
and $\Gc : z \mapsto \bar z $ is the conjugation of the complex plane.
Equivalently, $\rho_{2n} \circ \tL  = -  \tL \circ \rho_{2n}$.

The transformations 
preserving  the Hamiltonian structure  are called
  \emph{symplectic}, and  satisfy
$ Y^* \tJ_4 Y = \tJ_4 $. 
 If $Y$ is symplectic then $Y^*$ and $Y^{-1}$  are symplectic as well. A Hamiltonian matrix $\tL=\tJ_4 \tB$, with $\tB=\tB^*$, is conjugated through a symplectic matrix 
 $Y$ in a new Hamiltonian matrix.
We finally mention 
that the flow of a Hamiltonian 
reversibility-preserving matrix  is symplectic
and reversibility-preserving (see Lemma 3.8 in \cite{BMV1}).

\section{Expansion of $\tB_{\mu,\e} $}\label{sec:Bmue}

In this section we provide the Taylor 
expansion of  the matrix $\tB_{\mu,\e} $  in \eqref{tocomputematrix0}, i.e. 
\eqref{tocomputematrix}, at an order 
of accuracy higher  than 
 in \cite[Proposition 4.3]{BMV3}.  
 In particular we compute the 
 quadratic terms $\gamma_{11} \e^2$, $\phi_{21} \mu\e$, the cubic ones 
$\eta_{12}\mu\e^2$,  $\gamma_{12}\mu \e^2$,$\phi_{11} \e^3$, $\phi_{22} \mu^2 \e$, 
 and the quartic terms $\eta_{11} \e^4$, 
 $\gamma_{22} \mu^2 \e^2$, $\phi_{12} \mu \e^3$ in 
 the matrices \eqref{BinG1}-\eqref{BinG3} below. These are the coefficients 
 which enter in the constant $\etaWB$ (cfr. \eqref{etaWB}) 
of the Benjamin-Feir discriminant function \eqref{DWB}. 

For convenience  we decompose the $ 4 \times 4 $ matrix  $\tB_{\mu,\e} $ 
in \eqref{tocomputematrix0} as
\begin{equation} \label{tocomputematrix}
\begin{aligned}
\tB_{\mu,\e} 
	& =
\begin{pmatrix}
E(\mu,\e) & F(\mu,\e) \\ F^*(\mu,\e) & G(\mu,\e) \end{pmatrix}  
\end{aligned}
\end{equation}
where $ E  $, $ G  $ are the $ 2 \times 2$ self-adjoint matrices 
\begin{subequations}\label{tocomputeentries}
\begin{align}\label{tocomputeentries1}
&E(\mu,\e) 
:=\begin{pmatrix} E_{11}(\mu,\e) & \im  E_{12}(\mu,\e)  \\ -\im  E_{12}(\mu,\e) &  E_{22}(\mu,\e) \end{pmatrix} := \begin{pmatrix} \big(\mathfrak{B}_{\mu,\e} f_1^+, f_1^+ \big) &  \big(\mathfrak{B}_{\mu,\e} f_1^-, f_1^+ \big) \\  \big(\mathfrak{B}_{\mu,\e} f_1^+, f_1^- \big) &  \big(\mathfrak{B}_{\mu,\e} f_1^-, f_1^- \big) \end{pmatrix} \, ,\\  \label{tocomputeentries2}
&G(\mu,\e) 
:=\begin{pmatrix} G_{11}(\mu,\e) & \im  G_{12}(\mu,\e)  \\ -\im  G_{12}(\mu,\e) &  G_{22}(\mu,\e) \end{pmatrix}:= \begin{pmatrix} \big(\mathfrak{B}_{\mu,\e} f_0^+, f_0^+ \big) &  \big(\mathfrak{B}_{\mu,\e} f_0^-, f_0^+ \big) \\  \big(\mathfrak{B}_{\mu,\e} f_0^+, f_0^- \big) &  \big(\mathfrak{B}_{\mu,\e} f_0^-, f_0^- \big) \end{pmatrix} \, , 
\end{align}
and 
\begin{align}
\label{tocomputeentries3}
 &F(\mu,\e):=\begin{pmatrix} F_{11}(\mu,\e) & \im  F_{12}(\mu,\e)  \\ \im  F_{21}(\mu,\e) &  F_{22}(\mu,\e) \end{pmatrix}:= \begin{pmatrix} \big(\mathfrak{B}_{\mu,\e} f_0^+, f_1^+ \big) &  \big(\mathfrak{B}_{\mu,\e} f_0^-, f_1^+ \big) \\  \big(\mathfrak{B}_{\mu,\e} f_0^+, f_1^- \big) &  \big(\mathfrak{B}_{\mu,\e} f_0^-, f_1^- \big) \end{pmatrix} \, .
\end{align}
\end{subequations}
The main result of this section is the following proposition.

\begin{prop}\label{BexpG}
  The $ 2 \times 2 $  matrices $E:=E(\mu,\e)$, $F:=F(\mu,\e)$, $G:=G(\mu,\e)$ defined in  \eqref{tocomputeentries} admit the  expansion 
  \begin{subequations}\label{BinG}
\begin{align}\label{BinG1}
& E = 
\begin{pmatrix} 
  \te_{11} \e^2(1+ r_1'(\e^3,\mu\e)) + \eta_{11} \e^4 - \te_{22}\frac{\mu^2}{8}(1+r_1''(\e,\mu))  
  & 
  \im \big( \frac12\te_{12}\mu {+\eta_{12}\mu\e^2}+ r_2(\mu\e^3,\mu^2\e,\mu^3) \big)  \\
- \im \big( \frac12\te_{12} \mu {+\eta_{12}\mu\e^2}+ r_2(\mu\e^3,\mu^2\e,\mu^3) \big) & -\te_{22}\frac{\mu^2}{8}(1+r_5(\e^{{2}},\mu))
 \end{pmatrix} \\
 & \label{BinG2} G = 
\begin{pmatrix} 
1 {+\gamma_{11} \e^2}+r_8(\e^3, \mu\e^2,\mu^2\e 
) &  {-\im \gamma_{12} \mu\e^2} - \im r_9(\mu\e^3,\mu^2\e) 
 \\
 {\im \gamma_{12} \mu\e^2}+ \im  r_9(\mu\e^3, \mu^2\e 
  )  &\mu\tanh(\tth\mu) {+\gamma_{22} \mu^2\e^2} +r_{10}(\mu^2\e^3,\mu^3\e)
 \end{pmatrix} \\
 & \label{BinG3}
 F = 
\begin{pmatrix} 
\tf_{11}\e {+\phi_{11} \e^3}+ r_3(\e^4,\mu\e^2,\mu^2\e 
) & \im 
 \mu\e \ch^{-\frac12}  {+\im\phi_{12} \mu\e^3} +\im  r_4({\mu\e^{4}}, \mu^2 \e^2, \mu^3\e
 )  \\
 \im  {\phi_{21}\mu\e}+ \im   r_6(\mu\e^{{3}}, \mu^2\e
  )    &   {\phi_{22}\mu^2\e}+  r_7(\mu^2\e^{{3}}, \mu^3\e) 
 \end{pmatrix} \, , 
 \end{align}
 \end{subequations}
 where the coefficients  
 \begin{subequations}\label{mengascoeffs}
  \begin{align}
&\te_{11} 
  := 
 \dfrac{9\ch^8-10\ch^4+9}{8\ch^7} = \dfrac{9 (1-\ch^4)^2 +8 \ch^4 }{8\ch^7} > 
 0 \, , \qquad
 \tf_{11} :=    \tfrac12 \ch^{-\frac32}(1-\ch^4) \, , \\
&\te_{12}:= \ch+\ch^{-1}(1-\ch^4)\tth > 0  \, , \label{te12} \\
 &\te_{22}:= \dfrac{(1-\ch^4)(1+3\ch^4) \tth^2+2 \ch^2(\ch^4-1) \tth+\ch^4}{\ch^3} > 0  \, , \label{te22} 
  \end{align}
  were computed in \cite[Proposition 4.3]{BMV3}, whereas 
  \begin{align}\label{te11}
&\eta_{11}:=\frac1{256 \ch^{19} (\ch^2+1)}\big(-36 \ch^{26}-108 \ch^{24}-261 \ch^{22}-73 \ch^{20}+1429 \ch^{18}+1237 \ch^{16}\\ \notag
&\qquad -3666 \ch^{14}-3450 \ch^{12}+3774 \ch^{10}+3654 \ch^8-873 \ch^6-765 \ch^4+81 \ch^2+81\big)\, ,  \\ 
 \label{eta12}
&\eta_{12}:= \frac{\ch^2 \big(3 \ch^{12}-8 \ch^8+3 \ch^4+18\big)-\big(\ch^{16}-2 \ch^{12}+12 \ch^8-38 \ch^4+27\big) \tth}{16 \ch^9} \, , \\ \label{gammas}
&\gamma_{11} :=  \frac{-\ch^8+6 \ch^4-5}{8 \ch^4}  \, , \quad 
\gamma_{12} := \frac{2 \ch^{12}-\ch^8-9}{16 \ch^7}\, , \quad
\gamma_{22} := \frac{\ch^4-5}{4 \ch^2} \, , \\ \label{phi11} 
&\phi_{11} :=\frac{10 \ch^{20}+4 \ch^{18}-7 \ch^{16}-6 \ch^{14}-99 \ch^{12}+257 \ch^8-6 \ch^6-171 \ch^4+18}{64 \ch^{27/2}} \, ,\\ \label{phi12}
&\phi_{12}:= \frac{2 \ch^{18}-2 \ch^{16}-33 \ch^{14}-27 \ch^{12}+34 \ch^{10}+34 \ch^8-33 \ch^6-27 \ch^4+18 \ch^2+18}{32 \ch^{25/2} (\ch^2+1)}  \\ \label{phi2s}
&\phi_{21} :=\frac{\ch^2 (\ch^4-5)-(\ch^8+2 \ch^4-3) \tth}{8 \ch^{7/2}} \, , \quad
\phi_{22}:= \frac{-\ch^4 \tth+\ch^2+\tth}{4 \ch^{5/2}}\, .
\end{align}
\end{subequations}
\end{prop}
The rest of the section is devoted to the proof of this proposition.

In \cite[Proposition 4.3]{BMV3} we showed that
the matrices $ E,  G, F $ in \eqref{BinG1}, \eqref{BinG2}, \eqref{BinG3}  admit the following expansions
\begin{subequations}
\begin{align}\label{BinG1old}
 &E(\mu,\e)= \begin{pmatrix} 
  \te_{11} \e^2 - \te_{22}\frac{\mu^2}{8}
  & 
  \im\frac12\te_{12}\mu  \\
- \im  \frac12\te_{12} \mu & -\te_{22}\frac{\mu^2}{8} 
 \end{pmatrix} + \underbrace{ \begin{pmatrix} 
  r_1(\e^3,  \mu^2 \e,\mu^3) 
  & 
  \im r_2(\mu\e^2,\mu^2\e,\mu^3)   \\
- \im  r_2(\mu\e^2,\mu^2\e,\mu^3)  & r_5(\mu^2\e,\mu^3)
 \end{pmatrix}}_{=:\Eta(\mu,\e)}\, , \\
  \label{BinG2old} 
&  G(\mu,\e) = 
\begin{pmatrix} 
1 &   0
 \\
 0  &\mu\tanh(\tth\mu)
 \end{pmatrix} + \underbrace{ \begin{pmatrix} 
r_8(\e^2,\mu^2\e) &   - \im r_9(\mu\e^2,\mu^2\e) 
 \\
  \im  r_9(\mu\e^2, \mu^2\e  )  &r_{10}(\mu^2\e ) \, ,
 \end{pmatrix}}_{=:\Gamma(\mu,\e)} \\
 \label{BinG3old}
 &F(\mu,\e) := \begin{pmatrix} 
\tf_{11}\e & \im 
 \mu\e \ch^{-\frac12}    \\
   0  & 0
 \end{pmatrix} +
 \underbrace{\begin{pmatrix} 
 r_3(\e^3,\mu\e^2,\mu^2\e ) &\im  r_4({\mu\e^2}, \mu^2 \e )  \\
  \im   r_6(\mu\e )
     & r_7(\mu^2\e  ) 
 \end{pmatrix}}_{=:  \Phi(\mu,\e)} \, .
 \end{align}
 \end{subequations}
In order to get the expansion of 
  $\Eta(\mu,\e)$, $\Gamma(\mu,\e)$ and $\Phi(\mu,\e)$ 
  in Proposition \ref{BexpG} we first 
  expand the operators ${\cal B}_{\mu,\e}$ in \eqref{calL.ham} (Section \ref{secBmue}), 
  the projector $P_{\mu,\e}$ in \eqref{Pproj} (Section \ref{secPmue})
  and  the  operator $\mathfrak{B}_{\mu,\e}$ in \eqref{Bgotico} (Section \ref{secPgotmue}). 
  Finally we prove Proposition \ref{BexpG} in Section \ref{secProof31}. 
 
 \smallskip

\noindent{\bf Notation.} 
For an operator $A=A(\mu,\e)$ we denote its Taylor coefficients as
\begin{equation}\label{notazione}
A_{i,j} := \frac{1}{i!j!} \big(\pa^i_\mu\pa^j_\e A \big)(0,0)\, ,\qquad  A_k := A_k (\mu,\e) :=  \sum_{\substack{i+j=k\\ i,j\geq 0}} A_{i,j} \mu^i \e^j\, .
\end{equation}
Moreover we shall occasionally split $A_{i,j}=A_{i,j}^{[\mathtt{ ev}]} + A_{i,j}^{[\mathtt{odd}]}$, where $A_{i,j}^{[\mathtt{ ev}]}$ is the part of the operator $A_{i,j}$ having only even harmonics, whereas $A_{i,j}^{[\mathtt{odd}]}$ is the part having only odd ones. \smallskip

\subsection{Expansion of ${\cal B}_{\mu,\e}$}\label{secBmue}

In the sequel 
$\cO_5$ means an operator which maps $H^1(\bT,\bC^2)$ into 
$L^2(\bT,\bC^2)$-functions with size $\e^5$, $\mu \e^4$, $\mu^2\e^3$, $\mu^3\e^2$, $\mu^4\e$ or $\mu^5$.
 
 \begin{lem}
 The operator ${\cal B}_{\mu,\e}$ in \eqref{calL.ham}  has the Taylor expansion 
\begin{equation}\label{def.Bj}
{\cal B}_{\mu,\e} = {\cal B}_0 + {\cal B}_1 + {\cal B}_2+ {\cal B}_3+{\cal B}_4+\cO_5 \, ,  
\end{equation}
 where
\begin{subequations}\label{Bsani}
\begin{alignat}{2} \label{B0sano}
{\cal B}_0 &=
 &&\begingroup 
\setlength\arraycolsep{3pt}\begin{bmatrix} 1 & -\ch \pa_x \\ \ch \pa_x &  |D|\tanh\big(\tth |D|\big)\end{bmatrix}\endgroup \, , \\
  \label{B1sano} {\cal B}_{1} &=
 \e  &&\begingroup 
\setlength\arraycolsep{3pt} \begin{bmatrix} a_1(x) & -p_1(x)\pa_x\\ \pa_x\circ p_1(x) & 0 \end{bmatrix}\endgroup +\mu \ell_{1,0}(|D|) \Pi_{\mathtt{s}}  \, ,  \\
\label{B2sano}
 {\cal B}_{2} &= 
 \e^2 &&\begingroup 
\setlength\arraycolsep{3pt} \begin{bmatrix} a_2(x) & -p_2(x)\pa_x \\ \pa_x\circ p_2(x) & \ell_{0,2}(|D|) \end{bmatrix} \endgroup -\im \mu\e  p_1(x) \cJ +\mu^2 \ell_{2,0}(|D|) \Pi_{\mathtt{ ev}} \,  , \\
 \label{B3sano}
  {\cal B}_{3} &=  
 \e^3 &&\begingroup 
\setlength\arraycolsep{3pt} \begin{bmatrix} a_3(x) & -p_3(x)\pa_x \\ \pa_x\circ p_3(x) & 0 \end{bmatrix}\endgroup-\im \mu\e^2 p_2(x) \cJ +\mu^3\ell_{3,0}(|D|) \Pi_{\mathtt{s}} ,\\
  \label{B4sano}
  {\cal B}_{4} &=  
 \e^4 &&\begingroup 
\setlength\arraycolsep{3pt} \begin{bmatrix} a_4(x) & -p_4(x)\pa_x \\ \pa_x\circ p_4(x) & \ell_{0,4}(|D|)  \end{bmatrix}\endgroup 
-\im \mu\e^3 p_3(x) \cJ + \mu^2 \e^2 \ell_{2,2}(|D|) \Pi_{\mathtt{ ev}} + \mu^4 \ell_{4,0}(|D|) \Pi_{\mathtt{ ev}}\, ,
\end{alignat}
\end{subequations}
and  $p_i(x)$ and $a_i(x)$, $i=1,\dots,4$, are computed in \eqref{pino1fd}-\eqref{aino2fd}, 
$\cJ$ is the symplectic matrix in \eqref{PoissonTensor}, 
\begin{equation}\label{Pi_odd}
\Pi_{\mathtt{s}}:=\begin{bmatrix}  0 & 0 \\ 0 & \sgn(D) \end{bmatrix}\, , \qquad \Pi_{\mathtt{ev}}:=\begin{bmatrix}  0 & 0 \\ 0 & \uno \end{bmatrix}\, ,
\end{equation}
and
\begin{subequations}
\begin{align} \label{ell10}
\ell_{1,0}(|D|)&=\tanh(\tth |D|)+\tth|D| \big(1-\tanh^2(\tth |D|)\big)\, ,\\  \label{ell20}
\ell_{2,0}(|D|) &=\tth  \big(1-\tanh^2(\tth |D| )\big) \big( 1-\tth |D| \tanh(\tth |D|) \big) \, ,
\\  \label{ell02}
\ell_{0,2}(|D|) &= \tf_2 |D|^2 \big(1-\tanh^2(\tth |D|)\big)\, , \\ \label{ell22}
\ell_{2,2}(|D|) &={\footnotesize \begin{matrix}\tf_2 \big(1-\tanh^2(\tth |D|)\big)\left(-\tth^2 |D|^2+3 \tth^2 |D|^2 \tanh^2(\tth |D|)-4 \tth |D| \tanh (\tth |D|)+1 \right)  \end{matrix} }
\, ,\\ 
 \label{ell04}
\ell_{0,4}(|D|) &= \ttf_4|D|^2 \big(1-\tanh^2(\tth |D| ) \big) - \ttf_2^2 |D|^3 \tanh(\tth|D|)\big(1-\tanh^2(\tth |D| ) \big) \, , 
\end{align}
\end{subequations} 
 with $\ttf_2$ and $ \ttf_4 $  in \eqref{expfe}. 
\end{lem}

\begin{proof}
By Taylor expanding \eqref{calL.ham}. 
\end{proof}

We observe that, using the notation introduced in \eqref{notazione}, we have
\begin{equation}\label{cB evodd}
\cB_{i,j}^{[\mathtt{ ev}]} = \begin{cases} \cB_{i,j} &\textup{if }j\textup{ is even}\, , \\ 
0 &\textup{if }j\textup{ is odd}\, , 
\end{cases} \qquad \cB_{i,j}^{[\mathtt{odd}]} = \begin{cases} 0 &\textup{if }j\textup{ is even}\, , \\ 
\cB_{i,j}  &\textup{if }j\textup{ is odd}\, . 
\end{cases} 
\end{equation}

\subsection{Expansion of the projector $P_{\mu,\e}$}\label{secPmue}

The projectors $P_{\mu,\e}$ in  \eqref{Pproj} admit the expansion 
\begin{equation}\label{expP}
P_{\mu,\e} = P_0 + P_1 + P_2 + P_3 + \cO_4 \, ,
\end{equation}
where
\begin{equation}\label{Psani}
\begin{aligned}
 & P_0 := P_{0,0} \, , \quad P_1 := \mathcal{P} \big[\cB_1\big] \\
 & P_2 := \mathcal{P} \big[\cB_2\big] + \mathcal{P} \big[\cB_1,\cB_1 \big] \, ,\\
  &P_3 := \mathcal{P} \big[\cB_3\big] + \mathcal{P} \big[\cB_2,\cB_1 \big] + \mathcal{P} \big[\cB_1,\cB_2 \big]  + \mathcal{P} \big[\cB_1,\cB_1,\cB_1 \big] \, ,
 \end{aligned} 
 \end{equation}
and 
\begin{equation}\label{hP}
\begin{aligned}
&\mathcal{P} \big[A_1]: = \frac{1}{2\pi\im} \oint_\Gamma (\cL_{0,0}-\lambda)^{-1} \cJ A_1 (\cL_{0,0}-\lambda)^{-1} \de\lambda \, ,\quad \text{and for }k\geq 2 \\
&\mathcal{P} \big[A_1,\dots, A_k \big] := \frac{(-1)^{k+1} }{2\pi \im } \oint_\Gamma (\cL_{0,0}-\lambda)^{-1} \cJ A_1 (\cL_{0,0}-\lambda)^{-1} \dots \cJ A_k (\cL_{0,0}-\lambda)^{-1} \de\lambda \, .
 \end{aligned}
\end{equation}
In virtue of \eqref{notazione}, \eqref{Psani}-\eqref{hP} and \eqref{cB evodd} we obtain
\begin{equation}
\label{P evodd}
P_{i,j}^{[\mathtt{ ev}]} = \begin{cases} P_{i,j} &\textup{if }j\textup{ is even}\, , \\ 
0 &\textup{if }j\textup{ is odd}\, , 
\end{cases} \qquad P_{i,j}^{[\mathtt{odd}]} = \begin{cases} 0 &\textup{if }j\textup{ is even}\, , \\ 
P_{i,j}  &\textup{if }j\textup{ is odd}\, . 
\end{cases} 
\end{equation}

\paragraph{Action of $P_{\ell,j}$ on the unperturbed vectors.}
We now collect how the operators $P_{\ell,j}$ act on the vectors 
$ f_1^+,  f_1^- , f_0^+, f_0^-$ in \eqref{funperturbed}. 
We denote 
\begin{equation}
\label{fm1sigma} 
  f_{-1}^+ :=\vet{\ch^{1/2}\cos(x)}{-\ch^{-1/2}\sin(x)}\, ,\qquad 
f_{-1}^- :=\vet{\ch^{1/2}\sin(x)}{\ch^{-1/2}\cos(x)}  \  .
\end{equation}
We first consider the first order jets $ P_{0,1} $ and $P_{1,0} $ of $ P_1 $. 

\begin{lem} {\bf (First order jets)} The action of the jets $ P_{0,1} $ and $P_{1,0} $ of $ P_1 $
 in \eqref{Psani} on the basis  in \eqref{funperturbed} is 
\begin{equation}\label{tuttederivate}
\begin{aligned}
 & P_{0,1}f^+_1 ={ 
  \vet{\ka_{0,1}  \cos(2x)}{ \kb_{0,1} \sin(2x)}  } \, , \quad P_{0,1}f^-_1 = {
  \vet{-\ka_{0,1} \sin(2x)}{ \  \kb_{0,1}\cos(2x)}  } \, , \\
&   P_{0,1}f^+_0 =  \ku_{0,1} f^+_{-1} \, , \quad   P_{0,1}f^-_0 =0 \, , \quad    P_{1,0} f_0^+=0 \, ,\quad  
  P_{1,0} f_0^-=0 \, ,   \\
 & P_{1,0} f_1^+ = \im \ku_{1,0}  f^{-}_{-1} \, ,
 \quad
   P_{1,0} f_1^- = \im \ku_{1,0} f^+_{-1}\, ,
\end{aligned}
\end{equation}
where
\begin{equation}\label{u01}
\begin{aligned}
&\ka_{0,1} := \frac12 \ch^{-\frac{11}{2}} (3+\ch^4) \, ,\quad \kb_{0,1} := \frac14 \ch^{-\frac{13}{2}} (1+\ch^4)(3-\ch^4) \,, \\
& \ku_{0,1} :=  \tfrac14\ch^{-\frac52}(3+\ch^4) \, ,\quad \ku_{1,0}:= \tfrac14 \big(1+\ch^{-2}\tth (1-\ch^4)\big)\, .
\end{aligned}
\end{equation}
\end{lem}

\begin{proof}
See \cite[formula (A.16)]{BMV3}.
\end{proof}

\begin{lem}\label{lem:secondjetsP}
{\bf (Second order jets)}
The  action of the jet $P_{0,2}$ of $P_2$ in \eqref{Psani} on the basis  in \eqref{funperturbed} is given by
\begin{subequations}\label{secondjetsP}
\begin{align} \label{secondjetsPbis}
&P_{0,2} f_1^+ =  \kn_{0,2} f_1^+  +\ku_{0,2}^+ f_{-1}^+ 
+ \vet{\ka_{0,2}\cos(3x) }{\kb_{0,2} \sin(3x)}\, , 
\quad P_0 P_{0,2} f_0^+   = 0 \, ,
 \\ 
& P_{0,2} f_1^- = \kn_{0,2} f_1^- +  \ku_{0,2}^-  f_{-1}^- + \vet{\widetilde \ka_{0,2}\sin(3x)}{\widetilde \kb_{0,2} \cos(3x)} 
\, , \quad P_{0,2} f_0^- = 0\, ,\label{secondjetsPeps}
\end{align}
where $ \widetilde \ka_{0,2},  \widetilde \kb_{0,2} \in \bR $ and   
\begin{align}
&\kn_{0,2} :=  \frac{\ch^{12}+\ch^8-9 \ch^4-9}{8 \ch^{12}}  \, , \quad \ku_{0,2}^+ := \frac{-2 \ch^{12}-7 \ch^8+8 \ch^4+9}{32 \ch^8}\,,  
\quad \ku_{0,2}^- := \frac{2 \ch^{12}-11 \ch^8+20 \ch^4-3}{32 \ch^8}\, , \notag  
\\
&\ka_{0,2} := \dfrac{3 (\ch^{12}+17 \ch^8+51 \ch^4+27)}{64 \ch^{23/2}}\, , \quad \kb_{0,2} := \dfrac{3 (3 \ch^{12}-5 \ch^8+25 \ch^4+9)}{64 \ch^{25/2}}\, . \label{u02-}
\end{align}
The  action of the jet $P_{2,0}$ on  the vector $f_0^-$  in \eqref{funperturbed}  is 
\begin{equation}\label{secondjetsPmu}
P_{2,0} f_0^- = 0\, .
\end{equation}
The  action of the jet $P_{1,1}$ on  the basis  in \eqref{funperturbed}  is 
\begin{align} \notag
&P_{1,1} f_1^+ = \im  \widetilde \km_{1,1}  f_0^- + \im 
 \vet{ \widetilde \ka_{1,1} \sin(2x)}{ \widetilde \kb_{1,1} \cos(2x) }   \, , \qquad 
P_{1,1} f_1^- = 
\im \vet{\ka_{1,1} \cos (2 x)}{\kb_{1,1} \sin (2 x)}
\, ,\\
&  P_{1,1} f_0^- = -\frac{\im}{2} \ch^{-3/2}  f_{-1}^+   
\, , \qquad 
P_{1,1} f_0^+ = \im \widetilde \kn_{1,1}  f_1^- + \im  \widetilde \ku_{1,1} f_{-1}^- 
 \, ,\label{secondjetsPmix}
\end{align}
\end{subequations}
where $ \widetilde \km_{1,1}, \widetilde \ka_{1,1}, \widetilde \kb_{1,1},
\widetilde \kn_{1,1}, \widetilde \ku_{1,1}  \in \bR $ and 
\begin{equation}
\label{n11a11} 
\ka_{1,1}:=-\dfrac{3 (\ch^8-6 \ch^4+5) \tth-3\ch^2 (\ch^4+3)}{8 \ch^{15/2}} \, ,\quad \kb_{1,1} :=\dfrac{(\ch^8+8 \ch^4-9) \tth+3 (\ch^6+\ch^2)}{8 \ch^{17/2}} \, .
\end{equation}
\end{lem}

\begin{proof}
In Appendix \ref{conticinoleggeroleggero}.
\end{proof}

\begin{lem}\label{lem:P03acts}
{\bf (Third order jets)}
The  action of the jets $P_{0,3}$, $P_{1,2} $ and $P_{2,1}$ of $P_3$ in \eqref{Psani} on  $f_1^+$, $f_0^-$ is
\begin{equation}\label{P03acts}
\begin{aligned}
&P_{0,3} f_1^+ = \vet{\ka_{0,3} \cos(2x)}{\kb_{0,3} \sin(2x)}
 +  \vet{\widetilde \ka_{0,3} \cos(4x)}{\widetilde \kb_{0,3} \sin(4x)} \, ,
 \quad  P_{0,3} f_0^- = 0  \, , \\
&P_{1,2} f_0^- = \im \vet{\ka_{1,2} \cos (2 x)}{\kb_{1,2} \sin (2 x)} \, ,  \quad
P_{2,1} f_0^- =  \widetilde \kn_{2,1} f_1^- + \widetilde \ku_{2,1} f_{-1}^- \,  ,
\end{aligned}
\end{equation}
where $ \widetilde \ka_{0,3}, \widetilde \kb_{0,3}, \widetilde \kn_{2,1}, \widetilde \ku_{2,1} \in \bR $ and 
\begin{equation}\label{n03}
\begin{aligned}
&\ka_{0,3} := \frac1{64 \ch^{35/2}(\ch^2+1)} \Big(6 \ch^{22}+2 \ch^{20}+27 \ch^{18}+21 \ch^{16}-379 \ch^{14} \\
&\qquad \qquad -361 \ch^{12}+575 \ch^{10}+581 \ch^8-243 \ch^6-225 \ch^4-162 \ch^2-162 \Big)\, ,\\ 
&\kb_{0,3}:= \frac1{128 \ch^{37/2} (\ch^2+1)} \Big(6 \ch^{26}+10 \ch^{24}+35 \ch^{22}+21 \ch^{20}-146 \ch^{18}-146 \ch^{16}\\ 
&\qquad \qquad-46 \ch^{14}-34 \ch^{12}+470 \ch^{10}+482 \ch^8-333 \ch^6-315 \ch^4-162 \ch^2-162\Big)\, ,\\ 
& \ka_{1,2} = -\frac{ \ch^4+3}{4 \ch^7} , \quad
\kb_{1,2} := \frac{\ch^4+1 }{4 \ch^4}\, .
 \end{aligned}
 \end{equation}
\end{lem}

\begin{proof}
In Appendix \ref{conticinoleggeroleggero}.
\end{proof}

\subsection{Expansion of $\mathfrak{B}_{\mu,\e} $}\label{secPgotmue}

    In this section we provide the expansion of the operator $\mathfrak{B}_{\mu,\e} $ defined in \eqref{Bgotico}.
We introduce the  notation $\textup{\textbf{Sym}}[A]:= \frac12 A+ \frac12 A^*$.

\begin{lem}[{\bf Expansion of $\mathfrak{B}_{\mu,\e}$}\null] \label{lem:expBgot}
The operator $\mathfrak{B}_{\mu,\e} $ in \eqref{Bgotico} has the Taylor expansion 
 \begin{equation}\label{B.Tay.ex}
\mathfrak{B}_{\mu,\e} = \sum_{j=0}^4 \mathfrak{B}_j +\cO_5 
\end{equation}
where
\begin{subequations}\label{upto4exp} 
\begin{align}\label{ordini012}
\mathfrak{B}_0 & :=P_0^*{\cal B}_0P_0, 
 \quad 
 \mathfrak{B}_1 :=P_0^* {\cal B}_1 P_0, 
 \quad 
 \mathfrak{B}_2 := P_0^* \mathbf{Sym}\big[{\cal B}_2+{\cal B}_1P_1\big]P_0 \, , 
 \\  \label{ordine3}
 \mathfrak{B}_3 & := P_0^*\mathbf{Sym}\big[{\cal B}_3+{\cal B}_2P_1 + {\cal B}_1(\uno-P_0)P_2\big]P_0\, , 
 \\ \label{ordine4}
  \mathfrak{B}_4  &:= P_0^*\mathbf{Sym}\big[{\cal B}_4+ {\cal B}_3P_1 + {\cal B}_2(\uno-P_0)P_2  + {\cal B}_1(\uno-P_0)P_3
    -  {\cal B}_1P_1 P_0 P_2  + \mathfrak{N}P_0P_2\big] P_0\, , 
\end{align}
\end{subequations}
the operators $P_0, \ldots, P_3 $ are defined in \eqref{Psani} and 
\begin{equation}\label{deffrakN}
\mathfrak{N}:=  \frac14 \big( P_2^* {\cal B}_0-{\cal B}_0 P_2 \big) = - \mathfrak{N}^*  \,  . 
\end{equation} 
It results 
 \begin{equation}\label{kN}
 \big(\mathfrak{N} f_k^\sigma, f_{k'}^{\sigma'} \big)= 0\, ,\quad \forall f_k^\sigma, f_{k'}^{\sigma'}\in\{f_1^+,f_1^-,f_0^- \}\, .
 \end{equation}
\end{lem}

\begin{proof}
In order to expand 
$\mathfrak{B}_{\mu,\e} $ in \eqref{Bgotico} we first expand $U_{\mu,\e} P_0 $. 
In view of \eqref{OperatorU} we have, 
introducing the analytic  function 
$  g(x):= (1-x)^{-\frac12} $ for  
$  |x|<1 $, 
\begin{equation}\label{UP}
U_{\mu,\e} P_0 =    g((P_{\mu,\e}-P_0)^2)  \,  P_{\mu,\e}\, P_0  = P_{\mu,\e}\, g((P_{\mu,\e}-P_0)^2)  \,  P_0  \ , 
\end{equation}
using  that 
 $(P_{\mu,\e}-P_0)^2$ commutes with $P_{\mu,\e}$, and so  does 
$g((P_{\mu,\e}-P_0)^2)$.
The Taylor expansion  $ g(x) = 1+\frac12x+\frac38 x^2+ \cO(x^3) $  implies that 
\begin{equation}\label{geppoexp}
g((P_{\mu,\e}-P_0)^2) = 
\uno + \underbrace{\frac12 (P_{\mu,\e}-P_0)^2}_{=:g_2}  +\underbrace{\frac38 (P_{\mu,\e}-P_0)^4}_{=:g_4} + \,  \cO_6 \, ,
\end{equation}
where $\cO_6 = \cO((P_{\mu,\e}-P_0)^6) \in \cL(Y)$. 

Furthermore, since $P_{\mu, \e} \sL_{\mu,\e}  = 	\sL_{\mu,\e} P_{\mu, \e}$ (see Lemma \ref{lem:Kato1}- item 2), 
applying  $\cJ$ to both sides and using  \eqref{propPU}, yields 
\begin{equation}\label{PB}
P_{\mu, \e}^* {\cal B}_{\mu,\e}  = {\cal B}_{\mu,\e} P_{\mu, \e} \qquad \text{where}
\qquad P_{\mu, \e}^2 = P_{\mu, \e} \, . 
\end{equation} 
Therefore the  operator $\mathfrak{B}_{\mu, \e} $ in \eqref{Bgotico} has the expansion 
\begin{align}
\notag
\mathfrak{B}_{\mu, \e} &
 \eq^{\eqref{UP}} 
 P_0^*\,  g((P_{\mu,\e}-P_0)^2)^* \, P_{\mu,\e}^* \, {\cal B}_{\mu,\e} \, 
 P_{\mu,\e} \, g((P_{\mu,\e}-P_0)^2) \, P_0 \\
 \notag
 &\stackrel{\eqref{PB}}{=} 
  P_0^* \, g((P_{\mu,\e}-P_0)^2)^* \, {\cal B}_{\mu,\e} \, P_{\mu,\e} \, 
  g((P_{\mu,\e}-P_0)^2) \, P_0 \\
 \notag
& \stackrel{\eqref{geppoexp}} =  
  P_0^* \, (\uno + g_2^* + g_4^*+\cO_6) \, {\cal B}_{\mu,\e} \, 
P_{\mu,\e} \, \big(\uno+g_2+g_4+\cO_6\big) \,   P_0 \\
\label{expansionB}
&\;\;\;=\;\;\,  \mathbf{Sym}\big[ P_{0}^* \big({\cal B}_{\mu,\e}P_{\mu,\e} + 2{\cal B}_{\mu,\e}P_{\mu,\e} g_2 + g_2^* \cB_{\mu,\e} P_{\mu,\e} g_2 + 2{\cal B}_{\mu,\e}P_{\mu,\e} g_4  \big) P_{0} \big]
 \,  + \cO_6
\end{align}
using \eqref{PB} and that $ g_2 = \cO_2 $ and $ g_4 = \cO_4 $. 

A further analysis of the term \eqref{expansionB} relies on the following lemma.

\begin{lem}\label{lemmaboh}Let $\Pi_0^+$ be the orthogonal projector on $f_0^+$ in \eqref{funperturbed} and $\Pi^\angle:= \uno-\Pi_0^+ $. 
One has
\begin{align} \label{B.id}
&{\cal B}_0 \, P_0 = P_0^* \, {\cal B}_0  = \Pi_0^+ \, ,  \qquad 
{\cal B}_0 \, P_1 + {\cal B}_1 \, P_0 = P_1^* \, {\cal B}_0 +
P_0^* \, {\cal B}_1 \, ,  \\
& \label{P.id} P_0 \, P_1 \, P_0 = 0 \, , \qquad P_0 \, P_2 \, P_0 = - P_1^2 \, P_0  = - P_0 \, P_1^2 \, ,  \\
& \label{PP0eq} (P_{\mu, \e} - P_0)^2  P_{0} = 
P_{0} ( \uno - P_{\mu,\e} ) \, P_{0} \, , \qquad (P_{\mu, \e} - P_0)^4  P_{0} = 
P_{0} \big( ( \uno - P_{\mu,\e} )  P_{0} \big)^2 \, ,  \\
& \label{lastid} 
\cJ P_j = P_j^* \cJ  \, , \ \   \forall j \in \bN_0 \, , \qquad  
P_0^* \cB_0 P_j \Pi^\angle P_0 = \Pi_0^+ P_j\Pi^\angle P_0 \, .
\end{align}
\end{lem}

\begin{proof}
We deduce that $ {\cal B}_0 \, P_0  = \Pi_0^+ $ because 
 $ \cB_0 f_1^+ = \cB_0 f_1^- = \cB_0 f_0^- = 0 $, $  \cB_0 f_0^+=f_0^+ $
and the first identity in \eqref{B.id} follows also since   
$ P_0^* {\cal B}_0  = {\cal B}_0 P_0 $ by \eqref{PB}. 
The second one  follows by expanding the identity in  \eqref{PB} at order 1, using the expansions of  $P_{\mu,\e}$ and $B_{\mu,\e}$ in \eqref{expP} and \eqref{def.Bj}.
The identities in \eqref{P.id} follow by expanding
the identity $P_{\mu,\e}^2 = P_{\mu,\e}$ at order 1 and 2, getting
$ P_1 P_0 + P_0 P_1 = P_1 $ and 
$ P_2 P_0 + P_1^2 + P_0 P_2 = P_2 $, 
and applying $P_0$ to the right and the left of the identities above. The first
identity in \eqref{PP0eq} is verified using that $P_{\mu,\e}^2 = P_{\mu,\e}$ 
 and the second one follows by applying the first one twice. Finally the first identity in \eqref{lastid} follows by expanding the identity $\cJ P_{\mu,\e} = P_{\mu,\e}^*\cJ  $ in  \eqref{propPU} into homogeneous orders. The last identity in 
 \eqref{lastid}  descends from the first of \eqref{B.id}, since for any $g\in L^2(\bT,\bC^2)$ and $f_k^\sigma\in\{f_1^+,f_1^-,f_0^-\}$ one has
$$ \big(P_0^* \cB_0 P_j f_k^\sigma ,g \big)= \big( P_j f_k^\sigma ,  \cB_0 P_0 g \big) =   \big( P_j f_k^\sigma ,  \Pi_0^+ g \big)  = \big( \Pi_0^+ P_j  f_k^\sigma, g \big) \, . $$
This concludes the proof of the lemma.
\end{proof}
By \eqref{PP0eq} and \eqref{P.id} the Taylor expansions of $g_2P_0 $ and $g_4P_0$ in \eqref{geppoexp} are
\begin{subequations}
\begin{align}
\label{g2}
g_2 P_0 &= 
\frac12 P_0 (\textup{Id} - P_{\mu,\e}) P_0 = 
-\frac12 P_0 \, P_2  \, P_0
- \frac12   P_0 \,P_3 \, P_0 
 - \frac12  P_0 \,P_4 \, P_0 + \cO_5 \, , \\
\label{g4} g_4 P_0 &= 
\frac38 P_0  (\textup{Id} - P_{\mu,\e}) P_0 
 (\textup{Id} - P_{\mu,\e}) P_0  = \frac38 P_0 \, P_2 \, P_0 \, P_2 \, P_0
 +  \cO_5 \, .
\end{align}
\end{subequations}
We now Taylor expand the operators in \eqref{expansionB} and collect the terms of the same order.\\
\underline{Expression of $\mathfrak{B}_0$:} The term of order 0 in \eqref{expansionB} is simply $\mathfrak{B}_0 = P_0^* {\cal B}_{0}P_0 $. \\
\underline{Expression of $\mathfrak{B}_1$:} The term of order 1 is  
$$
\mathfrak{B}_1 = \frac12 P_0^* \big( {\cal B}_0 P_1 + {\cal B}_1 P_0 + P_1^* {\cal B}_0 + P_0^* {\cal B}_1 \big) P_0 
= P_0^* \, {\cal B}_1 \, P_0 
$$
using \eqref{P.id} and \eqref{B.id}
and that ${\cal B}_0, {\cal B}_1$ are self-adjoint.\\
\underline{Expression of $\mathfrak{B}_2$:} We compute the terms of order 2
in \eqref{expansionB}. By \eqref{g2} we get 
\begin{equation}\label{expkBord2par}
\mathfrak{B}_2 
=   \mathbf{Sym}[P_0^* \big({\cal B}_2 P_0 + {\cal B}_1P_1 + {\cal B}_0P_2 - {\cal B}_0 P_0 P_2  \big) P_0 ]  \, .
\end{equation}
Moreover 
$$
P_0^* \left( {\cal B}_0P_2 - {\cal B}_0 P_0 P_2  \right) P_0
=
P_0^*  {\cal B}_0 (\textup{Id} - P_0) P_2  P_0 \stackrel{\eqref{B.id}} = 
  {\cal B}_0  P_0 (\textup{Id} - P_0) P_2  P_0 = 0 \ ,
$$
and  from \eqref{expkBord2par} descends
the expression of  $\mathfrak{B}_2  $ in \eqref{ordini012}.\\
\underline{Expression of $\mathfrak{B}_3$:}  We compute the terms of order 3
in \eqref{expansionB}.
By \eqref{g2}, identity \eqref{ordine3} follows from
\begin{align*}
\mathfrak{B}_3 & =  \mathbf{Sym} [P_0^* \big({\cal B}_3 + {\cal B}_2 P_1 + {\cal B}_1 P_2 
+ {\cal B}_0 P_3
- 
 ({\cal B}_0P_1 + {\cal B}_1P_0)P_0 \, P_2 \, P_0
 - {\cal B}_0 \, P_0\,  P_3 \,  P_0 \big) P_0] \\
& = \mathbf{Sym}[P_0^* \big({\cal B}_3 + {\cal B}_2 P_1 + {\cal B}_1 P_2  - {\cal B}_1P_0P_2 
\big) P_0] ,
\end{align*}
where we used 
$P_0^* {\cal B}_0 P_3 = P_0^* {\cal B}_0  P_0 P_3   $ 
and
$P_0^* {\cal B}_0 P_1 P_0 P_2 P_0 = 0$
by \eqref{B.id} and \eqref{P.id}.\\
\underline{Expression of $\mathfrak{B}_4$:} 
At the fourth order we get, in view of \eqref{g2} and \eqref{g4},
\begin{align} \notag
\mathfrak{B}_4 & = 
 \mathbf{Sym} \big[
 P_0^* \Big( {\cal B}_0 P_4 + {\cal B}_1 P_3 + {\cal B}_2 P_2 + {\cal B}_3 P_1 + {\cal B}_4 - {\cal B}_0 P_0 P_4 P_0 - ({\cal B}_0P_1 +{\cal B}_1P_0) P_0 P_3 P_0 \\ \notag &\quad - ({\cal B}_2 P_0 + {\cal B}_1P_1 + {\cal B}_0P_2)P_0P_2P_0 + \frac34 {\cal B}_0P_0 P_2 P_0 P_2 P_0 + \frac14 P_2^* P_0^* {\cal B}_0 P_0 P_2 P_0 \Big) P_0 \big] \notag \\ \notag
& = \mathbf{Sym} \big[ P_0^* \Big( {\cal B}_1 (\uno -P_0) P_3 + {\cal B}_2 (\uno-P_0) P_2 + {\cal B}_3 P_1 + {\cal B}_4 - {\cal B}_1P_1 P_0 P_2 P_0 \\ &\quad -\frac14 {\cal B}_0(P_0P_2P_0)^2+\frac14 P_2^* {\cal B}_0P_0P_2P_0 \Big) P_0 \big]\,  , \label{4orderaux}
\end{align}
where to pass from the first to the second line we used $P_0^*{\cal B}_0 P_4  =P_0^*{\cal B}_0 P_0 P_4$ (by  \eqref{B.id}) and 
$P_0^* {\cal B}_0 P_1 P_0 P_3 P_0 = 0$ (by  \eqref{B.id} and \eqref{P.id}).
We sum up the last two terms in \eqref{4orderaux} into $\mathbf{Sym}[ P_0^* \mathfrak{N} P_0 P_2  P_0]$ where
$ \mathfrak{N} $ is in \eqref{deffrakN}.
We observe that, in view of \eqref{B.id}-\eqref{lastid}, we have, for any $ f_k^\sigma, f_{k'}^{\sigma'}\in\{f_1^+,f_1^-,f_0^- \} $, that \eqref{kN} holds. 
Thus we obtain \eqref{ordine4}.
In conclusion, we have proved formula \eqref{B.Tay.ex}.
\end{proof}

\paragraph{Action of the jets of  $\mathfrak{B}_{\mu,\e} $ on the kernel vectors.}
We now collect  how the operators $\kB_{i,j}$ (cfr. \eqref{notazione})  acts on the vectors $f_1^+, f_1^-, f_0^-$.

\begin{lem}
The first jets of the operator $\mathfrak{B}_{\mu,\e} $ in \eqref{Bgotico} act, for $f_k^\sigma \in \{f_1^+,f_1^-,f_0^- \}$, as
\begin{subequations}\label{Bgoticoexpbis}
\begin{align}\label{ordini02:20}
&\kB_{0,2}f_k^\sigma =P_0^* \big( \cB_{0,2}+\cB_{0,1}P_{0,1} \big) f_k^\sigma  \, ,\quad  
\kB_{2,0}f_k^\sigma  = P_0^* \big( \cB_{2,0}+\cB_{1,0}P_{1,0} \big) f_k^\sigma  \, , \\
\label{ordine11}
&\textup{$\kB$}_{1,1} f_k^\sigma =P_0^* \big( \cB_{1,1}+\cB_{1,0}P_{0,1} +\cB_{0,1}P_{1,0} +   \frac12 \Pi_0^+ P_{1,1} \big) f_k^\sigma\, , \\
\label{ordine0330}
&\begin{aligned} &\textup{$\kB$}_{0,3}f_k^\sigma =P_0^*\big( \cB_{0,3} + \cB_{0,2} P_{0,1} + \cB_{0,1} P_{0,2} - \textup{\textbf{Sym}}[\cB_{0,1}P_0P_{0,2}]\big)f_k^\sigma \, , \\
 &\textup{$\kB$}_{3,0}f_k^\sigma = P_0^*\big(\cB_{3,0} + \cB_{2,0} P_{1,0} + \cB_{1,0} P_{2,0} - \textup{\textbf{Sym}}[\cB_{1,0}P_0P_{2,0}]\big) f_k^\sigma \, ,
 \end{aligned} \\
 \label{ordine12}
&\textup{$\kB$}_{1,2} f_k^\sigma =P_0^* \big( \cB_{1,2} + \cB_{1,1} P_{0,1}+ \cB_{0,2} P_{1,0}  + \cB_{1,0} P_{0,2} + \cB_{0,1} P_{1,1}  + \frac12 \Pi_0^+ P_{1,2}  \\  \notag
&\qquad \qquad \qquad   - \textup{\textbf{Sym}}[\cB_{1,0}P_0P_{0,2}+\cB_{0,1} P_0 P_{1,1}] \big)f_k^\sigma \, , \\
\label{ordine21}
&\textup{$\kB$}_{2,1} f_k^\sigma = P_0^* \big( \cB_{2,1} + \cB_{1,1} P_{1,0}+ \cB_{2,0} P_{0,1}  + \cB_{0,1} P_{2,0} + \cB_{1,0} P_{1,1}  + \frac12 \Pi_0^+ P_{2,1}  \\  \notag
&\qquad \qquad \qquad   - \textup{\textbf{Sym}}[\cB_{0,1}P_0P_{2,0}+\cB_{1,0} P_0 P_{1,1}] \big)f_k^\sigma \, , \\ \label{ordine04}
&\textup{$\kB$}_{0,4} f_k^\sigma =P_0^* \big( {\cal B}_{0,4}+ {\cal B}_{0,3} P_{0,1} + {\cal B}_{0,2}P_{0,2}  + {\cal B}_{0,1} P_{0,3} \\  \notag
&\qquad \qquad \qquad   - \textup{\textbf{Sym}}[{\cal B}_{0,2}P_0P_{0,2}  + {\cal B}_{0,1}P_0P_{0,3}  + {\cal B}_{0,1}P_{0,1} P_0 P_{0,2} 
-\mathfrak{N}_{0,2} P_0 P_{0,2} ] \big) f_k^\sigma 
 \, ,\\ \notag 
&\textup{$\kB$}_{2,2} f_k^\sigma =P_0^* \big( {\cal B}_{2,2}+ {\cal B}_{1,2} P_{1,0}+ {\cal B}_{2,1} P_{0,1} + {\cal B}_{0,2}P_{2,0} + {\cal B}_{1,1}P_{1,1}+ {\cal B}_{2,0} P_{0,2}   + {\cal B}_{0,1} P_{2,1}  + {\cal B}_{1,0} P_{1,2} \\ \notag 
&\qquad  + \frac12 \Pi_0^+ P_{2,2}   - \textup{\textbf{Sym}}[{\cal B}_{0,2}P_0P_{2,0}+{\cal B}_{1,1}P_0P_{1,1}  +{\cal B}_{2,0}P_0P_{0,2}  + {\cal B}_{0,1}P_0P_{2,1} + {\cal B}_{1,0}P_0P_{1,2} \\ \label{ordine22}
 &\qquad \qquad \qquad + {\cal B}_{0,1}P_{0,1} P_0 P_{2,0}+ {\cal B}_{1,0}P_{0,1} P_0 P_{1,1}+ {\cal B}_{0,1}P_{1,0} P_0 P_{1,1}+ {\cal B}_{1,0}P_{1,0} P_0 P_{0,2} 
\\  \notag
&  \qquad \qquad \qquad - \mathfrak{N}_{2,0} P_0 P_{0,2} - \mathfrak{N}_{0,2} P_0 P_{2,0} - \mathfrak{N}_{1,1} P_0 P_{1,1}
 ] \big) f_k^\sigma  \, ,\\ \notag
&\textup{$\kB$}_{1,3} f_k^\sigma =P_0^*\big( {\cal B}_{1,3}+ {\cal B}_{0,3} P_{1,0}+ {\cal B}_{1,2} P_{0,1}  + {\cal B}_{0,2}P_{1,1}+ {\cal B}_{1,1}P_{0,2}  + {\cal B}_{1,0} P_{0,3}+ {\cal B}_{0,1} P_{1,2} \\  \notag
&\qquad \qquad  + \frac12 \Pi_0^+ P_{1,3}   - \textup{\textbf{Sym}}[{\cal B}_{0,2}P_0P_{1,1} +{\cal B}_{1,1}P_0P_{0,2}  + {\cal B}_{1,0}P_0P_{0,3} + {\cal B}_{0,1}P_0P_{1,2} \\  \label{ordine13}
&\qquad \qquad \qquad  + {\cal B}_{1,0}P_{0,1} P_0 P_{0,2}+ {\cal B}_{0,1}P_{1,0} P_0 P_{0,2}+ {\cal B}_{0,1}P_{0,1} P_0 P_{1,1}    - \mathfrak{N}_{1,1} P_0 P_{0,2} -  \mathfrak{N}_{0,2} P_0 P_{1,1}
] \big) f_k^\sigma \, ,
\end{align}
\end{subequations}
with ${\cal B}_j$, $j=0,\dots,4$, in \eqref{Bsani} and $P_j$, $j=0,\dots,3$, in \eqref{Psani}. 
\end{lem}

The proof of \eqref{Bgoticoexpbis} relies on formulas \eqref{ordini012}--\eqref{ordine4} and Lemmata \ref{lem:05051}, \ref{lem:05052} below.
\begin{lem}\label{lem:05051}
Let $f_k^\sigma \in \{ f_1^+, f_1^-,f_0^-\} $. For any $j \in \bN$ we have
\begin{equation}\label{perplimente} {\footnotesize \begin{matrix}
P_0^* \textup{\textbf{Sym}}[\cB_j+ \cB_{j-1} P_1 + \dots + \cB_1 P_{j-1}] P_0 f_k^\sigma = P_0^* \big( \cB_j+ \cB_{j-1} P_1 + \dots + \cB_1 P_{j-1} + \frac12 \Pi_0^+ P_j \big) P_0 f_k^\sigma\end{matrix} } \, ,
\end{equation}
where $\Pi_0^+ $ is the orthogonal projector on $f_0^+ $.
\end{lem}
\begin{proof}
By identity \eqref{PB} the operator $\cB_{\mu,\e} P_{\mu,\e}$ is, like $\cB_{\mu,\e}$, self-adjoint, hence its $j$-th jet fulfills
\begin{equation}\label{Symrotto}
\textup{\textbf{Sym}}[\cB_jP_0+  \dots + \cB_1 P_{j-1} ]  =\cB_jP_0 + \dots + \cB_{1} P_{j-1} + \cB_0 P_j -  \textup{\textbf{Sym}}[ \cB_0 P_j ]\, . 
\end{equation}
We claim that, for $f_k^\sigma\in \{ f_1^+, f_1^-,f_0^-\}$ we have
\begin{equation}\label{claim_max}
P_0^* ( \cB_0 P_j - \mathbf{Sym}[\cB_0 P_j] )P_0 f_k^\sigma = \frac12 P_0^* \Pi_0^+ P_j P_0  f_k^\sigma\, ,
\end{equation}
which, together with \eqref{Symrotto}, proves \eqref{perplimente}. Claim \eqref{claim_max} follows, by observing that $f_k^\sigma$ fulfills
 $\cB_0f_k^\sigma = 0 $ and $\Pi^\angle P_0 f_k^\sigma=f_k^\sigma$ (cfr. Lemma \ref{lemmaboh}), then
\begin{equation}\label{SymB0Pj}
P_0^* \textup{\textbf{Sym}}[\cB_0 P_{j}] P_0 f_k^\sigma =\frac12 P_0^* \cB_0 P_{j} f_k^\sigma +\frac12 P_0^* P_{j}^* \cB_0   f_k^\sigma =  \frac12 P_0^* \cB_0 P_{j}\Pi^\angle P_0  f_k^\sigma \stackrel{\eqref{lastid}}{=} \frac12 \Pi_0^+ P_j f_k^\sigma  \, .
\end{equation}
 Using again that $P_0^* \cB_0 P_{j} f_k^\sigma = \Pi_0^+ P_j f_k^\sigma $ we obtain \eqref{claim_max}. 
\end{proof}

\begin{lem}\label{lem:05052}
 For any  $f \in \{f_1^+, f_{1}^-, f_0^- \}$ and $j \in \bN$  we have $\Pi_0^+ P_{0,j} f = \Pi_0^+ P_{j,0} f = 0 $.
\end{lem}

\begin{proof}
We have that $ \Pi_0^+ P_{0,j} f = 0 $ if and only if  $ (  P_{0,j}  f, f_0^+  ) = 0 $.
By  \cite[formula (4.8)]{BMV3} we have that 
$ P_{0,\e} f_0^- = f_0^- $ for any $ \epsilon $ and 
we have the chain of identities 
$$
\big (  P_{0,\e}  f, f_0^+ \big )
= - \big ( \cJ  P_{0,\e}  f,  \cJ f_0^+ \big ) \stackrel{\eqref{propPU},\cJ f_0^+ = - f_0^- } = 
 \big ( P_{0,\e}^* \cJ    f,   f_0^- \big )  = 
 \big (  \cJ    f,  P_{0,\e} f_0^- \big ) = 
  \big (  \cJ    f,   f_0^- \big ) = 0
$$
for any  $f \in \{f_1^+, f_{1}^-, f_0^- \}$,  deducing,
in particular, that  $ (  P_{0,j}  f, f_0^+  ) = 0 $.
The proof that $\Pi_0^+ P_{j,0} f = 0 $ is obtained similarly, exploiting   that 
$P_{\mu,0} f_0^- = f_0^-$ as proved in \cite[Lemma A.5]{BMV3}.
\end{proof}

In virtue of \eqref{notazione}, \eqref{cB evodd} and \eqref{P evodd} and in view of \eqref{upto4exp}-\eqref{deffrakN} we obtain
\begin{equation}\label{kB evodd}
\kB_{i,j}^{[\mathtt{ ev}]} = \begin{cases} \kB_{i,j} &\textup{if }j\textup{ is even}\, , \\ 
0 &\textup{if }j\textup{ is odd}\, , 
\end{cases} \qquad \kB_{i,j}^{[\mathtt{odd}]} = \begin{cases} 0 &\textup{if }j\textup{ is even}\, , \\ 
\kB_{i,j}  &\textup{if }j\textup{ is odd}\, . 
\end{cases} 
\end{equation}

\subsection{Proof of Proposition \ref{BexpG}}\label{secProof31}

Proposition \ref{BexpG} is a direct consequence of  the next proposition.

\begin{prop}\label{Bthirdorderexp}
The $2\times 2$ matrices $\Eta := \Eta  (\mu,\e) $, $\Gamma := \Gamma (\mu,\e) $, 
$\Phi := \Phi (\mu,\e) $ in \eqref{BinG1old}-\eqref{BinG3old} admit the  expansions
\begin{subequations}
\begin{align}\label{BinG1new}
& \Eta := 
\begin{pmatrix} 
  \wt\eta_{11}\e^3 + \eta_{11} \e^4 +  r_1( \e^5,\mu\e^3,\mu^2\e,\mu^3)
  & 
  \im  \left( \eta_{12}\mu\e^2  + r_2(\mu\e^3,\mu^2\e,\mu^3)\right)\\
- \im \left( \eta_{12}\mu\e^2 +  r_2(\mu\e^3,\mu^2\e,\mu^3)\right)  & \widetilde \eta_{22} \mu^2 \e + r_5(\mu^2\e^2,\mu^3)
 \end{pmatrix} \, , \\ 
&   \label{BinG2new} \Gamma := 
\begin{pmatrix} 
\gamma_{11} \e^2 + r_8(\e^3,\mu\e^2,\mu^2\e)
 &  -\im \gamma_{12} \mu\e^2- \im r_9(\mu\e^3,\mu^2\e) 
 \\
 \im \gamma_{12} \mu\e^2  +  \im  r_9(\mu\e^3, \mu^2\e )  
 &
 \widetilde \gamma_{22} \mu^2 \e + 
  \gamma_{22}\mu^2 \e^2 +  r_{10}(\mu^2\e^3,\mu^3\e)
 \end{pmatrix}\, , \\
 & \label{BinG3new}
 \Phi :=  \begingroup 
\setlength\arraycolsep{-13pt}
\begin{pmatrix} 
\!\!\!\!\!\!\!\!\!\!\!\!\!\!\!\! \phi_{11} \e^3  +  r_3(\e^4,\mu\e^2,\mu^2\e ) 
& 
 \im \widetilde \phi_{12} \mu \e^2 + \im \phi_{12} \mu \e^3 + 
\im \wt  \psi_{12} \mu^2 \e 
+\im  r_4(\mu\e^4, \mu^2 \e^2,\mu^3\e )  \\
 \im \phi_{21}\mu\e  + \im \wt \phi_{21} \mu \e^2 + \im   r_6(\mu\e^{3}, \mu^2\e)  
   & 
 \phi_{22}\mu^2\e + 
\wt \phi_{22}   \mu^2 \e^2 + 
    r_7(\mu^2\e^3, \mu^3\e) 
 \end{pmatrix} \endgroup\, ,
  \end{align}
  \end{subequations}
 where 
 \begin{subequations}\label{coeff2comp}
 \begin{equation}\label{zerocoeff}
 \begin{aligned}
& \widetilde \eta_{11}:= \big(\textup{$\mathfrak{B}_{0,3}$} f_1^+,f_1^+ \big) = 0 \, , \qquad
 \widetilde \eta_{22}:=  \big(\textup{$\mathfrak{B}_{2,1}$} f_1^-, f_1^- \big)  = 0  \, , 
 \qquad
  \widetilde \gamma_{22} :=    \big(\textup{$\mathfrak{B}_{2,1}$} f_0^-, f_0^- \big) = 0 \, , \\
 & \im  \widetilde \phi_{12}:= \big(\textup{$\mathfrak{B}_{1,2}$} f_0^-,f_1^+ \big)  = 0 \, , \ \quad
  \im \widetilde \phi_{21}:=\big(\textup{$\mathfrak{B}_{1,2}$} f_0^+,f_1^- \big)  = -\big(\textup{$\mathfrak{B}_{1,2}$} f_1^-,f_0^+ \big) = 0 \, , \\
& 
  \wt \phi_{22} := \big(\textup{$\mathfrak{B}_{2,2}$} f_0^-,f_1^- \big) = 0 \, ,  \qquad   \im \wt\psi_{12} :=    \big(\textup{$\mathfrak{B}_{2,1}$} f_0^-, f_1^+ \big)= 0\, ,
\end{aligned}
\end{equation}
whereas the coefficients
\begin{equation}\label{nzerocoeff}
\begin{aligned}
& \eta_{11}:= \big(\textup{$\mathfrak{B}_{0,4}$} f_1^+,f_1^+ \big) \, , \qquad \ 
\im  \eta_{12}:= \big(\textup{$\mathfrak{B}_{1,2}$} f_1^-,f_1^+ \big) \, , \\
 & \gamma_{11}:=  \big(\textup{$\mathfrak{B}_{0,2}$} f_0^+,f_0^+ \big) \, , 
 \qquad 
 \im \gamma_{12}:=  \big(\textup{$\mathfrak{B}_{1,2}$} f_0^-,f_0^+ \big) \, , 
 \qquad
 \gamma_{22}:=  \big(\textup{$\mathfrak{B}_{2,2}$} f_0^-,f_0^- \big) \, , 
 \\
 & \phi_{11}:=   \big(\textup{$\mathfrak{B}_{0,3}$} f_0^+,f_1^+ \big) = \big(\textup{$\mathfrak{B}_{0,3}$} f_1^+,f_0^+ \big) \, , \quad 
 \im \phi_{12}:=  \big(\textup{$\mathfrak{B}_{1,3}$} f_0^-,f_1^+ \big) \, , \\
& \im  \phi_{21}:=  \big(\textup{$\mathfrak{B}_{1,1}$} f_0^+,f_1^- \big)= - \big(\textup{$\mathfrak{B}_{1,1}$} f_1^-,f_0^+ \big) \, , \quad
 \phi_{22} =   \big(\textup{$\mathfrak{B}_{2,1}$} f_0^-,f_1^- \big)\, , 
 \end{aligned}
 \end{equation}
 \end{subequations}
are given in \eqref{te11}-\eqref{phi2s}. 
\end{prop}
The rest of the section is devoted to the proof of  Proposition \ref{Bthirdorderexp}.

\begin{lem}\label{3ordcanc}
 The coefficients $\widetilde \eta_{11},  \widetilde \eta_{22}, 
  \widetilde \gamma_{22},  \widetilde \phi_{12},   \widetilde \phi_{21}, \widetilde \phi_{22}$ in \eqref{zerocoeff} vanish.
\end{lem}

\begin{proof}
The first six  
coefficients 
in  \eqref{zerocoeff} are (use also the self-adjointness of the jets of $\kB_{\mu,\e}$)
\begin{equation*}
\big(\textup{$\mathfrak{B}_{0,3}^{[\mathtt{ ev}]}$} f_1^+,f_1^+ \big) \, ,\  \big(\textup{$\mathfrak{B}_{2,1}^{[\mathtt{ ev}]}$} f_1^-,f_1^- \big) \, ,\ 
\big(\textup{$\mathfrak{B}_{2,1}^{[\mathtt{ ev}]}$} f_0^-,f_0^- \big) \, , \
\big(\textup{$\mathfrak{B}_{1,2}^{[\mathtt{odd}]}$} f_0^-,f_1^+ \big) \, ,\ 
\big( f_0^+, \textup{$\mathfrak{B}_{1,2}^{[\mathtt{odd}]}$} f_1^- \big) \, , \ 
 \big(\textup{$\mathfrak{B}_{2,2}^{[\mathtt{odd}]}$} f_0^-,f_1^- \big) \, ,
\end{equation*}
which are zero because,
 by \eqref{kB evodd}, the operators 
 $\mathfrak{B}_{0,3}^{[\mathtt{ ev}]}\mathfrak{B}_{1,2}^{[\mathtt{odd}]}=\mathfrak{B}_{2,1}^{[\mathtt{ ev}]} =   \mathfrak{B}_{2,2}^{[\mathtt{odd}]}  = 0 $.
\end{proof}

For the computation of the other coefficients we use the following lemma. 
\begin{lem}
We have
\begin{align}\label{cBacts} 
&\cB_{0,1} f_1^+  = \vet{\frac12 (a_1^{[1]} \ch^{\frac12} - p_1^{[1]}\ch^{-\frac12} )\cos(2x)}{- p_1^{[1]}\ch^\frac12 \sin(2x)} + \kh^{[0]}(x) \, , \\ \notag
&\cB_{1,0} f_1^+  = \vet{0}{-\im \ch^{-\frac12}\big(\ch^2+\tth(1-\ch^4)\big)\cos(x)}\, ,\quad \cB_{1,1} f_1^+  = \frac{\im p_1^{[1]}}2\vet{-\ch^{-\frac12}\sin(2x)}{\ch^\frac12 \cos(2x)}+\kh^{[0]}(x) \, ,\\ \notag
&\cB_{0,2} f_1^+  = \vet{\big( (a_2^{[0]} +\frac12 a_2^{[2]} )\ch^{\frac12} -  (p_2^{[0]} +\frac12 p_2^{[2]} )\ch^{-\frac12}  \big) \cos(x)}{\big(\ttf_2 (1-\ch^4) \ch^{-\frac12}-  (p_2^{[0]} +\frac12 p_2^{[2]})\ch^{\frac12} \big) \sin(x)} + \vet{\frac12 \big(  a_2^{[2]} \ch^{\frac12} -  p_2^{[2]} \ch^{-\frac12}  \big) \cos(3x)}{-\frac32 p_2^{[2]}\ch^{\frac12}\sin(3x)} \, ,\\  \notag
&\cB_{0,3} f_1^+  = \vet{\frac12 ( a_3^{[1]}\ch^{\frac12} -p_3^{[1]}\ch^{-\frac12}) }{0} + \vet{\frac12 \big(a_3^{[1]}\ch^{\frac12} + a_3^{[3]}\ch^{\frac12} -p_3^{[1]}\ch^{-\frac12} - p_3^{[3]}\ch^{-\frac12}  \big)\cos(2x) }{-(p_3^{[1]}+p_3^{[3]})\ch^{\frac12}  \sin(2x) }+\kh^{[4]}(x)\, ,\\  \notag
&\cB_{0,4} f_1^+  = \vet{\big(\ch^\frac12 (a_4^{[0]}+\frac12 a_4^{[2]}) -\ch^{-\frac12} (p_4^{[0]}+\frac12 p_4^{[2]})\big)\cos(x)}{\big(\ch^{-\frac12}(1-\ch^4)(\ttf_4-\ttf_2^2\ch^2)-\ch^{\frac12}(p_4^{[0]}+\frac12p_4^{[2]} \big)\sin(x) }+\kh^{[3,5]}(x) \, ,\\ \notag
&\cB_{1,0} f_1^-  = \vet{0}{\im \ch^{-\frac12}\big(\ch^2+\tth(1-\ch^4)\big)\sin(x)}\, ,\quad \cB_{1,1} f_1^-  =-\frac{\im p_1^{[1]}}2 \vet{\ch^{-\frac12}}{0} +\kh^{[2]}(x) \, , \\ \notag
&\cB_{1,2} f_1^- = -\im \vet{\ch^{-\frac12}(p_2^{[0]}+\frac12 p_2^{[2]})\cos(x)}{\ch^{\frac12} (p_2^{[0]}-\frac12 p_2^{[2]})\sin(x)} + \kh^{[3]}(x) \, ,\\ \notag
&\cB_{0,1} f_0^+  = \vet{a_1^{[1]}\cos(x)}{-p_1^{[1]}\sin(x)}\, ,\quad 
\cB_{0,2} f_0^+  = \vet{a_2^{[0]}}{0}+ \vet{a_2^{[2]}\cos(2x)}{-2p_2^{[2]}\sin(2x)}\, ,\\ \notag
&\cB_{1,1} f_0^-  =-\im p_1^{[1]} \vet{\cos(x)}{0} \, ,\quad 
\cB_{1,2} f_0^-  = -\im p_2^{[0]}\vet{1}{0} - \im p_2^{[2]} \vet{\cos(2x)}{0} \, ,\\ \notag
&\cB_{1,3} f_0^-  = -\im p_3^{[1]}\vet{\cos(x)}{0}+\kh^{[3]}(x) \,, \quad
\cB_{2,2} f_0^-  = \vet{0}{\ttf_2}\, ,  
\end{align}
with $p_j^{[i]}$ and $a_j^{[i]}$, $j =1,\dots ,4$, $i=0,\dots,j$, in \eqref{apexp}, \eqref{expa04}
and $\ttf_2$ in \eqref{expfe} and
where $\kh^{[\kappa_1,\dots,\kappa_\ell]}(x)$ denotes a function supported on Fourier modes $\kappa_1,\dots,\kappa_\ell\in {\bN_0} $. 
\end{lem}

\begin{proof}
By \eqref{Bsani}-\eqref{Pi_odd} and \eqref{funperturbed}.
\end{proof}
We now compute the remaining coefficients in \eqref{coeff2comp}. 
\\[1mm] 
\noindent {\bf Computation of $\gamma_{11}$.} 
In view of \eqref{nzerocoeff} and \eqref{ordini012} we have  
\begin{equation}\label{expgamma11}
\gamma_{11}  = \underbrace{\big( {\cal B}_{0,2} f_0^+ , f_0^+ \big)}_{a_2^{[0]}
\ {\text{by} \ \eqref{B2sano}} } + \underbrace{
\frac12 \big( {\cal B}_{0,1} P_{0,1} f_0^+ , f_0^+ \big)   
+
\frac12\big( {\cal B}_{0,1}  f_0^+ , P_{0,1} f_0^+ \big)}_{\ku_{0,1} \big(   f_{-1}^+ , {\cal B}_{0,1} f_0^+ \big) \ {\text{by} \ \eqref{tuttederivate}} }
 \, .
\end{equation}
By \eqref{cBacts} 
and \eqref{fm1sigma} it results that  \eqref{expgamma11} is equal to 
\begin{equation*}
\gamma_{11} =  a_2^{[0]}+\frac12 \ku_{0,1} (a_1^{[1]}\ch^{\frac12}+p_1^{[1]}\ch^{-\frac12})\, , 
\end{equation*}
which in view of \eqref{apexp}, \eqref{u01}  gives the term in \eqref{gammas}. 
\\[1mm]
\noindent {\bf Computation of $ \phi_{21}$.} 
 In view of \eqref{nzerocoeff} and \eqref{ordine11} we have
\begin{equation*}
\begin{aligned}
\im \phi_{21} & =-\big(\mathfrak{B}_{1,1} f_1^- , f_0^+ \big) \\
&  =- \big( {\cal B}_{1,1} f_1^- , f_0^+ \big) - \underbrace{ \big( {\cal B}_{0,1} P_{1,0} f_1^- , f_0^+ \big)}_{\im \ku_{1,0} \big( f_{-1}^+ , {\cal B}_{0,1} f_0^+ \big) 
\ {\text{by} \ \eqref{tuttederivate}} } 
-\underbrace{ \big( {\cal B}_{1,0} P_{0,1} f_1^- , f_0^+ \big) 
}_{0 \ {\text{by} \ \eqref{tuttederivate}}}
-\underbrace{ 
\frac12 \big(\Pi_0^+ P_{1,1} f_1^- , f_0^+ \big)}_{0 \ {\text{by} \ \eqref{secondjetsPmix}}}
\, . 
\end{aligned}
\end{equation*}
By \eqref{cBacts}, \eqref{funperturbed} and \eqref{fm1sigma} it results
\begin{equation*}
\phi_{21} = \frac12 \big(\ch^{-\frac12} p_1^{[1]}-\ku_{1,0} \ch^{\frac12} a_1^{[1]}-\ku_{1,0} \ch^{-\frac12} p_1^{[1]}  \big)
\end{equation*}
which in view of \eqref{apexp}, \eqref{u01}   gives the term in \eqref{phi2s}.
\\[1mm]
\noindent {\bf Computation of $  \eta_{12}$.}  In view of \eqref{nzerocoeff} and \eqref{ordine12}, \eqref{tuttederivate}, Lemmata \ref{lem:secondjetsP} and \ref{lem:P03acts}  we have
\begin{align*}
\im \eta_{12}&=  \big( {\cal B}_{1,2} f_1^- , f_1^+ \big) + 
 \underbrace{\big( {\cal B}_{1,1} P_{0,1} f_1^- , f_1^+\big)}_{\big(  {\scriptsize \vet{-\ka_{0,1}\sin(2x)}{\kb_{0,1}\cos(2x)}} ,{\cal B}_{1,1} f_1^+\big) \ \textup{by } \eqref{tuttederivate} }
 + 
 \underbrace{\big( {\cal B}_{0,2} P_{1,0}  f_1^- , f_1^+ \big)}_{\im \ku_{1,0} \big(  f_{-1}^+, {\cal B}_{0,2} f_1^+ \big) \textup{ by } \eqref{tuttederivate} } \\ 
 \notag &+
\underbrace{\big( {\cal B}_{1,0}  P_{0,2}  f_1^- , f_1^+ \big)}_{\underline{\kn_{0,2}(\cB_{1,0} f_1^-, f_1^+)} + \ku_{0,2}^-( f_{-1}^-,\cB_{1,0} f_1^+) \textup{ by } \eqref{secondjetsPeps} }
+\underbrace{\big( {\cal B}_{0,1}  P_{1,1}  f_1^- , f_1^+ \big)}_{\im \big(\vet{\ka_{1,1}\cos(2x)}{\kb_{1,1}\sin(2x)}, {\cal B}_{0,1}  f_1^+ \big)\textup{ by } \eqref{secondjetsPmix}} 
\\
\notag
&+\underbrace{\frac12 \big( \Pi_0^+ P_{1,2} f_1^-, f_1^+ \big)}_{=0\textup{ since }\Pi_0^+ f_1^+ =0} \underbrace{- \frac12 \big( {\cal B}_{1,0} P_0 P_{0,2}  f_1^- , f_1^+ \big)}_{\underline{-\frac12 \kn_{0,2}(\cB_{1,0} f_1^-, f_1^+)}\textup{ by } \eqref{secondjetsPeps}} 
 \underbrace{- \frac12 \big( {\cal B}_{0,1} P_0  P_{1,1}  f_1^- , f_1^+\big)}_{=0\textup{ by } \eqref{secondjetsPmix}}\\
 \notag
 & \underbrace{- \frac12 \big( {\cal B}_{1,0}  f_1^- ,P_0 P_{0,2} f_1^+ \big)}_{\underline{-\frac12 \kn_{0,2}(\cB_{1,0} f_1^-, f_1^+)} \textup{ by } \eqref{secondjetsPbis}}\ \
 \underbrace{ - \frac12 \big( {\cal B}_{0,1} f_1^- , P_0  P_{1,1} f_1^+ \big) }_{\stackrel{\eqref{secondjetsPmix}}{=}+\im \frac12 \widetilde \km_{1,1}( f_1^-,\cB_{0,1} f_0^-) =0} \, ,
\end{align*}
where the three underlined terms cancel out. 
Hence, by \eqref{cBacts}, \eqref{funperturbed} and \eqref{fm1sigma},
\begin{align*}
\eta_{12} = &- p_2^{[0]} - \frac14 p_1^{[1]} (\kb_{0,1} \ch^{\frac12} + \ka_{0,1} \ch^{-\frac12}) +\ku_{1,0} \big( \frac12 \ch a_2^{[0]}+ \frac14 \ch a_2^{[2]} - \frac12 \ch^{-1} \tf_2 (1-\ch^4) \big) \\ \notag
&+\frac12  \ku_{0,2}^- (\ch^2+\tth(1-\ch^4))\ch^{-1} + \frac14 \ch^{\frac12}  a_1^{[1]}\ka_{1,1} - \frac12 \ch^{\frac12}  p_1^{[1]} \kb_{1,1} -\frac14\ch^{-\frac12} p_1^{[1]} \ka_{1,1} 
\end{align*}
which in view of \eqref{apexp}, \eqref{u01} and  \eqref{u02-}
 gives the term in \eqref{eta12}.
\\[1mm]
\noindent {\bf Computation of $ \gamma_{12}$.}
By \eqref{nzerocoeff} and \eqref{ordine12}, \eqref{tuttederivate}, Lemma \ref{lem:secondjetsP} and \ref{lem:P03acts} and since $\cB_{0,1} f_0^- = 0$ and $\cB_{1,0} f_0^- = 0$  we have 
\begin{align*}
\im \gamma_{12}&= \big( {\cal B}_{1,2}f_0^- , f_0^+ \big) 
\underbrace{+  \big( {\cal B}_{1,1} P_{0,1} f_0^- , f_0^+\big)}_{= 0}
\underbrace{+\big( {\cal B}_{0,2} P_{1,0}  f_0^- , f_0^+ \big)}_{=0} \\ \notag &\underbrace{+\big( {\cal B}_{1,0}  P_{0,2}  f_0^- , f_0^+ \big)}_{=0}
\underbrace{+\big( {\cal B}_{0,1}  P_{1,1}  f_0^- , f_0^+ \big) }_{-\im \frac12 \ch^{-3/2} ( f_{-1}^+, \cB_{0,1} f_0^+)}
\underbrace{+\frac12 \big( \Pi_0^+ P_{1,2} f_0^-, f_0^+ \big)}_{=0}\\
\notag
& 
\underbrace{- \frac12 \big( {\cal B}_{1,0} P_0 P_{0,2}  f_0^- , f_0^+ \big)}_{=0}
\underbrace{ - \frac12 \big( {\cal B}_{0,1} P_0  P_{1,1}  f_0^- , f_0^+\big)}_{=0} \\
 & \underbrace{- \frac12 \big( {\cal B}_{1,0}  f_0^- ,P_0 P_{0,2} f_0^+ \big)}_{=0} \underbrace{- \frac12 \big( {\cal B}_{0,1} f_0^- , P_0  P_{1,1} f_0^+ \big)}_{=0}  \, .
\end{align*}
So, by \eqref{cBacts}, \eqref{funperturbed} and \eqref{fm1sigma},
\begin{equation*}
\gamma_{12} = -p_2^{[0]} - \frac14 \ch^{-3/2} (a_1^{[1]}\ch^{\frac12} + p_1^{[1]}\ch^{-\frac12} ) 
\end{equation*}
which in view of  \eqref{apexp} gives the term \eqref{gammas}. 
\\[1mm]
\noindent {\bf Computation of $\phi_{11}$.}
By \eqref{nzerocoeff} and \eqref{ordine0330}, \eqref{tuttederivate}, Lemma \ref{lem:secondjetsP} and \ref{lem:P03acts}  we have
\begin{align*}
\phi_{11}&
= \big( {\cal B}_{0,3}f_1^+ , f_0^+ \big) +  \underbrace{\big( {\cal B}_{0,2} P_{0,1} f_1^+ , f_0^+\big)}_{\big({\scriptsize \vet{\ka_{0,1}\cos(2x)}{\kb_{0,1}\sin(2x)}}, \cB_{0,2} f_0^+\big)}+\underbrace{\big( {\cal B}_{0,1}  P_{0,2}  f_1^+ , f_0^+ \big) }_{
\kn_{0,2}(f_1^+,\cB_{0,1} f_0^+) + 
\ku_{0,2}^+(f_{-1}^+, \cB_{0,1} f_0^+) 
}\\ \notag  
&
\underbrace{- \frac12 \big( {\cal B}_{0,1} P_0 P_{0,2}  f_1^+ , f_0^+ \big)}_{ -\frac12 \kn_{0,2}(f_1^+, \cB_{0,1} f_0^+)
}
\underbrace{- \frac12 \big( {\cal B}_{0,1}  f_1^+ ,P_0 P_{0,2} f_0^+ \big) }_{\stackrel{\eqref{secondjetsPbis}}{=}0}   \, .
\end{align*}
Thus, by \eqref{cBacts}, \eqref{funperturbed} and \eqref{fm1sigma},
$$
\begin{aligned}
\phi_{11} & = \frac12 a_3^{[1]} \ch^\frac12 -\frac12 p_3^{[1]}\ch^{-\frac12}  + \frac12 \ka_{0,1} a_2^{[2]} -   \kb_{0,1} p_2^{[2]}   \\
& \ \ + \frac14 \kn_{0,2} (a_1^{[1]}\ch^\frac12 -p_1^{[1]}\ch^{-\frac12} ) + \frac12 \ku_{0,2}^+ (a_1^{[1]}\ch^\frac12 +p_1^{[1]}\ch^{-\frac12} ) 	
\end{aligned}
$$
which in view of 
\eqref{apexp}, \eqref{u01}, \eqref{u02-},   gives the term \eqref{phi11}. 
\\[1mm]
\noindent {\bf Computation of $\phi_{22}$.}
By \eqref{nzerocoeff}  and \eqref{ordine21}, \eqref{tuttederivate}, Lemma \ref{lem:secondjetsP} and \ref{lem:P03acts}  and since $\cB_{2,1}=0$ and  ${\cal B}_{0,1}  f_0^-=0$, ${\cal B}_{1,0}  f_0^-=0$ we have
\begin{align*}
\phi_{22} &=
\underbrace{ \big( {\cal B}_{2,1}f_0^- , f_1^-\big)}_{0}
 +  \underbrace{\big( {\cal B}_{1,1} P_{1,0} f_0^- , f_1^-\big)}_{0}
 \underbrace{+\big( {\cal B}_{2,0} P_{0,1}  f_0^- , f_1^- \big)}_{0}
  \\ \notag &
  \underbrace{+\big( {\cal B}_{0,1}  P_{2,0}  f_0^- , f_1^- \big)}_{0}
 \underbrace{ +\big( {\cal B}_{1,0}  P_{1,1}  f_0^- , f_1^- \big) }_{-\frac{\im}{2}\ch^{-3/2} \big(f_{-1}^+,\cB_{1,0}  f_1^-\big) }+
\underbrace{\frac12 \big( \Pi_0^+ P_{2,1} f_0^-, f_1^-)}_{\frac12 \big( \tilde \kn_{2,1} \Pi_0^+ f_1^- +\tilde \ku_{2,1} \Pi_0^+ f_{-1}^-, f_1^-)=0} \\
& \underbrace{- \frac12 \big( {\cal B}_{0,1} P_0 P_{2,0}  f_0^- , f_1^- \big)}_{=0} 
 -\underbrace{\frac12 \big( {\cal B}_{1,0} P_0  P_{1,1}  f_0^- , f_1^-\big)}_{=0}\\
 & - \underbrace{ \frac12 \big( {\cal B}_{0,1}  f_0^- ,P_0 P_{2,0} f_1^- \big)}_{=0} - \underbrace{ \frac12 \big( {\cal B}_{1,0} f_0^- , P_0  P_{1,1} f_1^- \big)}_{=0}  \, .
\end{align*}
So, by \eqref{cBacts} and \eqref{fm1sigma}, 
\begin{equation*}
\phi_{22} = \frac14 \ch^{-\frac52} (\ch^2+\tth(1-\ch^4))
\end{equation*}
which  is the term \eqref{phi2s}. 
\\[1mm]
\noindent {\bf Computation of $ \widetilde \psi_{12}$.}
By \eqref{zerocoeff}, \eqref{ordine21} and since $\cB_{2,1} =0$, {$ P_{1,0} f_0^- = 0 $,
$ P_{0,1} f_0^- = 0 $} 
by \eqref{tuttederivate}, Lemmata \ref{lem:secondjetsP} and \ref{lem:P03acts} 
 this term is given by
\begin{align*}
\im \widetilde \psi_{12} & =\underbrace{ \big( {\cal B}_{2,1}f_0^- , f_1^+ \big)}_{=0} + \underbrace{ \big( {\cal B}_{1,1} P_{1,0} f_0^- , f_1^+\big)}_{=0}+\underbrace{\big( {\cal B}_{2,0} P_{0,1}  f_0^- , f_1^+ \big)}_{=0} \\ \notag &+\underbrace{\big( {\cal B}_{0,1}  P_{2,0}  f_0^- , f_1^+ \big)}_{=0 \ {\text{by} \, \eqref{secondjetsPmu}} }
+\underbrace{\big( {\cal B}_{1,0}  P_{1,1}  f_0^- , f_1^+ \big) }_{-\frac\im2 \ch^{-\frac32}  \big(f_{-1}^+,{\cal B}_{1,0} f_1^+ \big) 
\ \text{by} \eqref{secondjetsPmix}}
+\underbrace{\frac12 \big(\Pi_0^+ P_{2,1} f_0^-, f_1^+ \big)}_{
=0 \ \text{as } \Pi_0^+    f_1^+ = 0
}\\
&- \underbrace{\frac12 \big( {\cal B}_{0,1} P_0 P_{2,0}  f_0^- , f_1^+ \big) }_{=0
\ {\text{by} \, \eqref{secondjetsPmu}} }
 - \frac12 \underbrace{\big( {\cal B}_{1,0} P_0  P_{1,1}  f_0^- , f_1^+\big)}_{=0
 \ \text{by}  \, \eqref{secondjetsPmix} } \\
 & - \frac12 \underbrace{\big( {\cal B}_{0,1}  f_0^- ,P_0 P_{2,0} f_1^+ \big)}_{=0
  \ \text{since} \ {\cal B}_{0,1}  f_0^- = 0  } - \frac12 \underbrace{\big( {\cal B}_{1,0} f_0^- , P_0  P_{1,1} f_1^+ \big) }_{=0  { \ \text{since} \ {\cal B}_{1,0}  f_0^- = 0 } } 
\end{align*}
and finally, by \eqref{fm1sigma} and \eqref{cBacts}, 
$$
\im \widetilde \psi_{12} =  -\tfrac{\im}{2 } \ch^{-3/2}
\big(f_{-1}^+,{\cal B}_{1,0} f_1^+ \big)  = 0 \, . 
$$ 
\noindent {\bf Computation of $\eta_{11}$.}
By \eqref{nzerocoeff}, \eqref{ordine04}, \eqref{tuttederivate}, Lemma \ref{lem:secondjetsP} and Lemma \ref{lem:P03acts} and \eqref{kN},  we have
\begin{align*}
\eta_{11} &=  \big( {\cal B}_{0,4} f_1^+,f_1^+\big)+
\underbrace{ \big({\cal B}_{0,3} P_{0,1} f_1^+,f_1^+\big) }_{ \big( {\scriptsize \vet{\ka_{0,1} \cos(2x)}{\kb_{0,1} \sin(2x)}}, \cB_{0,3} f_1^+ \big)}
+
\underbrace{\big( {\cal B}_{0,2}P_{0,2} f_1^+,f_1^+\big)}_{\big(\underline{\kn_{0,2} f_1^+} + \ku_{0,2}^+ f_{-1}^+ + {\scriptsize\vet{\ka_{0,2} \cos(3x)}{\kb_{0,2} \sin(3x)}}, \cB_{0,2} f_1^+ \big)} \notag 
\\ 
& +
\underbrace{ \big( {\cal B}_{0,1} P_{0,3}f_1^+,f_1^+\big)}_{\big( {\scriptsize\vet{\ka_{0,3} \cos(2x)}{\kb_{0,3} \sin(2x)}}, \cB_{0,1} f_1^+ \big)}  \underbrace{
 -  \frac12 \big({\cal B}_{0,2}P_0P_{0,2}  f_1^+,f_1^+\big)}_{\underline{-\frac12 \kn_{0,2} \big({\cal B}_{0,2}  f_1^+,f_1^+\big)}} \notag \\ 
& \underbrace{-
\frac12  \big({\cal B}_{0,2} f_1^+, P_0P_{0,2}  f_1^+\big)}_{\underline{-\frac12 \kn_{0,2} \big({\cal B}_{0,2}  f_1^+,f_1^+\big)}}
-\underbrace{\frac12  \big( {\cal B}_{0,1}P_0P_{0,3} f_1^+,f_1^+\big)}_{=0} \\
&-\frac12  \underbrace{\big({\cal B}_{0,1} f_1^+,P_0 P_{0,3}f_1^+\big)}_{=0}
\underbrace{- \frac12  \big( {\cal B}_{0,1}P_{0,1} P_0 P_{0,2} f_1^+,f_1^+\big)}_{-\frac12  \kn_{0,2} \big( {\scriptsize\vet{\ka_{0,1} \cos(2x)}{\kb_{0,1}\sin(2x)}}, {\cal B}_{0,1} f_1^+\big)}
\underbrace{- \frac12  \big( {\cal B}_{0,1} f_1^+,P_{0,1} P_0 P_{0,2} f_1^+\big)}_{-\frac12  \kn_{0,2} \big(  {\scriptsize\vet{\ka_{0,1} \cos(2x)}{\kb_{0,1}\sin(2x)}},{\cal B}_{0,1} f_1^+\big)}  \\
& +\frac12 \underbrace{\big(\mathfrak{N}_{0,2} P_0 P_{0,2} f_1^+, f_1^+ \big)}_{=0}
+ \underbrace{\frac12 \big( f_1^+, \mathfrak{N}_{0,2} P_0 P_{0,2} f_1^+ \big)}_{=0}
\, ,
\end{align*}
where the three underlined terms cancel out.
Thus, in view of \eqref{cBacts}, \eqref{funperturbed} and \eqref{fm1sigma}, we get
\begin{align}
\eta_{11} & = \ch \frac{a_4^{[0]}}{2} + 
\ch \frac{a_4^{[2]}}{4}
- p_4^{[0]} - \frac{p_4^{[2]}}{2} + \frac{1}{2\ch} (1-\ch^4)(\ttf_4-\ttf_2^2\ch^2) \\ \notag
&+ \frac14 \big( \ch^{\frac12} (a_3^{[1]} + a_3^{[3]})- \ch^{-\frac12} ( p_3^{[1]} + p_3^{[3]})  \big)\ka_{0,1}   
- \frac12 \ch^{\frac12} ( p_3^{[3]} + p_{3}^{[1]}  ) \kb_{0,1} \notag
\\ \notag
& +
\frac{1}{4}\big( a_{2}^{[2]} \ch^{\frac12} - p_{2}^{[2]} \ch^{-\frac12} \big) \ka_{0,2}
-
\frac34 \ch^{\frac12} p_{2}^{[2]} \kb_{0,2} \\ \notag
&+ \frac12 \ku_{0,2}^+ 
 \big( \ch a_2^{[0]} + \frac12 \ch a_2^{[2]} - \ch^{-1} \ttf_2 (1-\ch^4) \big) \notag \\ \notag
& + \frac14\ka_{0,3} \big(a_1^{[1]} \ch^{\frac12} - p_{1}^{[1]} \ch^{-\frac12} \big) 
-\frac{1}{2} \kb_{0,3} \ch^{\frac12} p_1^{[1]} \\ \notag
&- \frac14 \kn_{0,2} \ka_{0,1}(a_{1}^{[1]} \ch^{\frac12} - p_1^{[1]} \ch^{-\frac12} )  + \frac12 \kn_{0,2} \kb_{0,1} \ch^{\frac12} p_{1}^{[1]} 
\end{align}
which in view of  \eqref{apexp}, \eqref{aino4}, \eqref{u01}, \eqref{u02-}, 
\eqref{n03}  gives \eqref{te11}. 
\\[1mm]
\noindent {\bf Computation of $\gamma_{22}$.}
By \eqref{nzerocoeff}, \eqref{ordine22}, where we exploit that $\big(\mathbf{Sym}[A] f,f\big)= \mathfrak{Re} \big(Af,f \big) $,  \eqref{tuttederivate}, Lemma \ref{lem:secondjetsP}, Lemma \ref{lem:P03acts} and since $\cB_{0,1} f_0^-=0$ and $\cB_{1,0} f_0^-=0$ we have
\begin{align*}
\gamma_{22} & = 
\underbrace{\big( {\cal B}_{2,2} f_0^-, f_0^- \big)}_{\tf_2} 
+ \underbrace{\big( {\cal B}_{1,2} P_{1,0} f_0^-, f_0^- \big)}_{=0}
  +
  \underbrace{ \big( {\cal B}_{2,1} P_{0,1}f_0^-, f_0^- \big)}_{=0}
 + \underbrace{\big( {\cal B}_{0,2}P_{2,0}f_0^-, f_0^- \big)}_{=0}
 \\ 
 \notag
&  + 
\underbrace{\big( {\cal B}_{1,1}P_{1,1}f_0^-, f_0^- \big)}_{-\frac{\im}{2} \ch^{-3/2}\big(  f_{-1}^+,{\cal B}_{1,1} f_0^-\big) }
 + \underbrace{\big( {\cal B}_{2,0} P_{0,2} f_0^-, f_0^- \big)}_{=0}
  +
 \underbrace{\big( {\cal B}_{0,1} P_{2,1} f_0^-, f_0^- \big)}_{\big(  P_{2,1} f_0^-,{\cal B}_{0,1} f_0^- \big)=0}
 + \underbrace{\big( {\cal B}_{1,0} P_{1,2} f_0^-, f_0^- \big)}_{ = 0}  + \frac12 \underbrace{\big( \Pi_0^+ P_{2,2} f_0^-, f_0^-\big) }_{=0}
 \\ 
 \notag 
&  \underbrace{- \mathfrak{Re} \big( {\cal B}_{0,2}P_0P_{2,0}f_0^-, f_0^- \big)}_{= 0}
  \underbrace{-\mathfrak{Re} \big ({\cal B}_{1,1}P_0P_{1,1}f_0^-, f_0^- \big)}_{ = 0 }
  \underbrace{-\mathfrak{Re} \big ({\cal B}_{2,0}P_0P_{0,2}f_0^-, f_0^- \big)}_{=0}
 \underbrace{ -\mathfrak{Re}  \big ( {\cal B}_{0,1}P_0P_{2,1}f_0^-, f_0^- \big)}_{\big ( P_0P_{2,1}f_0^-, {\cal B}_{0,1} f_0^- \big)=0} \\ \notag
& \underbrace{-\mathfrak{Re}  \big ( {\cal B}_{1,0}P_0P_{1,2}f_0^-, f_0^- \big)}_{ = 0}
  \underbrace{-\mathfrak{Re} \big (  {\cal B}_{0,1}P_{0,1} P_0 P_{2,0}f_0^-, f_0^- \big) }_{=0} \\ \notag
 &  \underbrace{-\mathfrak{Re} \big (  {\cal B}_{1,0}P_{0,1} P_0 P_{1,1} f_0^-, f_0^- \big)}_{=0}  \underbrace{-\mathfrak{Re}  \big (  {\cal B}_{0,1}P_{1,0} P_0 P_{1,1}f_0^-, f_0^- \big) }_{=0}\underbrace{-\mathfrak{Re}  \big (  {\cal B}_{1,0}P_{1,0} P_0 P_{0,2} f_0^-, f_0^- \big) }_{=0}\\
 &
 \underbrace{+\mathfrak{Re} \big( \mathfrak{N}_{2,0} P_0 P_{0,2} f_0^-, f_0^- \big)}_{=0}  \underbrace{+\mathfrak{Re} \big( \mathfrak{N}_{0,2} P_0 P_{2,0}f_0^-, f_0^-)}_{=0}    \underbrace{+\mathfrak{Re}\big( \mathfrak{N}_{1,1} P_0 P_{1,1} f_0^-, f_0^- \big)}_{=0}\, . 
\end{align*}
Then by \eqref{B4sano}, \eqref{Pi_odd}, \eqref{ell22}, \eqref{fm1sigma} and \eqref{cBacts}  we get
\begin{equation*}
\gamma_{22}  = \tf_2 + \frac{p_1^{[1]}}{4\ch}
\end{equation*}
which, in view of \eqref{expfe}, \eqref{pino4} gives \eqref{gammas}. 
\\[1mm]
\noindent {\bf Computation of $\phi_{12}$.}
By \eqref{nzerocoeff}, \eqref{ordine13},  \eqref{tuttederivate}, Lemma \ref{lem:secondjetsP}, Lemma \ref{lem:P03acts}, \eqref{kN} and since $\cB_{0,2} f_0^-=0$, $\cB_{1,0} f_0^-=0$ and $\cB_{0,1} f_0^-=0$ we have
 \begin{align*} 
\im \phi_{12} & = 
\big( {\cal B}_{1,3} f_0^-, f_1^+ \big)
+ \underbrace{\big( {\cal B}_{0,3} P_{1,0}f_0^-, f_1^+ \big)}_{=0}
+ \underbrace{\big( {\cal B}_{1,2} P_{0,1}f_0^-, f_1^+ \big)}_{=0} \\ 
 \notag
&+ \underbrace{\big( {\cal B}_{0,2}P_{1,1}f_0^-, f_1^+ \big)}_{-\frac{\im}{2}\ch^{-3/2}  \big(  f_{-1}^+, {\cal B}_{0,2} f_1^+ \big)}
+ \underbrace{\big( {\cal B}_{1,1}P_{0,2} f_0^-, f_1^+ \big)}_{=0} 
+ \underbrace{\big({\cal B}_{1,0} P_{0,3}f_0^-, f_1^+ \big)}_{=0}
+\underbrace{\big( {\cal B}_{0,1} P_{1,2}f_0^-, f_1^+ \big)}_{\im  \big({\scriptsize \vet{\ka_{1,2}\cos(2x)}{\kb_{1,2}\sin(2x)}},  {\cal B}_{0,1} f_1^+ \big)} \notag  \\
& +\underbrace{\frac12 \big( \Pi_0^+ P_{1,3} f_0^- , f_1^+ \big)}_{=0} 
   - \frac12 \underbrace{\big( {\cal B}_{0,2}P_0P_{1,1}f_0^-, f_1^+ \big)}_{=0} 
- \frac12 
\underbrace{\big( {\cal B}_{0,2}f_0^-,P_0P_{1,1} f_1^+ \big)}_{=0} \\
& 
-\frac12
\underbrace{\big({\cal B}_{1,1} P_0P_{0,2} f_0^-, f_1^+ \big)}_{=0}
\underbrace{-\frac12\big({\cal B}_{1,1}f_0^-, P_0P_{0,2} f_1^+ \big)}_{-\frac12 \kn_{0,2} \big({\cal B}_{1,1}f_0^-, f_1^+ \big) } 
-\frac12  \underbrace{\big({\cal B}_{1,0}P_0P_{0,3}f_0^-, f_1^+ \big)}_{=0} \notag \\
& -\frac12 \underbrace{ \big({\cal B}_{1,0}f_0^-,  P_0P_{0,3} f_1^+ \big)}_{=0} -\frac12 \underbrace{\big( {\cal B}_{0,1}P_0P_{1,2} f_0^-, f_1^+ \big)}_{=0}
-\frac12 \underbrace{\big( {\cal B}_{0,1}f_0^-,P_0P_{1,2}  f_1^+ \big)}_{=0} \\  \notag 
& -\frac12\underbrace{\big( {\cal B}_{1,0}P_{0,1} P_0 P_{0,2}f_0^-, f_1^+ \big)}_{=0}
-\frac12\underbrace{\big( {\cal B}_{1,0}f_0^-, P_{0,1} P_0 P_{0,2}  f_1^+ \big)}_{=0}
-\frac12
\underbrace{\big( {\cal B}_{0,1}P_{1,0} P_0 P_{0,2}f_0^-, f_1^+ \big)}_{=0}\\ \notag
& -\frac12
\underbrace{\big( {\cal B}_{0,1}f_0^-, P_{1,0} P_0 P_{0,2} f_1^+ \big)}_{=0}
-\frac12\underbrace{\big( {\cal B}_{0,1}P_{0,1} P_0 P_{1,1}f_0^-, f_1^+ \big)}_{=0}
-\frac12
\underbrace{\big( {\cal B}_{0,1}f_0^-, P_{0,1} P_0 P_{1,1} f_1^+ \big)}_{=0} \\
&
+\underbrace{\frac12 \big( \mathfrak{N}_{1,1} P_0 P_{0,2} f_0^-, f_1^+ \big)}_{=0} 
+\underbrace{\frac12 \big(  f_0^-, \mathfrak{N}_{1,1} P_0 P_{0,2} f_1^+ \big)}_{=0} \\
& +\underbrace{\frac12 \big(   \mathfrak{N}_{0,2} P_0 P_{1,1} f_0^-, f_1^+ \big) }_{=0}
+\underbrace{\frac12 \big(   f_0^-,  \mathfrak{N}_{0,2} P_0 P_{1,1} f_1^+ \big) }_{-\frac\im 2 \widetilde \km_{1,1} \big(f_0^- , \mathfrak{N}_{0,2} f_0^- \big)\stackrel{\eqref{kN}}{=} 0 } 
\, .
\end{align*}
Hence by \eqref{cBacts}, \eqref{funperturbed}, \eqref{fm1sigma}, we have
\begin{equation*}
\begin{aligned}
\phi_{12} = &- \frac12 p_3^{[1]} \ch^{1/2} - \frac14 \ch^{-\frac12} (a_2^{[0]} +\frac12 a_2^{[2]}) +\frac14 \ch^{-\frac52}\ttf_2 (1-\ch^4) \\
 &+ \frac14 \ka_{1,2} (a_1^{[1]}\ch^\frac12 - p_1^{[1]}\ch^{-\frac12} )- \frac12 \kb_{1,2} \ch^{\frac12} p_1^{[1]} +\frac14 \kn_{0,2} p_1^{[1]}\ch^{\frac12} \, ,
\end{aligned}
\end{equation*}
which, in view of \eqref{apexp}, \eqref{n03} gives the term \eqref{phi12}.

\section{Block-decoupling and proof of Theorem \ref{abstractdec}} 
\label{sec:block}

In this section we prove Theorem \ref{abstractdec} by block-decoupling
the $ 4 \times 4 $ Hamiltonian matrix $\tL_{\mu,\e} = \tJ_4 \tB_{\mu,\e} $ in 
\eqref{Lmuepsi} obtained 
in Proposition \ref{BexpG}, expanding the computations of \cite{BMV3}
at a higher degree of accuracy.

We  first perform the singular symplectic and reversibility-preserving
change of coordinates  in \cite[Lemma 5.1]{BMV3}. 

\begin{lem}\label{decoupling1prep}
{{\bf (Singular symplectic rescaling)}}
 The conjugation of the Hamiltonian and reversible matrix $\tL_{\mu,\e} = \tJ_4 \tB_{\mu,\e} $ 
 in \eqref{Lmuepsi}  
 obtained in Proposition \ref{BexpG} 
 through the symplectic and reversibility-preserving 
 $ 4 \times 4 $-matrix 
$$
 Y := \begin{pmatrix} Q & 0 \\ 0 & Q \end{pmatrix}
 \quad  \text{with} \quad 
 Q:=\begin{pmatrix} \mu^{\frac12} & 0 \\ 0 & \mu^{-\frac12}\end{pmatrix} \, ,  \ \ \mu > 0 \, ,
$$
yields the Hamiltonian and reversible matrix 
\begin{align}
&\tL_{\mu,\e}^{(1)} := Y^{-1} \tL_{\mu,\e} Y = \tJ_4\tB^{(1)}_{\mu,\e}  =
\begin{pmatrix}   \tJ_2 E^{(1)}  & \tJ_2  F^{(1)}  \\  
\tJ_2 [F^{(1)}]^*  &  \tJ_2 G^{(1)}  \end{pmatrix}  \label{LinHprep}
 \end{align}
where $ \tB_{\mu,\e}^{(1)} $ is a self-adjoint and reversibility-preserving 
 $ 4 \times 4$  matrix 
$$
\tB_{\mu,\e}^{(1)} =
\begin{pmatrix} 
E^{(1)} & F^{(1)} \\ 
[F^{(1)}]^* & G^{(1)} 
\end{pmatrix}, \quad 
E^{(1)}  = [E^{(1)}]^* \, , \ G^{(1)} = [G^{(1)}]^* \, , 
$$
 where the $ 2 \times 2 $  reversibility-preserving matrices $E^{(1)} :=
 E^{(1)} (\mu, \epsilon) $, $ G^{(1)} :=  G^{(1)} (\mu, \epsilon) $  and $ F^{(1)} :=
 F^{(1)} (\mu, \epsilon ) $ extend analytically at $ \mu  = 0 $ with 
 the expansion
 \begin{subequations}\label{BinHprep}
\begin{align}\label{BinH1prep}
& E^{(1)} = 
\begin{pmatrix} 
\te_{11}  \mu \e^2(1+ r_1'(\e^3,\mu\e)) + \eta_{11}  \mu \e^4 - \te_{22}\frac{\mu^3}{8}(1+r_1''(\e,\mu))  
  & 
  \im \big( \frac12\te_{12}\mu+\eta_{12}\mu\e^2+ r_2(\mu\e^3,\mu^2\e,\mu^3) \big)  \\
- \im \big( \frac12\te_{12} \mu+\eta_{12}\mu\e^2 + r_2(\mu\e^3,\mu^2\e,\mu^3) \big) & -\te_{22}\frac{\mu}{8}(1+r_5(\e^2,\mu))
 \end{pmatrix} \\
 & \label{BinH2prep} G^{(1)} := 
\begin{pmatrix} 
 \mu +\gamma_{11}  \mu \e^2 +r_8( \mu \e^3, {\mu^2\e^2},\mu^3\e) &   
-\im \gamma_{12} \mu\e^2 - \im r_9(\mu\e^3,\mu^2\e) 
 \\
 \im \gamma_{12} \mu\e^2+\im  r_9(\mu\e^3, \mu^2\e 
  )  &
  \tanh(\tth\mu)+\gamma_{22} \mu\e^2 +r_{10}(\mu\e^3,\mu^2\e)
 \end{pmatrix} \\
 & \label{BinH3prep}
 F^{(1)} :=\begingroup 
\setlength\arraycolsep{2pt}
\begin{pmatrix} 
\tf_{11} \mu \e+\phi_{11}  \mu \e^3+ r_3( \mu \e^4,\mu^2\e^2,\mu^3\e ) & \im 
 \mu\e \ch^{-\frac12} +\im\phi_{12} \mu\e^3 +\im  r_4({\mu\e^4}, \mu^2 \e^2, \mu^3\e
 )  \\
  \im  \phi_{21}\mu\e+\im r_6(\mu\e^3,\mu^2\e  )    & \phi_{22}\mu\e + r_7(\mu \e^3,\mu^2\e ) 
 \end{pmatrix} \endgroup 
 \end{align}
 \end{subequations}
 where the coefficients appearing in the entries are the same of \eqref{BinG}.
\end{lem}
Note  that the matrix $\tL_{\mu,\e}^{(1)}$, 
initially defined only for $\mu \neq 0$,  extends analytically to the zero matrix at  
 $\mu = 0$.
For $\mu \neq 0$ the spectrum of $\tL_{\mu,\e}^{(1)}$ coincides with the spectrum of $\tL_{\mu,\e}$.

\paragraph{Non-perturbative step of block-decoupling.}  \label{sec:5.2}

The following lemma
computes  the first order Taylor expansions \eqref{expXentries}
of the matrix entries in \eqref{Xsylvy}  
and  then the expansion \eqref{Bsylvy1} at a higher degree of accuracy
with respect to \cite[Lemma 5.4]{BMV3}.

\begin{lem} {\bf (Step of block-decoupling)}\label{decoupling2}
There exists a $2\times 2$ reversibility-preserving matrix $ X $,
analytic in $(\mu, \e) $,  of the form
\begin{equation} \label{Xsylvy}
X := \begin{pmatrix} x_{11} & \im x_{12} \\ \im x_{21} & x_{22} \end{pmatrix}  \, ,
\qquad x_{ij}\in\bR \, , \ i,j=1,2 \, , 
\end{equation}
with  
\begin{subequations} \label{expXentries}
\begin{align}
x_{11} &= x_{11}^{(1)} \e + r(\e^3,\mu\e) \, ,
\qquad
x_{12} =  x_{12}^{(1)} \e  + r(\e^3, \mu\e) 
\\ \notag
x_{21} &= x_{21}^{(1)} \e  + x_{21}^{(3)} \e^3  + r(\e^4,\mu\e^2,\mu^2\e) \, , 
\qquad
x_{22} = x_{22}^{(1)} \e  + x_{22}^{(3)} \e^3  + r(\e^4,\mu\e^2,\mu^2\e) \, , 
\end{align}
where 
\begin{equation} \label{46b}
 x_{21}^{(1)} := - \tfrac12 \mathtt{D}_\tth^{-1}  \big( \te_{12} \tf_{11} + 2\ch^{-\frac12} \big) \, ,\quad
x_{22}^{(1)} := \tfrac12 \mathtt{D}_\tth^{-1} \big( \ch^{-\frac12}\te_{12} +2\tth \tf_{11} \big) \, , 
\end{equation}
and 
\begin{equation}\label{4.6cnew}
\begin{aligned}
 x_{11}^{(1)}   :&=   \mathtt{D}_\tth^{-1} \big(\tfrac1{16} \te_{12}\te_{22} x_{21}^{(1)} -\tfrac12 \te_{12} \phi_{21} +\phi_{22} -\tfrac18 \te_{22} x_{22}^{(1)} \big) \, ,\\  
 x_{12}^{(1)} :&=  \mathtt{D}_\tth^{-1}\big(\tfrac18 \tth \te_{22} x_{21}^{(1)}-\tth \phi_{21} +\tfrac12 \te_{12} \phi_{22}-\tfrac1{16} \te_{12}\te_{22} x_{22}^{(1)} \big) \, , \\ 
 x_{21}^{(3)} :&= \mathtt{D}_\tth^{-1}\big( -\tfrac12 \te_{11}\te_{12} x_{11}^{(1)}+\tfrac12(\gamma_{12}+\eta_{12} ) \te_{12} x_{21}^{(1)}+\tfrac12 \te_{12} \gamma_{11} x_{22}^{(1)}  \\
 &\qquad -\tfrac12 \phi_{11} \te_{12} - \te_{11} x_{12}^{(1)}-\gamma_{22} x_{21}^{(1)} -(\gamma_{12} +\eta_{12}) x_{22}^{(1)}-\phi_{12} \big) \, ,\\  
x_{22}^{(3)} :&={\mathtt{D}_\tth^{-1} \big(\tth \te_{11} x_{11}^{(1)} -\tth (\gamma_{12} +\eta_{12})x_{21}^{(1)} -\tth \gamma_{11} x_{22}^{(1)} +\tth \phi_{11} } \\
&\qquad +\tfrac12 \te_{11} \te_{12} x_{12}^{(1)}+\tfrac12 \te_{12} \gamma_{22} x_{21}^{(1)}  +\tfrac12 \te_{12} (\gamma_{12} +\eta_{12} )x_{22}^{(1)}+ {\tfrac12 \te_{12}} \phi_{12} \big)\, ,
\end{aligned}
\end{equation}
\end{subequations}
with 
$\te_{12}$, $ \te_{22} $, $\te_{11} $, 
$ \phi_{21}, \phi_{22}, \gamma_{12}, \eta_{12}, \gamma_{11}, \phi_{11}, \gamma_{22}, \phi_{12}, \tf_{11} $ computed 
 in \eqref{mengascoeffs} and 
(cfr. \cite[(5.7)]{BMV3})
 \begin{equation}\label{defDh}
 \mathtt{D}_\tth := \tth-\frac14\te_{12}^2 >0\, , \quad \forall\tth>0 \, ,
 \end{equation} 
such that the following holds true. By conjugating the Hamiltonian and reversible matrix 
$\tL_{\mu,\e}^{(1)}$, defined in \eqref{LinHprep}, with the symplectic and reversibility-preserving $ 4 \times 4 $ matrix 
\begin{equation}\label{formaS}
\exp\big(S^{(1)} \big) \, , 
\quad \text{ where } 
\qquad S^{(1)} := \tJ_4 \begin{pmatrix} 0 & \Sigma \\ \Sigma^* & 0 \end{pmatrix} \, , \qquad \Sigma:= \tJ_2 X \, , 
\end{equation}
we get the Hamiltonian and reversible matrix  
\begin{equation}\label{sylvydec}
 \tL_{\mu,\e}^{(2)} := \exp\big(S^{(1)} \big) \tL_{\mu,\e}^{(1)} \exp\big(-S^{(1)} \big)= \tJ_4 \tB_{\mu,\e}^{(2)} =
\begin{pmatrix}   \tJ_2 E^{(2)}  & \tJ_2  F^{(2)}  \\  
\tJ_2 [F^{(2)}]^*  &  \tJ_2 G^{(2)}  \end{pmatrix}\, ,
\end{equation}
where the reversibility-preserving $2\times 2$ self-adjoint 
matrix   $ E^{(2)}$ has the form 
\begin{align} 
\label{Bsylvy1}
& E^{(2)} = {
\begin{pmatrix} 
 \teWB \mu\e^2+  \etaWB \mu\e^4 + r_1'(\mu \e^5, \mu^2 \e^3 )-\te_{22}\frac{\mu^3}{8}(1+r_1''(\e,\mu))  & \im  \big( \frac12\te_{12}\mu+ r_2(\mu\e^2,\mu^2\e,\mu^3) \big)  \\
- \im  \big( \frac12\te_{12}\mu+ r_2(\mu\e^2,\mu^2\e,\mu^3) \big) & -\te_{22}\frac{\mu}{8}(1+r_5(\e,\mu))
 \end{pmatrix} }\, ,
 \end{align}
  where 
 $\teWB $
is the Whitham-Benjamin function in \eqref{funzioneWB} and
\begin{align}\label{etaWB}
&\etaWB = \eta_{11} + { x_{21}^{(1)} \phi_{12} + x_{21}^{(3)} \ch^{-\frac12} - x_{22}^{(1)}\phi_{11} - x_{22}^{(3)} \tf_{11} + \frac32 (x_{21}^{(1)})^2 x_{22}^{(1)}\phi_{22} +(x_{21}^{(1)})^2 x_{12}^{(1)}\ch^{-\frac12}} \\ \notag
&\qquad \qquad   {- \frac32 x_{21}^{(1)} x_{12}^{(1)} x_{22}^{(1)}\tf_{11}+ \frac32 (x_{22}^{(1)})^2  x_{21}^{(1)}\phi_{21} -\frac32 x_{22}^{(1)}x_{11}^{(1)} x_{21}^{(1)} \ch^{-\frac12} + (x_{22}^{(1)})^2  {x_{11}^{(1)}\tf_{11}}}  
\end{align}
with $x_{11}^{(1)}$, $x_{12}^{(1)}$, $x_{22}^{(1)}$, $x_{21}^{(1)}$, 
$x_{21}^{(3)}$, $x_{22}^{(3)} $  in \eqref{expXentries} and
the remaining coefficients in   \eqref{mengascoeffs}, whereas
the reversibility-preserving $2\times 2$ self-adjoint 
matrix  $ G^{(2)} $ has the form
\begin{equation}
  \label{Bsylvy2} G^{(2)} = 
 \begin{pmatrix} 
\mu+ r_8(\mu\e^2, \mu^3 \e )
&   - \im r_9(\mu\e^2,\mu^2\e) \\
  \im  r_9(\mu\e^2, \mu^2\e)  & \tanh(\tth\mu) + r_{10}(\mu\e) 
 \end{pmatrix}\, ,
 \end{equation}
and finally
 \begin{equation}
 \label{Bsylvy3}F^{(2)}= \begin{pmatrix}
  r_3(\mu\e^3 ) 
& \im r_4(\mu\e^3 ) \\
\im r_6(\mu\e^3 ) &
r_7(\mu\e^3)
\end{pmatrix} \, .
 \end{equation}
\end{lem}

The rest of the section is devoted to the proof of Lemma \ref{decoupling2}.
In Lemma 5.4 of \cite{BMV3} we proved 
the existence of a matrix $ X $ as in \eqref{Xsylvy} such that we 
obtain \eqref{sylvydec} with matrices  $ G^{(2)}, F^{(2)}  $  as in 
\eqref{Bsylvy2}-\eqref{Bsylvy3}
and a $ 2 \times 2 $-self adjoint and reversibility preserving 
matrix $ E^{(2)} $ whose first entry  has the form 
$ [E^{(2)}]_{11} =  \teWB  \mu \epsilon^2
+ r_1 (\mu \e^3, \mu^2 \e^2 ) $. The main result of Lemma 
\ref{decoupling2} is 
that the first entry $ [E^{(2)}]_{11}  $ 
 has the better expansion 
$$ 
[E^{(2)}]_{11} =  \teWB  \mu \epsilon^2
+ r_1 (\mu \e^3, \mu^2 \e^2 ) = 
\teWB \mu\e^2+  \etaWB \mu\e^4 + r_1'(\mu \e^5, \mu^2 \e^3 ) \, 
$$
with $\etaWB  $ computed in \eqref{etaWB}, 
which is relevant to determine the stability/instablity of the Stokes wave at the 
critical depth. Clearly we could compute explicitly also other  Taylor coefficients
of the matrix entries of 
$ E^{(2)}, G^{(2)}, F^{(2)}  $,  but it is not needed. 

The coefficients 
$x_{21}^{(1)}$ and $x_{22}^{(1)}$ in \eqref{46b} 
were already computed in \cite[Lemma 5.4]{BMV3}.

We now expand in Lie series 
the Hamiltonian and reversible matrix $ \tL_{\mu,\e}^{(2)} 
= \exp (S)\tL_{\mu,\e}^{(1)} \exp (-S) $
where for simplicity we set $ S := S^{(1)} $. 
We split $\tL_{\mu,\e}^{(1)}$ 
into its $2\times 2$-diagonal and off-diagonal Hamiltonian and reversible matrices
\begin{align}
& \qquad  \qquad \qquad  \qquad  \qquad \qquad \tL_{\mu,\e}^{(1)} = D^{(1)} + R^{(1)}  \, , \label{LDR}\\
& 
D^{(1)} :=\begin{pmatrix} D_1 & 0 \\ 0 & D_0 \end{pmatrix} :=  \begin{pmatrix} \tJ_2 E^{(1)} & 0 \\ 0 & \tJ_2 G^{(1)} \end{pmatrix}, \quad 
R^{(1)} := \begin{pmatrix}  0 & \tJ_2 F^{(1)} \\ \tJ_2 [F^{(1)}]^* & 0 \end{pmatrix} , \notag
\end{align} 
and we perform the Lie expansion
\begin{align}
\label{lieexpansion}
& \tL_{\mu,\e}^{(2)} 
 = \exp(S)\tL_{\mu,\e}^{(1)} \exp(-S)  =  D^{(1)} + [S,D^{(1)}] + 
 \frac12 \big[ S, [S, D^{(1)}] \big] + 
 R^{(1)}  + [S, R^{(1)}]  \\
 & + 
\frac12 \int_0^1 (1-\tau)^2 \exp(\tau S)  \text{ad}_S^3( D^{(1)} )  \exp(-\tau S) \, \de \tau 
+ \int_0^1 (1-\tau) \, \exp(\tau S) \, \text{ad}_S^2( R^{(1)} ) \, \exp(-\tau S) \, \de \tau \notag 
\end{align} 
where $\text{ad}_A(B) := [A,B] := AB - BA $ denotes the commutator 
between the linear operators $ A, B $.

We look for a $ 4 \times 4 $ matrix $S$ as in \eqref{formaS}
 that solves
the homological equation
$  R^{(1)}  +\lie{S}{ D^{(1)} } = 0  $, 
which, recalling \eqref{LDR}, reads 
\begin{equation}\label{homoesp}
\begin{pmatrix} 0 & \tJ_2F^{(1)}+\tJ_2\Sigma D_0
- D_1\tJ_2\Sigma \\ 
\tJ_2{[F^{(1)}]}^*+\tJ_2\Sigma^*D_1-D_0\tJ_2\Sigma^* & 0 \end{pmatrix} =0 \, .
\end{equation}
Writing  $ \Sigma =\tJ_2  X  $, namely $ X = - \tJ_2  \Sigma $,  
the equation \eqref{homoesp} amounts to solve the  ``Sylvester'' equation 
\begin{equation}\label{Sylvestereq}
D_1 X - X D_0 = - \tJ_2F^{(1)}   \, .
\end{equation}
We write  the matrices $ E^{(1)}, F^{(1)}, G^{(1)}$ in \eqref{LinHprep}
as 
\begin{equation}\label{splitEFGprep}
E^{(1)} = 
\begin{pmatrix} 
E_{11}^{(1)} & \im E_{12}^{(1)} \\ 
- \im E_{12}^{(1)} & E_{22}^{(1)}
\end{pmatrix}\, , \quad
F^{(1)} = 
\begin{pmatrix} 
F_{11}^{(1)} & \im F_{12}^{(1)} \\ 
\im F_{21}^{(1)} & F_{22}^{(1)}
\end{pmatrix} \, , \quad
G^{(1)} = 
\begin{pmatrix} 
G_{11}^{(1)} & \im G_{12}^{(1)} \\ 
- \im G_{12}^{(1)} & G_{22}^{(1)} 
\end{pmatrix} 
\end{equation} 
where the real numbers 
$ E_{ij}^{(1)}, F_{ij}^{(1)}, G_{ij}^{(1)} $, $ i , j = 1,2 $, have the expansion 
in  \eqref{BinH1prep}, \eqref{BinH2prep}, \eqref{BinH3prep}.  
Thus, by  \eqref{LDR}, \eqref{Xsylvy}  and \eqref{splitEFGprep},
the equation \eqref{Sylvestereq}  amounts  to solve the 
 $4\times 4$ real linear system 
\begin{align}\label{Sylvymat}
\underbrace{ \begin{pmatrix}  
G_{12}^{(1)} - E_{12}^{(1)} &   G_{11}^{(1)} &  E_{22}^{(1)} & 0 \\   
G_{22}^{(1)} & G_{12}^{(1)} - E_{12}^{(1)} & 0 & - E_{22}^{(1)} \\
 E_{11}^{(1)} & 0 & G_{12}^{(1)} - E_{12}^{(1)}  & - G_{11}^{(1)} \\
 0 &  - E_{11}^{(1)} & -G_{22}^{(1)}  &  G_{12}^{(1)} - E_{12}^{(1)}
 \end{pmatrix}}_{=: {\cal A} }
 \underbrace{ \begin{pmatrix} x_{11} \\ x_{12} \\ x_{21} \\ x_{22} \end{pmatrix}}_{ =: \vec x}
  =
  \underbrace{
    \begin{pmatrix} 
 -F_{21}^{(1)}  \\  F_{22}^{(1)} \\ - F_{11}^{(1)} \\  F_{12}^{(1)}
 \end{pmatrix}
 }_{=: \vec f}.
\end{align}
By \cite{BMV3} system \eqref{Sylvymat} admits a unique solution. We now prove
that it has the form \eqref{expXentries}. 

\begin{lem}
The vector $ \vec x = (x_{11}, x_{12}, x_{21}, x_{22}) $ with entries in \eqref{expXentries}
solves \eqref{Sylvymat}. 
\end{lem}

\begin{proof}
Since $ \tanh(\tth\mu) = \tth\mu + r(\mu^3)$, we have
\begin{align}
&  G_{12}^{(1)} - E_{12}^{(1)}  \label{abcde} = 
- \te_{12}\frac{\mu}{2}- (\gamma_{12}+\eta_{12}) \mu\e^2 +r( \mu \e^3, \mu^2\e,\mu^3 ) \, ,  \\ \notag
&  G_{11}^{(1)} 
=\mu + \gamma_{11} \mu\e^2+   r_8(\mu\e^3,{\mu^2\e^2}, \mu^3 \e  
) \, ,\ \  
E_{22}^{(1)} =-\te_{22}\frac{\mu}{8}(1+r_5(\e^2,\mu)) \, 
,
\\ \notag
&  G_{22}^{(1)}
= \mu \tth  + \gamma_{22} \mu\e^2+ r( \mu \e^3, \mu^2\e, \mu^3 )\, , \ \ 
  E_{11}^{(1)} = \te_{11} \mu\e^2 +r(\mu\e^4,\mu^2\e^3, \mu^3) \, ,
\end{align}
with  coefficients  $\te_{12}$, $\gamma_{12}$, $\eta_{12}$, $\gamma_{11}$, $\te_{22}$, $\gamma_{22}$ and $\te_{11}$ computed in \eqref{mengascoeffs}. We exploit that the
 terms $ x_{21}^{(1)} $ and $ x_{22}^{(2)} $ have been already computed in 
  \cite[Lemma 5.4]{BMV3}, in order  to get $ x_{11}^{(1)} $  and $ x_{12}^{(2)} $  in \eqref{expXentries} as solutions of the system
\begin{equation}\label{newx3}
 \begingroup 
\setlength\arraycolsep{1pt}\begin{pmatrix} -\tfrac12 \te_{12} & 1 \\ \tth & -\tfrac12 \te_{12} \end{pmatrix}\endgroup  \begin{pmatrix} x_{11}^{(1)} \\ x_{12}^{(1)} \end{pmatrix} = \begin{pmatrix} \tfrac18 \te_{22} x_{21}^{(1)}-\phi_{21} \\ \phi_{22}-\tfrac18 \te_{22} x_{22}^{(1)} \end{pmatrix}  \, ,\quad \det  \begingroup 
\setlength\arraycolsep{1pt}\begin{pmatrix} -\tfrac12 \te_{12} & 1 \\ \tth & -\tfrac12 \te_{12} \end{pmatrix}\endgroup \stackrel{\eqref{defDh}}= - \mathtt{D}_\tth < 0\, ,  
\end{equation}
given, using also \eqref{BinH3prep}, 
by the first two lines in \eqref{Sylvymat} at order $ \mu \e $.
 Similarly, $x_{21}^{(3)}$ and $x_{22}^{(3)}$  in \eqref{expXentries} solve the system
\begin{equation}\label{newx31}
 \begingroup 
\setlength\arraycolsep{0pt}\begin{pmatrix} -\tfrac12 \te_{12} & -1 \\ -\tth & -\tfrac12 \te_{12} \end{pmatrix}\endgroup  \begin{pmatrix} x_{21}^{(3)} \\ x_{22}^{(3)} \end{pmatrix} = \begin{pmatrix} -\te_{11} x_{11}^{(1)} + (\gamma_{12} +\eta_{12}) x_{21}^{(1)}+\gamma_{11} x_{22}^{(1)}-\phi_{11}  \\ \te_{11} x_{12}^{(1)}+\gamma_{22} x_{21}^{(1)}+(\gamma_{12}+\eta_{12})x_{22}^{(1)}+\phi_{12} \end{pmatrix}  \, ,
\end{equation}
which comes, also by \eqref{BinH3prep}, from
 the last two lines of \eqref{Sylvymat} at order $\mu\e^3$. 
 The solutions of  \eqref{newx3}-\eqref{newx31} are given in \eqref{4.6cnew}. 
\end{proof}

We now prove the expansion \eqref{Bsylvy1}.
 Since the matrix $ S $ solves the homological equation $\lie{S}{ D^{(1)} }+  R^{(1)} =0$, identity \eqref{lieexpansion} simplifies to 
\begin{equation}\label{Lie2prep}
\tL_{\mu,\e}^{(2)}
  =   D^{(1)}  +\frac12 \big[ S, R^{(1)} \big]+
\frac12 \int_0^1 (1-\tau^2) \, \exp(\tau S) \, \text{ad}_S^2( R^{(1)} ) \, \exp(-\tau S) \de \tau \, .
\end{equation}
By plugging the Lie expansion 
$$
\begin{aligned}
&\exp(\tau S) \, \text{ad}_S^2( R^{(1)} )   
 \, \exp(-\tau S) \\ 
 &\qquad= 
 \text{ad}_S^2( R^{(1)} )   + \tau
  \text{ad}_S^3( R^{(1)} ) +\tau^2 
  \int_0^1 (1-\tau') \exp(\tau'\tau S) \, \text{ad}_S^4( R^{(1)} )   
 \, \exp(-\tau'\tau S) \, \de \tau' 
 \end{aligned}
$$
into \eqref{Lie2prep} we get 
\begin{subequations}\label{Lie2tutta}
\begin{align} \label{Lie2}
&\tL_{\mu,\e}^{(2)}
  =  D^{(1)}  +\frac12 \big[ S, R^{(1)} \big] +
\frac13  \text{ad}_S^2( R^{(1)} )+ 
\frac{1}{8}
\text{ad}_S^3( R^{(1)} )  \\ \label{Lieremainder}
&+\frac12 \int_0^1 (1-\tau^2)\tau^2 \int_0^1 (1-\tau') \exp(\tau\tau'S) \text{ad}_S^4(R^{(1)}) \exp(-\tau\tau' S) \de \tau' \de \tau\, .
\end{align}
\end{subequations}
Next we compute the commutators in the expansion \eqref{Lie2}.

\begin{lem}\label{lem:EGtilde}
One has
\begin{equation}\label{Lieeq2}
 \frac12 \big[S, R^{(1)} \big] = 
 \begin{pmatrix} 
 \tJ_2 \tilde E_1 & 0 \\ 0 &\tJ_2 \tilde G_1
 \end{pmatrix}
\end{equation}
where  $ \tilde E_1 $, $  \tilde G_1  $
are  self-adjoint and reversibility-preserving matrices of the form
\begin{equation}
\begin{aligned}
& \tilde E_1 = \begin{pmatrix} 
 \tilde\te_{11}\mu\e^2  + \tilde \eta_{11}^{(a)} \mu \e^4  +  \tilde r_1(\mu\e^5,\mu^2\e^3, \mu^3\e^2) & 
 \im \big(\tilde \te_{12} \mu \e^2 + \tilde  r_2(\mu\e^4, \mu^2 \e^2)\big) \\ 
 - \im \big(  \tilde\te_{12} \mu \e^2 +  \tilde r_2(\mu\e^4, \mu^2 \e^2) \big) &   
  \tilde  r_5(\mu\e^2) \end{pmatrix} \, , \\
& \tilde G_1 = \begin{pmatrix} 
 \tilde\tg_{11}\mu\e^2  +   \tilde r_8(\mu\e^4,\mu^2\e^2) & 
 \im \big(\tilde \tg_{12} \mu \e^2 + \tilde  r_9(\mu\e^4, \mu^2 \e^2)\big) \\ 
 - \im \big(  \tilde\tg_{12} \mu \e^2 +  \tilde r_9(\mu\e^4, \mu^2 \e^2) \big) &   \tilde \tg_{22} \mu\e^2
 + \tilde  r_{10}(\mu\e^4,\mu^2\e^2) \end{pmatrix} \, ,  \end{aligned}  \label{tilde.E.G}
\end{equation}
where 
\begin{equation}
\begin{aligned} \label{primocomm}
&\tilde\te_{11}:=  x_{21}^{(1)} \ch^{-\frac12} - x_{22}^{(1)}\tf_{11}\, ,\quad \tilde \eta_{11}^{(a)}:=  
x_{21}^{(1)}\phi_{12}+x_{21}^{(3)}\ch^{-\frac12}-x_{22}^{(1)}\phi_{11}-x_{22}^{(3)}\tf_{11}
 \, ,\quad  \\ 
&\tilde\te_{12}:= -\tilde\tg_{12}:={\frac12 \big( x_{21}^{(1)}\phi_{22}+x_{22}^{(1)}\phi_{21} -x_{11}^{(1)}\ch^{-\frac12}-x_{12}^{(1)}\tf_{11} \big)} \, , \\ 
&\tilde\tg_{11}:=  x_{11}^{(1)}\tf_{11}+x_{21}^{(1)}\phi_{21} \, ,\quad
\tilde\tg_{22}:= x_{22}^{(1)}\phi_{22} +x_{12}^{(1)}\ch^{-\frac12} \,  .
 \end{aligned}
 \end{equation}
\end{lem}

\begin{proof}
By \eqref{formaS}, \eqref{LDR},  and since $ \Sigma = \tJ_2 X $, we have  
\begin{equation}\label{EGtilde}
\frac12 \lie{S}{ R^{(1)} } 
= \begin{pmatrix} \tJ_2 \tilde E_1 & 0 \\ 0 &\tJ_2 \tilde G_1 \end{pmatrix},\quad 
\tilde E_1 := \textup{\textbf{Sym}} \big[ \tJ_2 X \tJ_2  [F^{(1)}]^* \big]
\, , \quad 
\tilde G_1 :=  \textup{\textbf{Sym}} \big[  X^* F^{(1)} \big] \, ,  
\end{equation}
where $ \textup{\textbf{Sym}}[A] := \frac12 (A+ A^* ) $, see \cite[(5.28)-(5.29)]{BMV3}. 
By   \eqref{Xsylvy}, \eqref{splitEFGprep}, setting $F=F^{(1)}$, we have   
\begin{align}\label{reuseEtilde}
\tJ_2 X \tJ_2  F^* &  =  
 \begin{pmatrix}   x_{21}F_{12}-x_{22}F_{11} &  \im  (x_{21}F_{22}+x_{22} F_{21})
 \\ 
 \im (x_{11}F_{12}+x_{12}F_{11}) & - x_{11}F_{22} + x_{12}F_{21} \end{pmatrix} \, , \\   
X^*   F &=  \begin{pmatrix}
x_{11}F_{11}+x_{21}F_{21} &  \im ( x_{11}F_{12} - x_{21}F_{22}) \\ \im (x_{22}F_{21}-x_{12}F_{11} ) & x_{22}F_{22}+ x_{12}F_{12}  \end{pmatrix}  \, , \label{reuseEtild1}
\end{align}
and the expansions in \eqref{tilde.E.G} with the coefficients given in  \eqref{primocomm} follow by 
\eqref{reuseEtilde}, \eqref{reuseEtild1}, \eqref{expXentries} and \eqref{BinH3prep}.
\end{proof}

\begin{lem}\label{lem:EGtilde2}
One has
\begin{equation}\label{Lieeq22}
\begin{aligned}
& \frac13 \mathrm{ad}_S^2( R^{(1)} ) = \begin{pmatrix} 0 & \tJ_2 \tilde F \\ \tJ_2 \tilde F^* & 0  \end{pmatrix},
\end{aligned}
\end{equation}
where  $ \tilde F  $ is 
a  reversibility-preserving matrix of the form 
\begin{equation}  \label{tilde.F}
\tilde F = \begin{pmatrix} 
 \tilde\tf_{11}\mu\e^3  +   \tilde r_3(\mu\e^5,\mu^2\e^3) & 
 \im\tilde \tf_{12}\mu\e^3 + \im  \tilde  r_4(\mu\e^5,\mu^2\e^3) \\ 
 \im  \tilde r_6(\mu\e^3)  &   
  \tilde  r_{7}(\mu\e^3) \end{pmatrix} \, ,
\end{equation}
with
\begin{equation}\label{tildeFcoeffs}
\begin{aligned}
\tilde\tf_{11} & := {\frac43 x_{21}^{(1)} x_{11}^{(1)} \ch^{-\frac12} -\frac43 x_{22}^{(1)} x_{11}^{(1)} \tf_{11} -\frac43 x_{22}^{(1)} x_{21}^{(1)} \phi_{21} +\frac23 x_{21}^{(1)} x_{12}^{(1)} \tf_{11} - \frac23 (x_{21}^{(1)})^2 \phi_{22} } \, ,\\
\tilde\tf_{12} & := {\frac43 x_{21}^{(1)} x_{22}^{(1)} \phi_{22} + \frac43 x_{12}^{(1)} x_{21}^{(1)} \ch^{-\frac12} -\frac43 x_{12}^{(1)} x_{22}^{(1)} \tf_{11} + \frac23(x_{22}^{(1)})^2 \phi_{21} -\frac23 x_{11}^{(1)} x_{22}^{(1)} \ch^{-\frac12} }  \,.
\end{aligned}
\end{equation}
\end{lem}

\begin{proof}
Using the form of  $ S $ in \eqref{formaS} and  $[S,  R^{(1)} ]$ in \eqref{Lieeq2} we deduce \eqref{Lieeq22}  with
\begin{equation}\label{def:tildeF}
\tilde F:= 
\frac2{3} \big(\tJ_2 X  \tJ_2 \tilde G_1 + \tilde E_1 X \big)  
\end{equation}
where $ \tilde E_1 $ and $\tilde G_1 $ are the matrices in  \eqref{tilde.E.G}. 
Writing  $ \tilde E_1 = 
{\footnotesize \begin{pmatrix} 
{[\tilde E_1]}_{11} & \im{[\tilde E_1]}_{12} \\ 
- \im {[\tilde E_1]}_{12} & {[\tilde E_1]}_{22}
\end{pmatrix} }$, 
$\tilde G_1 = 
{\footnotesize\begin{pmatrix} 
{[\tilde G_1]}_{11} & \im {[\tilde G_1]}_{12}\\ 
- \im {[\tilde G_1]}_{12} &{[\tilde G_1]}_{22}
\end{pmatrix} }$
 we have, in view of  \eqref{Xsylvy}, 
\begin{equation}\label{secondocomm}
\begin{aligned}
\tJ_2 X \tJ_2 \tilde G_1 &= \begin{pmatrix} 
x_{21} {[\tilde G_1]}_{12}- x_{22} {[\tilde G_1]}_{11} & \im\big(x_{21} {[\tilde G_1]}_{22}- x_{22} {[\tilde G_1]}_{12} \big)  \\ 
\im\big(x_{11} {[\tilde G_1]}_{12}+ x_{12} {[\tilde G_1]}_{11} \big) &- x_{11} {[\tilde G_1]}_{22}- x_{12} {[\tilde G_1]}_{12}
\end{pmatrix}\, , \\
\tilde E_1 X &= \begin{pmatrix} 
x_{11} {[\tilde E_1]}_{11}- x_{21} {[\tilde E_1]}_{12} & \im\big(x_{12} {[\tilde E_1]}_{11}+ x_{22} {[\tilde E_1]}_{12} \big)  \\ 
\im\big(x_{21} {[\tilde E_1]}_{22}- x_{11} {[\tilde E_1]}_{12} \big) & x_{12} {[\tilde E_1]}_{12}+ x_{22} {[\tilde E_1]}_{22}
\end{pmatrix}\, . \\
\end{aligned}
\end{equation}
By \eqref{def:tildeF}, \eqref{secondocomm}, \eqref{expXentries} and \eqref{tilde.E.G} we deduce that the matrix $\tilde F$ has the expansion \eqref{tilde.F} 
with
\begin{equation*}
\begin{aligned}
\tilde\tf_{11} & =  \frac23 \big( x_{21}^{(1)}\tilde \tg_{12} - x_{22}^{(1)} \tilde\tg_{11} + x_{11}^{(1)}\tilde \te_{11}- x_{21}^{(1)} \tilde \te_{12} \big) \, ,\\
\tilde\tf_{12} & =  \frac23 \big( x_{21}^{(1)}\tilde \tg_{22} - x_{22}^{(1)} \tilde\tg_{12} + x_{12}^{(1)}\tilde \te_{11}+ x_{22}^{(1)} \tilde \te_{12} \big) \, ,
\end{aligned}
\end{equation*}
which, by \eqref{primocomm}, gives \eqref{tildeFcoeffs}.
\end{proof}

\begin{lem}\label{lem:EGtilde3}
One has
\begin{equation}\label{Lieeq23}
\begin{aligned}
& \frac1{8} \mathrm{ad}_S^3( R^{(1)} ) = \begin{pmatrix} \tJ_2 \tilde E_3 & 0 \\ 0 &\tJ_2 \tilde G_3 \end{pmatrix},
\end{aligned}
\end{equation}
where
the  self-adjoint and reversibility-preserving matrices
 $ \tilde E_3, \tilde G_3$ in \eqref{Lieeq23} have entries of size $\cO(\mu\e^4) $. In particular the first entry of the matrix $ \tilde E_3$ has the expansion
\begin{equation}\label{tilde.E.G3}
{[\tilde E_3]}_{11} =
 \tilde\eta_{11}^{(b)}\mu\e^4  +  r(\mu\e^5,\mu^2\e^4) 
\end{equation}
with
\begin{align}\label{terzocomm}
 \tilde\eta_{11}^{(b)} 
& := { \frac32 (x_{21}^{(1)})^2 x_{22}^{(1)}\phi_{22} +(x_{21}^{(1)})^2 x_{12}^{(1)}\ch^{-\frac12}  - \frac32 x_{21}^{(1)} x_{12}^{(1)} x_{22}^{(1)}\tf_{11}}\\  \notag
 &\ \quad {+ \frac32 (x_{22}^{(1)})^2  x_{21}^{(1)}\phi_{21} -\frac32 x_{22}^{(1)}x_{11}^{(1)} x_{21}^{(1)} \ch^{-\frac12}
 + (x_{22}^{(1)})^2  {x_{11}^{(1)}\tf_{11}}} 
 \, .
\end{align}
\end{lem}

\begin{proof}
Since $\frac1{8} \mathrm{ad}_S^3( R^{(1)} ) = \frac38 [S, \frac13 \mathrm{ad}_S^2(R^{(1)})]$
and using \eqref{Lieeq22}, the identity 
 \eqref{Lieeq23} holds with 
\begin{equation}\label{EGtilde3}
\tilde E_3 := \frac34\textup{\textbf{Sym}} \big[ \tJ_2 X \tJ_2  [\tilde F]^* \big] 
\, , \qquad 
\tilde G_3 :=  \frac34 \textup{\textbf{Sym}} \big[  X^* \tilde F \big] \, .
\end{equation}
Since, by \eqref{expXentries} the matrix $X$ in \eqref{Xsylvy} 
has entries of size $ \cO(\e) $ and the matrix $\tilde F$ 
in  \eqref{tilde.F} has entries of size $\cO(\mu\e^3) $ we 
deduce that  the matrices $\tilde E_3$, $\tilde G_3$ in \eqref{EGtilde3} have entries of 
size $\cO(\mu\e^4) $. 
By \eqref{EGtilde3} and denoting $\tilde F = {\footnotesize
\begin{pmatrix} 
\tilde F_{11} & \im \tilde F_{12} \\ 
\im \tilde F_{21} & \tilde F_{22}
\end{pmatrix} } $,
we deduce, similarly to \eqref{reuseEtilde},
that ${[\tilde E_3]}_{11} =\frac34( x_{21} \tilde F_{12} - x_{22} \tilde F_{11}) $ which, by \eqref{expXentries} and \eqref{tilde.F}, gives  \eqref{tilde.E.G3} with
$ \tilde\eta_{11}^{(b)} = \frac34 \big( x_{21}^{(1)} \tilde \tf_{12} - x_{22}^{(1)}\tilde \tf_{11} \big) $
which by \eqref{tildeFcoeffs}  gives \eqref{terzocomm}.
\end{proof}

Finally we show that the term in \eqref{Lieremainder}  is small.

\begin{lem}\label{lem:series}
The $ 4 \times 4 $  Hamiltonian and reversible  matrix  $\footnotesize \begin{pmatrix}
\tJ_2 \widehat E & \tJ_2 \widehat F\\
\tJ_2 [ \widehat F ]^* & \tJ_2 \widehat G
\end{pmatrix}$ given by \eqref{Lieremainder}
has the $ 2 \times 2 $ 
self-adjoint and reversibility-preserving  blocks  $\widehat E $, $ \widehat G$
and the $2\times 2$ reversibility-preserving block $ \widehat F$ all with entries of size  $\cO(\mu\e^5)$. 
\end{lem}

\begin{proof}
By the Hamiltonian and reversibility properties of $S $ and $  R^{(1)}  $ the matrix $ \textup{ad}_S^4  (R^{(1)}) $ is Hamiltonian  and reversible and the same holds, for any $\tau,\tau' \in [0,1]$, for 
\begin{equation}\label{Lieaux}
\exp(\tau\tau' S) \, \textup{ad}_S^4( R^{(1)} ) \, \exp(-\tau\tau' S) =  
[S,\textup{ad}_S^3( R^{(1)} )]  (1 + \cO(\mu,\e))\, .
\end{equation}
The claimed estimate on the entries  of the matrix given by \eqref{Lieremainder} follows by \eqref{Lieaux} and
because 
$ S $ in \eqref{formaS} has entries of size $\cO(\e)$ and  $ \textup{ad}_S^3( R^{(1)} ) $ in \eqref{Lieeq23} has entries of size $\cO(\mu\e^4)$.
\end{proof}

\begin{proof}[Proof of Lemma \ref{decoupling2}.]
It follows by \eqref{Lie2tutta}, \eqref{LDR} and Lemmata  \ref{lem:EGtilde}, \ref{lem:EGtilde2}, \ref{lem:EGtilde3} and 
\ref{lem:series}. 
The matrix  $E^{(2)} := E^{(1)} + \tilde E_1 
 + \tilde E_3 +  \widehat{ E}$ has the expansion  in \eqref{Bsylvy1}, with 
 $$ 
 \teWB := \te_{11} + \tilde \te_{11} = 
  \te_{11} -  
\mathtt{D}_\tth^{-1} \big( \ch^{-1} + \tth \tf_{11}^2 +\te_{12}\tf_{11}\ch^{-\frac12} \big)
 \, , 
 \qquad \etaWB := \eta_{11}+ \tilde \eta_{11}^{(a)}+ \tilde \eta_{11}^{(b)} \, , 
 $$ 
 as in \eqref{etaWB}.
Furthermore $G^{(2)} := G^{(1)} + \tilde G_1 +   \tilde G_3 + \widehat{G} $ has the expansion  in \eqref{Bsylvy2}
and  $F^{(2)} :=  \tilde F +   \widehat{F} $  has the expansion in \eqref{Bsylvy3}.
\end{proof}

\paragraph{Complete block-decoupling and proof of the main result.}\label{section34}

Finally Theorem \ref{abstractdec} is proved as in \cite{BMV3}
by  block-diagonalizing the  $ 4\times 4$ Hamiltonian and reversible 
matrix $\tL_{\mu,\e}^{(2)}$    in \eqref{sylvydec},
\begin{align}
&  \tL_{\mu,\e}^{(2)} = D^{(2)} + R^{(2)} \, , \quad 
D^{(2)}:= 
  \begin{pmatrix} \tJ_2 E^{(2)} & 0 \\ 0 & \tJ_2 G^{(2)}  \end{pmatrix}, \quad 
R^{(2)}:= \begin{pmatrix}  0 & \tJ_2 F^{(2)} \\ \tJ_2 [F^{(2)}]^* & 0 \end{pmatrix} . \label{LDR2}
\end{align} 
The next lemma is  \cite[Lemma 5.9]{BMV3}. 

\begin{lem}\label{ultimate}
There exist a  $4\times 4$ reversibility-preserving Hamiltonian  matrix $S^{(2)}:=S^{(2)}(\mu,\e)$ of the form \eqref{formaS}, analytic in $(\mu, \e)$, of size 
{$\cO(\e^3)$}, and a $4\times 4$ block-diagonal reversible Hamiltonian matrix $P:=P(\mu,\e)$, analytic in $(\mu, \e)$, of size ${ \cO(\mu\e^6)}$  such that 
\begin{equation}\label{ultdec}
\exp(S^{(2)})(D^{(2)}+R^{(2)}) \exp(-S^{(2)}) = D^{(2)}+P \, . 
\end{equation}
\end{lem}

By \eqref{ultdec}, \eqref{Bsylvy1}-\eqref{Bsylvy2} and the fact that  
$ P $ has size ${ \cO(\mu\e^6)}$ we deduce Theorem \ref{abstractdec}: 
there exists a symplectic and reversibility-preserving linear map that conjugates
the matrix $\im\ch\mu+\tL_{\mu,\e} $ (which represents $ {\cal L}_{\mu,\e} $) 
with $\tL_{\mu,\e} $ in \eqref{Lmuepsi}
into the  Hamiltonian and reversible matrix \eqref{matricefinae}  with 
$\mathtt{U} $ in \eqref{Udav} and $\mathtt S$   in \eqref{S}. 
The function $\DeltaBF(\tth;\mu,\e)$ expands as in \eqref{DWB}.

\appendix

\section{Fourth-order expansion of the Stokes waves}\label{sec:App3}

In this Appendix we compute the Taylor coefficients \eqref{allcoefStokes}
of  the fourth order expansions  \eqref{exp:Sto} of the Stokes waves. 
We also compute 
the  fourth order expansion \eqref{SN1} of the 
$ 2 \pi $-periodic functions $p_\e (x) $ and $a_\e (x) $ and 
the expansion \eqref{expfe0} 
of the constant $ \mathtt{f}_\e $  in \eqref{cLepsilon}.

\subsection{Expansion of Stokes waves}\label{sec:App31}

 \begin{prop}
{\bf (Expansion of Stokes waves)} \label{expstokes} The Stokes waves 
 $ \eta_\e (x) $, $\psi_\e (x) $ and the speed $ c_\epsilon $ in Theorem \ref{LeviCivita} have the expansions 
 \eqref{exp:Sto} with coefficients 
\begin{subequations}\label{allcoefStokes}
  \begin{align} \label{expcoef}
 &\eta_{2}^{[0]} :=  \frac{\ch^4-1}{4\ch^2} \, , \qquad 
\eta_{2}^{[2]} := \frac{3-\ch^4}{4\ch^6} \, , \qquad 
\psi_{2}^{[2]} := \frac{3+\ch^8}{8\ch^7} \, , \qquad 
\\  
\label{expc2}
&c_2:=
\frac{9-10\ch^4+9\ch^8  }{16\ch^7}
+ 
{\frac{(1-\ch^4) }{2\ch}}
\eta_2^{[0]} =  
  \frac{-2 \ch^{12}+13 \ch^8-12 \ch^4+9}{16 \ch^7}\, , \\
\label{expcoefbis}
 & \begin{aligned} &\eta_{3}^{[1]} := \frac{-2\ch^{12}+3\ch^8+3}{16\ch^8(1+\ch^2)}  \, , \qquad 
\eta_{3}^{[3]} :=  \frac{-3\ch^{12}+9\ch^8-9\ch^4+27}{64\ch^{12}}  \, , \\
&\psi_{3}^{[1]} := \frac{2\ch^{12}-3\ch^{8}-3}{16\ch^7(1+\ch^2)} \, , \qquad \psi_{3}^{[3]} :=   \frac{-9\ch^{12}+19\ch^{8}+5\ch^4+9}{64\ch^{13}}\, , 
\end{aligned} \\
\label{expcoeftris}
 & \begin{aligned} &\eta_{4}^{[0]} :=  \frac{-4 \ch^{20}-4 \ch^{18}+17 \ch^{16}+6 \ch^{14}-48 \ch^8+6 \ch^6+36 \ch^4-9}{64 \ch^{14}} \, , \\  
 &\eta_{4}^{[2]} := \frac{1}{384 \ch^{18} (\ch^2+1)} \Big(-24 \ch^{22}+285 \ch^{18}+177 \ch^{16}-862 \ch^{14}-754 \ch^{12}\\ 
 &\qquad +1116 \ch^{10}+1080 \ch^8-162 \ch^6-54 \ch^4-81 \ch^2-81 \Big) \, , \\ 
&\eta_{4}^{[4]} := \frac{21 \ch^{20}+\ch^{16}-262 \ch^{12}+522 \ch^8+81 \ch^4+405}{384 \ch^{18} (\ch^4+5)} \, , \\
&\psi_{4}^{[2]} := \frac{1}{768 \ch^{19} (\ch^2+1)} \Big(-12 \ch^{26}-36 \ch^{24}+57 \ch^{22}+93 \ch^{20}+51 \ch^{18}-21 \ch^{16}-646 \ch^{14}\, , \\
&\qquad -502 \ch^{12}+1098 \ch^{10}+1098 \ch^8-243 \ch^6-135 \ch^4-81 \ch^2-81\Big)   \, ,\\
&\psi_{4}^{[4]} :=  \frac{-21 \ch^{24}+60 \ch^{20}+343 \ch^{16}-1648 \ch^{12}+3177 \ch^8+756 \ch^4+405}{1536 \ch^{19} (\ch^4+5)} \, ,
\end{aligned}  \\ \label{expc4}
&c_4=\frac1{1024 \ch^{19}(\ch^2+1)}\Big(56 \ch^{30}+88 \ch^{28}-272 \ch^{26}-528 \ch^{24}-7 \ch^{22}+497 \ch^{20}+1917 \ch^{18}\\ \notag
&\quad +1437 \ch^{16}-4566 \ch^{14}-4038 \ch^{12}+4194 \ch^{10}+3906 \ch^8-891 \ch^6-675 \ch^4+81 \ch^2+81 \Big) \, .
\end{align} 
\end{subequations}
\end{prop}

The rest of this section is devoted to the proof  of Proposition \ref{expstokes}. 

In \cite[Theorem 2.1]{BMV3} we have yet computed the second order expansion 
of the 
Stokes waves in \eqref{exp:Sto}  and proved that  the coefficients 
$\eta_2^{[0]}$, $\eta_2^{[2]}$, $\psi_2^{[2]}$ 
are  \eqref{expcoef} 
and $c_2$  is \eqref{expc2}.
Then we write
 \begin{equation}\label{etapsic}
\begin{aligned}
 & \eta_\e(x) = \e \cos(x) + \e^2 (\eta_2^{[0]}+\eta_2^{[2]}\cos(2x)) + \e^3 \eta_3(x)+ \e^4 \eta_4(x)+\cO(\e^5) \, , \\
 &  \psi_\e(x) = \e \ch^{-1} \sin(x) + \e^2 \psi_2^{[2]}\sin(2x) + \e^3 \psi_3(x)+\e^4 \psi_4(x)+ \cO(\e^5) \, ,  
\\ &c_\e = \ch  + \e^2 c_2+ \e^3 c_3 +\e^4 c_4+ \cO(\e^5) \, ,  
\end{aligned}
\end{equation}
where the functions 
$\eta_3(x)$, $ \eta_4 (x) $  are even and $\psi_3(x) $, $\psi_4 (x) $ are odd.
We rewrite the equations \eqref{travelingWW} as the system 
\begin{equation}
\label{Sts} 
\begin{cases}
\eta -c \, \psi_x   + \dfrac{\psi_x^2}{2} - 
\dfrac{\eta_x^2}{2(1+\eta_x^2)} ( c  -  \psi_x )^2  = 0 \\
c \, \eta_x+G(\eta)\psi = 0 	\, ,
\end{cases}
\end{equation}
where  in the first equation we have replaced $G(\eta)\psi $ with $-c \, \eta_x $. 

We Taylor expand 
 the Dirichlet-Neumann operator $G(\eta)$ as 
 $$
 G(\eta)=  G_0+ G_1(\eta) + G_2(\eta)+ G_3(\eta) + \cO(\eta^4)  
 $$
where, by \cite[formulae (39)-(40)]{CN},
\begin{equation}
\begin{aligned}\label{expDiriNeu}
 & G_0  := D\tanh(\tth D) = |D| \tanh(\tth |D|) \,, \\
 &  G_1(\eta) := -\pa_x \eta \pa_x-G_0\eta G_0\,,\\
 &   G_2(\eta)  := -\frac12 G_0 \pa_x \eta^2 \pa_x+ \frac12 \pa_x^2 \eta^2 G_0 -G_0 \eta G_1(\eta) \, ,  \\ 
 &  G_3(\eta) :=\frac16 \pa_x^3  \eta^3 \pa_x + \frac16 G_0 \pa_x^2 \eta^3  G_0 - G_0  \eta G_2(\eta) +\frac12\pa_x^2  \eta^2 G_1(\eta)	\\
 &  G_4(\eta) := \frac{1}{24} G_0 \pa_x^3  \eta^4 \pa_x
 - \frac{1}{24} \pa_x^4  \eta^4 G_0  + \frac12 \pa_x^2 \eta^2 G_2 (\eta)  +
 \frac{1}{6} G_0 \pa_x^2 \eta^3 G_1 (\eta) - G_0 \eta G_3 (\eta ) 
 \,.
 \end{aligned}
 \end{equation}
 
\begin{rmk}
In order to check that \eqref{expDiriNeu} coincides with 
\cite[formulae (39)-(40)]{CN} use  the identity $D^2=-\pa_x^2 $.
We point out that \eqref{expDiriNeu} coincides with \cite[formulae (2.13)-(2.14)]{CS} and the recursion formulae of \cite[p. 24]{Wilkening}.
\end{rmk}

The `unperturbed" linear part of the system \eqref{Sts}
is associated with the self-adjoint closed operator 
\begin{equation}\label{defB0}
\cB_0= \begin{bmatrix} 1 & -\ch\pa_x \\ \ch\pa_x & G_0  \end{bmatrix} : \mathcal{D}\subset H^{\sigma,s}_{ev}\times H^{\sigma,s}_{odd} \longrightarrow H^{\sigma,s}_{ev}\times H^{\sigma,s}_{odd}\, ,
\end{equation}
with domain $\mathcal{D}:=H^{\sigma,s+1}_{ev}\times H^{\sigma,s+1}_{odd} $.
To compute the higher-order expansions one needs the following
\begin{lem}\label{aboutB0} 
The kernel of the operator $ \cB_0 $ in \eqref{defB0} is 
\begin{subequations}
\begin{equation}\label{KernelB0}
K:=
\text{Ker}\;\cB_0= 
\text{span}\,\Big\{\vet{\cos(x)}{\ch^{-1}\sin(x)} \Big\}
\end{equation}
and its range $R:=\text{Rn}\;\cB_0= K^{\perp_{\mathcal{D}}}$ is given by $R=R_0\oplus R_1\oplus R_\off $, where
\begin{equation}
\begin{aligned}
& \quad R_0 := \text{span}\,\Big\{\vet{1}{0} \Big\}\, ,\ \ R_1 := \text{span}\,\Big\{\vet{-\cos(x)}{\ch\sin(x)} \Big\}\, , \\
& R_\off := \overline{\bigoplus_{k=2}^\infty R_k}\, ,\ \ R_k := \text{span}\,\Big\{\vet{\cos(kx)}{0}\, ,\;\vet{0}{\sin(kx)} \Big\}\, .
\end{aligned}
\end{equation}
\end{subequations}
Consequently there exists a unique self-adjoint bounded linear operator $\cB_0^{-1}:R\to R $, given by
\begin{align} \label{cB0inv}
&\cB_0^{-1} \vet{1}{0} = \vet{1}{0}\, ,\quad \cB_0^{-1} \vet{-\cos(x)}{\ch \sin(x)} = \frac{1}{1+\ch^2}\vet{-\cos(x)}{\ch \sin(x)}\, , \\[2mm] \notag
&\cB_0^{-1} \vet{f(x)}{g(x)} = \big(|D|\tanh(\tth |D|)+\ch^2\pa_x^2 \big)^{-1}
\begin{bmatrix} |D|\tanh(\tth |D|) & \ch\pa_x \\ -\ch\pa_x & 1 \end{bmatrix} \vet{f(x)}{g(x)} \, , \ \forall \vet{f}{g}\in  R_\off\, , 
\end{align}
 such that $\cB_0\cB_0^{-1}=\cB_0^{-1}\restr{\cB_0}{R}=\uno_R $.
\end{lem}
\begin{proof}
This is the content of \cite[Lemma B.1]{BMV3}. 
Then \eqref{cB0inv}  follows  by inspection.
\end{proof}

From the second-order expansion of the Stokes waves in \cite[Appendix B]{BMV3} we recover 
\begin{equation}\label{G0psi1G1eta1psi1}
G_0\psi_1 = \ch \sin(x)\, ,\quad   G_1(\eta_1)\psi_1 
 =\frac{1-\ch^4}{\ch(1+\ch^4)}\sin(2x)\, .
\end{equation}
{\bf Third order in $ \e $.} By plugging 
the expansion \eqref{etapsic}  in system \eqref{Sts} and discarding quartic terms we find the following linear system (we drop the dependence w.r.t. $x$) 
\begin{equation}\label{syslin3}
 \cB_0 \vet{\eta_3}{\psi_3} = \vet{c_2(\psi_1)_x 
 - (\psi_1)_x (\psi_2)_x - (\eta_1)_x^2 (\psi_1)_x \ch + 
 (\eta_1)_x (\eta_2)_x \ch^2
 }{-c_2(\eta_1)_x - G_1(\eta_1)\psi_2- G_1(\eta_2)\psi_1 - G_2(\eta_1)\psi_1}=: \vet{f_3}{g_3}\, .
\end{equation}
In view of \eqref{exp:Sto} we have 
\begin{equation}\label{projder}
\begin{aligned}
&(\psi_1)_x (\psi_2)_x = \frac{\psi_2^{[2]}}{\ch}(\cos(x)+ \cos(3x))\, ,\\
&(\eta_1)_x^2 (\psi_1)_x \ch  = \frac14 (\cos(x)-\cos(3x))\, ,
\end{aligned} \quad \ch^2 (\eta_1)_x (\eta_2)_x = \ch^2 \eta_2^{[2]}(\cos(x)- \cos(3x))\, .
\end{equation}
By means of \eqref{expDiriNeu} and since
\begin{equation}\label{Tartagliate}\tanh(\tth)=\ch^2\, ,\quad \tanh(2\tth) = \frac{2\ch^2}{1+\ch^4}\, ,\quad \tanh(3\tth)= \frac{3\ch^2+\ch^6}{1+3\ch^4} \, ,
\end{equation}
whereby
\begin{equation}\label{G0acts2}
G_0\psi_2 = \frac{4\ch^2}{1+\ch^4}\psi_2^{[2]} \sin(2x)  \, ,\quad 
\end{equation}
we have, in view of \eqref{G0psi1G1eta1psi1} too,
\begin{equation}\label{projDN}
\begin{aligned}
&G_1(\eta_1)\psi_2 =\psi_2^{[2]} \frac{1-\ch^4}{1+\ch^4}\sin(x)+ 3\psi_2^{[2]} \frac{1-2\ch^4+\ch^8}{1+4\ch^4+3\ch^8}\sin(3x)\, , \\ &G_2(\eta_1)\psi_1=\frac{\ch}{4}\frac{3\ch^4-1}{1+\ch^4}\sin(x)-\frac34 \ch \frac{\ch^8-4\ch^4+3}{1+4\ch^4+3\ch^8}\sin(3x)\, , 
 \\ &G_1(\eta_2)\psi_1 =\frac{1}{\ch}\big(\eta_2^{[0]}(1-\ch^4)+\frac12\eta_2^{[2]}(1+\ch^4)\big)\sin(x)+ \frac3{2\ch} \eta_2^{[2]} \frac{1-\ch^8}{1+3\ch^4}\sin(3x)\, .  
\end{aligned}
\end{equation}
By \eqref{expc2}, \eqref{projder}, \eqref{projDN} and \eqref{expcoef} the right-hand side
 of system \eqref{syslin3} is given by
 \begin{equation}\label{RHS3}
 \begin{bmatrix}f_3(x)\\[2mm] g_3(x)\end{bmatrix} = \begin{bmatrix} f_3^{[1]}\cos(x)+f_3^{[3]}\cos(3x) \\[2mm] g_3^{[1]}\sin(x)+g_3^{[3]}\sin(3x) \end{bmatrix}\, ,
 \end{equation}
 with
 \begin{equation}
 \begin{aligned}
& f_3^{[1]}:= \frac{-2\ch^{12}+3\ch^8+3}{16\ch^8} \, ,\quad 
f_3^{[3]}:= \frac{3\ch^8-6\ch^4-3}{8\ch^8}\, , \\
& g_3^{[1]}:= \frac{2\ch^{12}-3\ch^{8}-3}{16\ch^7}  \, ,\quad g_3^{[3]}:=  \frac{-6\ch^8+15\ch^4-9}{4\ch^7(1+3\ch^4)}\, .
\end{aligned}
 \end{equation}
We recall that the term $c_2$ in \eqref{expc2} is determined (cfr. \cite[Appendix B]{BMV3}) to make the vector $(f(x),g(x))$ in \eqref{RHS3} orthogonal to the kernel of the operator $\cB_0$, therefore 
ensuring the existence of a solution 
$$
\vet{\eta_3(x)}{\psi_3(x)}:= \vet{\eta_3^{[1]}\cos(x)+\eta_3^{[3]}\cos(3x)}{\psi_3^{[1]}\sin(x)+\psi_3^{[3]}\sin(3x)}:=\cB_0^{-1}\vet{f_3(x)}{g_3(x)} \, .$$
By Lemma \ref{aboutB0} and \eqref{Tartagliate} we have
\begin{align}\label{syslin3tris}
 \vet{\eta_3^{[3]}\cos(3x)}{\psi_3^{[3]}\sin(3x)} =  \cB_0^{-1}\vet{\beta\cos(3x)
}{\delta\sin(3x)} &= -\frac{1+3\ch^4}{24\ch^6}  \begin{bmatrix} |D|\tanh(\tth |D|) & \ch\pa_x \\ -\ch\pa_x & 1 \end{bmatrix}\vet{f_3^{[3]} \cos(3x)}{g_3^{[3]} \sin(3x)} \\
&=  -\frac{1+3\ch^4}{24\ch^6} \vet{3\big(\frac{3\ch^2+\ch^6}{1+3\ch^4}f_3^{[3]}+\ch g_3^{[3]}\big)\cos(3x)}{(3\ch f_3^{[3]} + g_3^{[3]})\sin(3x)}\, ,
\end{align}
and $
\vet{\eta_3^{[1]}\cos(x)}{\psi_3^{[1]}\sin(x)} =  \cB_0^{-1}\vet{f_3^{[1]}\cos(x)
}{g_3^{[1]}\sin(x)} = \frac{1}{1+\ch^2}\vet{f_3^{[1]}\cos(x)
}{g_3^{[1]}\sin(x)}$.
The coefficients in \eqref{expcoefbis} follow. \medskip 
\\[1mm]
{\bf Fourth order in $ \e $.}
By plugging 
the expansion \eqref{etapsic}  in system \eqref{Sts} and discarding quintic terms we find the  linear system 
\begin{equation}\label{syslin4}
 \cB_0 \vet{\eta_4}{\psi_4} = \vet{f_4}{g_4}\, ,
\end{equation}
with 
\begin{subequations}\label{RHS4}
\begin{align}\label{RHS4a}
 f_4 &:= c_3(\psi_1)_x +c_2(\psi_2)_x 
 - (\psi_1)_x (\psi_3)_x - \frac12 (\psi_2)_x^2 + \ch (\eta_1)_x^2 \big(c_2-(\psi_2)_x\big) \\ \notag &\ \ + 
 (\eta_1)_x (\eta_3)_x \ch^2 +\frac12 (\eta_2)_x^2\ch^2+\frac12(\eta_1)_x^2(\psi_1)_x^2-\frac12\ch^2 (\eta_1)_x^4 -2\ch (\eta_1)_x(\eta_2)_x(\psi_1)_x\, ,\\
  \label{RHS4b}
 g_4 & := -c_3(\eta_1)_x-c_2(\eta_2)_x - G_1(\eta_1)\psi_3- G_1(\eta_2)\psi_2- G_1(\eta_3)\psi_1 \\ \notag 
  &\ \ - G_2'(\eta_1)[\eta_2]\psi_1-G_2(\eta_1)\psi_2- G_3(\eta_1)\psi_1 ,
 \end{align} 
 \end{subequations}
where, in view of \eqref{expDiriNeu}, 
\begin{equation}\label{G2prime}
G_2'(\eta)[\hat\eta] := -G_0 \pa_x \eta\hat\eta \pa_x +\pa_x^2 \eta\hat\eta G_0 - G_0 \hat\eta G_1(\eta) - G_0 \eta G_1(\hat\eta)\, .
\end{equation}
Let us inspect the terms in \eqref{RHS4a}. In view of \eqref{exp:Sto} we have
\begin{align}\label{f4pieces}
&c_3 (\psi_1)_x = \ch^{-1}c_3 \cos(x)\, , \quad c_3 (\eta_1)_x =- c_3 \sin(x) \, ,\quad c_2 (\psi_2)_x = 2c_2 \psi_2^{[2]} \cos(2x)\, ,\\ \notag &c_2(\eta_2)_x = -2c_2\eta_2^{[2]}  \sin(2x)\, ,\quad \frac12(\psi_2)_x^2 = (\psi_2^{[2]})^2 + (\psi_2^{[2]})^2\cos(4x)\, , \\ \notag
&(\psi_1)_x(\psi_3)_x =  \frac12 \ch^{-1}\psi_3^{[1]} +\frac12 \ch^{-1}\big(\psi_3^{[1]}+3\psi_3^{[3]}\big)\cos(2x)+ \frac32 \ch^{-1}\psi_3^{[3]}\cos(4x)\, ,  \\ \notag
&\ch (\eta_1)^2_x (\psi_2)_x = -\frac12 \ch\psi_2^{[2]}+\ch\psi_2^{[2]}\cos(2x) - \frac12 \ch\psi_2^{[2]} \cos(4x) \, , \\ \notag
&\frac12\ch^2 (\eta_2)_x^2 = \ch^2(\eta_2^{[2]})^2-\ch^2(\eta_2^{[2]})^2\cos(4x)\, ,\quad \ch c_2 (\eta_1)_x^2 = \frac12 \ch c_2 \big(1-\cos(2x)\big)\, , \\ \notag
&\ch^2 (\eta_1)_x (\eta_3)_x = \frac12 \ch^2 \eta_3^{[1]} + \frac12 \ch^2 \big(- \eta_3^{[1]} +3\eta_3^{[3]} \big) \cos(2x) - \frac32 \ch^2 \eta_3^{[3]} \cos(4x)\, , \\ \notag
&\frac12 (\eta_1)^2_x (\psi_1)^2_x = \frac{\ch^{-2}}{16} \big(1-\cos(4x)\big)\, ,\quad \frac12 \ch^2 (\eta_1)_x^4 = \frac{\ch^2}{16}  \big( 3 - 4\cos(2x) +  \cos(4x) \big) \, , \\ \notag
&2\ch (\eta_1)_x(\eta_2)_x(\psi_1)_x = \eta_2^{[2]}-\eta_2^{[2]}\cos(4x)\, .
\end{align}
Let us inspect the terms in \eqref{RHS4b}. In view  of  \eqref{expDiriNeu}, \eqref{etapsic}, \eqref{Tartagliate},  whereby
\begin{equation}\label{G0acts3}
 G_0\psi_3 = \ch^2 \psi_3^{[1]}\sin(x) + 3\ch^2 \frac{3+\ch^4}{1+3\ch^4} \psi_3^{[3]}\sin(3x)\, ,
\end{equation}
and since
\begin{equation}\label{bisTartagliate}
 \tanh(4\tth) = \frac{4\ch^2+4\ch^6}{1+6\ch^4+\ch^8}\, ,
\end{equation}
we have
\begin{equation}\label{g4pieces}
\begin{aligned}
G_1(\eta_1)\psi_3 &=\Big(\frac{1-\ch^4}{1+\ch^4}\psi_3^{[1]} +\frac{3(1-\ch^4)^2\psi_3^{[3]}}{(1+\ch^4)(1+3\ch^4)} \Big)\sin(2x)  + 6\frac{(1-\ch^4)^3\psi_3^{[3]}\sin(4x)}{(1+3\ch^4)(1+6\ch^4+\ch^8)}\, ,\\
G_1(\eta_2)\psi_2 &=4\frac{(1-\ch^4)^2}{(1+\ch^4)^2}\psi_2^{[2]}\eta_2^{[0]}\sin(2x)+ 4 \frac{(1-\ch^4)^2\psi_2^{[2]}\eta_2^{[2]}}{1+6\ch^4+\ch^8}\sin(4x)\, ,\\
G_1(\eta_3)\psi_1 &=\Big( \frac{1-\ch^4}{\ch(1+\ch^4)}\eta_3^{[1]} + \frac{1+3\ch^4}{\ch(1+\ch^4)} \eta_3^{[3]}  \Big) \sin(2x)+ 2\frac{(1-\ch^4)(1+3\ch^4)}{\ch(1+6\ch^4+\ch^8)}\eta_3^{[3]}\sin(4x)\, .
\end{aligned}
\end{equation}
In view of \eqref{exp:Sto}, \eqref{G2prime}, \eqref{Tartagliate}, \eqref{G0psi1G1eta1psi1}, \eqref{projDN} and  \eqref{bisTartagliate}, we have
\begin{equation}
\label{G2primeacts}
\begin{aligned}
&G_2'(\eta_1)[\eta_2]\psi_1  = \Big(\frac{8\ch^5\eta_2^{[2]}}{(1+3\ch^4)(1+\ch^4)}-\frac{4\ch(1-\ch^4)}{(1+\ch^4)^2}\eta_2^{[0]}\Big)\sin(2x) 
 - \frac{8\ch(1-\ch^8)\eta_2^{[2]}}{(1+3\ch^4)(1+6\ch^4+\ch^8)}\sin(4x)\, .
\end{aligned}
\end{equation}
By  \eqref{exp:Sto}, \eqref{expDiriNeu}, \eqref{projDN}, \eqref{Tartagliate} and \eqref{bisTartagliate} we have
\begin{equation}\label{G2eta1psi2}
G_2(\eta_1)\psi_2= -\frac{8\ch^2(1-\ch^4)\psi_2^{[2]}}{(1+\ch^4)^2(1+3\ch^4)}\sin(2x)-\frac{16\ch^2(1-\ch^4)^2\psi_2^{[2]}}{(1+\ch^4)(1+3\ch^4)(1+6\ch^4+\ch^8)}\sin(4x)\, .
\end{equation}
Finally, by \eqref{exp:Sto}, \eqref{expDiriNeu}, \eqref{projDN}, \eqref{G0psi1G1eta1psi1}, \eqref{Tartagliate} and \eqref{bisTartagliate}, we have
\begin{equation}\label{G3acts}
G_3(\eta_1)\psi_1 
= \frac{-1+14\ch^4-9\ch^8}{3\ch(1+\ch^4)^2(1+3\ch^4)}\sin(2x)
 +2 \frac{-1+15\ch^4-23\ch^8+9\ch^{12}}{3\ch(1+\ch^4)(1+3\ch^4)(1+6\ch^4+\ch^8)}\sin(4x)\, .
\end{equation}
By \eqref{RHS4}, \eqref{f4pieces}, \eqref{g4pieces}, \eqref{G2primeacts}, \eqref{G2eta1psi2}, \eqref{G3acts} and \eqref{expcoef}-\eqref{expcoefbis} system \eqref{syslin4} reads as
\begin{equation}\label{f4g4exp}
\begin{bmatrix} f_4(x) \\[2mm] g_4(x) \end{bmatrix}=\begin{bmatrix} f_4^{[0]} + \ch^{-1}c_3\cos(x) + f_4^{[2]}\cos(2x) +  f_4^{[4]}\cos(4x) \\[2mm] c_3\sin(x) + g_4^{[2]}\sin(2x) +  g_4^{[4]}\sin(4x) \end{bmatrix} \, ,
\end{equation}
with
\begin{align} \notag
f_4^{[0]} &= \frac{-4 \ch^{20}-4 \ch^{18}+17 \ch^{16}+6 \ch^{14}-48 \ch^8+6 \ch^6+36 \ch^4-9}{64 \ch^{14}}\, ,\ \ f_4^{[4]} = \frac{7 \ch^{16}-48 \ch^{12}+126 \ch^8-168 \ch^4-45}{128 \ch^{14}} \, ,\\ \notag
f_4^{[2]} &= \frac{4 \ch^{22}+12 \ch^{20}-27 \ch^{18}-31 \ch^{16}+78 \ch^{14}+66 \ch^{12}-72 \ch^{10}-84 \ch^8+6 \ch^6-6 \ch^4+27 \ch^2+27}{128 \ch^{14}(1+\ch^2)}\, ,\\ 
g_4^{[2]} &=\frac{3 \ch^{16}-12 \ch^{14}-39 \ch^{12}+18 \ch^{10}+139 \ch^8-225 \ch^4+18 \ch^2+18}{48\ch^9 (1+\ch^4)}\, , \label{f40tog44} \\ \notag
g_4^{[4]} &=\frac{-21 \ch^{20}+61 \ch^{16}+14 \ch^{12}-198 \ch^8+279 \ch^4-135}{48 \ch^{13} (\ch^8+6 \ch^4+1)}\, . 
\end{align}
As a consequence of \eqref{f4g4exp} we have
$$ \Big(\vet{f_4(x)}{g_4(x)}\, , \, \vet{\cos(x)}{\ch^{-1}\sin(x)} \Big) = \ch^{-1}c_3\, . $$
By Lemma \ref{aboutB0}, with $c_3=0$ as stated in \eqref{cexp}, one ensures that the system \eqref{syslin4} is solved by
\begin{equation}
\vet{\eta_4(x)}{\psi_4(x)} = \cB_0^{-1} \vet{f_4(x)}{g_4(x)}\, ,
\end{equation}
where, in view of \eqref{cB0inv} and \eqref{f4g4exp},
\begin{equation}\label{structordo4}
\begin{bmatrix} \eta_4(x) \\[2mm] \psi_4(x) \end{bmatrix}=\begin{bmatrix} \eta_4^{[0]}+ \eta_4^{[2]}\cos(2x) +  \eta_4^{[4]}\cos(4x) \\[2mm] \psi_4^{[2]}\sin(2x) +  \psi_4^{[4]}\sin(4x) \end{bmatrix}  \, ,
\end{equation}
 with, in view of \eqref{Tartagliate} and \eqref{bisTartagliate} too, 
\begin{align} \label{coeffordo4}
&\eta_4^{[0]} :=f_4^{[0]}\, , \quad \begin{bmatrix} \eta_4^{[2]}\\[2mm] \psi_4^{[2]} \end{bmatrix} = - \frac{1+\ch^4}{4\ch^6} \begin{bmatrix} \frac{4\ch^2}{1+\ch^4} & 2\ch \\ 2\ch & 1 \end{bmatrix} \begin{bmatrix} f_4^{[2]}\\[2mm] g_4^{[2]}  \end{bmatrix} = - \frac{1+\ch^4}{4\ch^6} \begin{bmatrix} \frac{4\ch^2}{1+\ch^4} f_4^{[2]} +2\ch g_4^{[2]} \\[2mm] 2\ch f_4^{[2]} +g_4^{[2]} \end{bmatrix} \, , \\ \notag
&\begin{bmatrix} \eta_4^{[4]}\\[2mm] \psi_4^{[4]} \end{bmatrix} = - \frac{1+6\ch^4+\ch^8}{16\ch^6(5+\ch^4)} \begin{bmatrix} 16\ch^2 \frac{1+\ch^4}{1+6\ch^4+\ch^8} & 4\ch \\ 4\ch & 1 \end{bmatrix} \begin{bmatrix} f_4^{[4]}\\[2mm] g_4^{[4]}  \end{bmatrix} =  - \frac{1+6\ch^4+\ch^8}{16\ch^6(5+\ch^4)} \begin{bmatrix} 16\ch^2 \frac{1+\ch^4}{1+6\ch^4+\ch^8} f_4^{[4]}+ 4\ch g_4^{[4]} \\[2mm] 4\ch f_4^{[4]} + g_4^{[4]}  \end{bmatrix} \, .
\end{align} 
By \eqref{structordo4} we conclude the proof of \eqref{etaexp}-\eqref{psiexp} and, in view of  \eqref{coeffordo4} and \eqref{f40tog44}, of \eqref{expcoeftris}. 
The proof of \eqref{exp:Sto}, \eqref{allcoefStokes} needs only the computation of $c_4$.
\\[1mm]
{\bf Fifth order in $ \e $.}
By plugging 
the expansion \eqref{etapsic}  in system \eqref{Sts} and discarding sextic terms we find the  linear system 
\begin{equation}\label{syslin5}
 \cB_0 \vet{\eta_5}{\psi_5} = \vet{{\mathtt f}_5}{{\mathtt g}_5} + c_4 \vet{(\psi_1)_x}{-(\eta_1)_x}\, ,
\end{equation}
with 
\begin{subequations}\label{RHS5}
\begin{align}\label{RHS5a}
{\mathtt f}_5 & := c_2(\psi_3)_x 
 - (\psi_1)_x (\psi_4)_x- (\psi_2)_x (\psi_3)_x+ (\eta_1)_x^2 \big(-\ch (\psi_3)_x-c_2 (\psi_1)_x+(\psi_1)_x(\psi_2)_x \big)\\ \notag
 &\ \  + (\eta_1)_x(\eta_2)_x \big((\psi_1)_x^2 + 2\ch c_2 - 2\ch (\psi_2)_x \big)-\ch (\psi_1)_x \big( (\eta_2)_x^2 - (\eta_1)_x^4 + 2(\eta_1)_x (\eta_3)_x \big) \\ \notag
 &\ \  + \ch^2 \big( (\eta_1)_x (\eta_4)_x + (\eta_2)_x (\eta_3)_x - 2(\eta_1)_x^3(\eta_2)_x \big)\, , \\
  \label{RHS5b}
 {\mathtt g}_5 & := 
 -c_2 (\eta_3)_x- G_1(\eta_1)\psi_4 - G_1(\eta_2)\psi_3 - G_1(\eta_3)\psi_2- G_1(\eta_4)\psi_1- G_2(\eta_1)\psi_3 \\  \notag
 &\ \ -G_2(\eta_2)\psi_1 - G_2'(\eta_1)[\eta_2,\psi_2] - G_2'(\eta_1)[\eta_3,\psi_1] - G_3(\eta_1)\psi_2 - G_3'(\eta_1)[\eta_2,\psi_1] - G_4(\eta_1)\psi_1\,,
 \end{align} 
 \end{subequations}
 where $G_2'$ is in \eqref{G2prime} and, by \eqref{expDiriNeu},
 \begin{equation}\label{G3prime}
 G_3'(\eta)[\hat\eta] = \frac12 \pa_x^3  \eta^2 \hat\eta \pa_x + \frac12 G_0 \pa_x^2 \eta^2 \hat\eta  G_0  - G_0  \hat\eta G_2(\eta) 
  - G_0  \eta G_2'(\eta)[\hat\eta] + \pa_x^2  \eta\hat\eta G_1(\eta) + \frac12 \pa_x^2  \eta^2 G_1(\hat\eta) \, .
 \end{equation}
The term $c_4$ is obtained by imposing the expression on the right-hand side of \eqref{syslin5} to be orthogonal to the kernel of the operator $\cB_0$ in \eqref{KernelB0}, obtaining, in view of \eqref{scalar},
\begin{equation}\label{c4conta}
c_4  = -\dfrac{\big( {\mathtt f}_5, \eta_1\big)+\big( {\mathtt g}_5, \psi_1\big)}{\big( (\psi_1)_x, \eta_1\big)-\big( (\eta_1)_x, \psi_1\big)} = -\ch \Big(\big( {\mathtt f}_5, \eta_1\big)+\big( {\mathtt g}_5, \psi_1\big)\Big) \, .
\end{equation}
By \eqref{RHS5} and \eqref{exp:Sto} we find that $c_4$ in \eqref{c4conta} has the explicit expression in \eqref{expc4}.

\begin{rmk}\label{check-controllo}
Expansion \eqref{exp:Sto}, \eqref{allcoefStokes} coincides with that in \cite[formulae (12)-(14)]{Fenton}, provided one
rescales properly their amplitude $ \varepsilon_{\rm Fen} = \e + f(\tth) \e^3 $ with a suitable
$f(\tth) $, 
translates their bottom to $ d:=\tth+\e^2 \eta_2^{[0]}+\e^4 \eta_4^{[0]}+\cO(\e^5) $
(in \cite{Fenton} the water surface $ \eta $ has zero average) and 
removes from the velocity potential  a 
shear term $-\overline{u}x$ (which corresponds to a Galilean reference frame).
\end{rmk}

\subsection{Expansions of  $ a_\e (x) $ and $ p_\e (x) $}\label{sec:exape}

In this section we compute the fourth order expansion of 
the functions $ a_\e (x) $ and $ p_\e (x) $ and of the constant 
$  \ttf_\e $ in \eqref{cLepsilon}. 

\smallskip

By \cite[formula (2.18)]{BMV3}) 
the functions $a_\e (x) $ and $p_\e (x) $ in \eqref{cLepsilon} 
are given by
\begin{align}\label{def:pa}
\ch+p_\e(x) :&=  \displaystyle{\frac{ c_\e-V(x+\mathfrak{p}(x))}{ 1+\mathfrak{p}_x(x)}} \, , \\ \notag
 1+a_\e(x):&=   \displaystyle{\frac{1+ (V(x + \mathfrak{p}(x)) - c_\e)
 B_x(x + \mathfrak{p}(x))  }{1+\mathfrak{p}_x(x)}}= \dfrac{1}{1+\mathfrak{p}_x(x)} - (\ch+p_\e(x))B_x(x+\mathfrak{p}(x)) \, ,
 \end{align}
and by \cite[formula (2.14)]{BMV3}) the constant $  \ttf_\e $ 
is 
\begin{equation} \label{def:ttf}
 \ttf_\e = \frac{1}{2\pi} \int_\bT \eta_\e (x +  \mathfrak{p}(x)) \de x \, .
  \end{equation}
Here
the function $\mathfrak p(x)$ is determined (cfr. \cite[formula (2.14)]{BMV3}, 
\cite[formula (A.15)]{BBHM}) 
by the fixed point equation 
\begin{equation}\label{fixedpoint}
\mathfrak{p} (x)  =  \frac{\cH}{\tanh \big((\tth + \ttf_\e)|D| \big)}[\eta_\e ( x + \mathfrak{p}(x))] \, , 
\end{equation} 
where $ {\mathcal H} := -\im \sgn(D) $ is the Hilbert transform\footnote{$ \sgn(D) $ is the Fourier multiplier operator 
 with symbol 
$ \sgn(k) := 1\  \forall k > 0 $, $  \sgn(0):=0 $, $  \sgn(k) := -1 \ \forall k < 0 $.}, 
whereas the functions $V$ and $B$  (cfr. \cite[formula (2.10)]{BMV3}) are given by
\begin{equation}\label{espVB}
V(x)  := -B(x) \eta'_\e(x) + \psi'_\e(x) \, , \quad  
 B(x) := 
\frac{ \psi'_\e(x)- c_\e}{1+(\eta'_\e(x))^2} \eta'_\e(x)
  \, . 
\end{equation}
To provide the Taylor expansion of the functions in \eqref{def:pa} we need some preparatory results.
\begin{lem} {\bf (Fourth order expansion of $ \mathfrak{p} (x) $ and $  \ttf_\e $)}\label{lem:pfrak}
The function $ \mathfrak{p} (x)  $ in \eqref{fixedpoint} admits the following Taylor expansion 
\begin{equation}\label{coeffnotipf}
\begin{aligned}
\mathfrak{p}(x) &  = \e \ch^{-2} \sin(x)  + \e^2 
 \mathfrak{p}_2^{[2]}\sin(2x)  + \e^3
 \big( \mathfrak{p}_3^{[1]}\sin(x) +\mathfrak{p}_3^{[3]}\sin(3x) \big) \\
& \quad   + \e^4 \big(  \mathfrak{p}_4^{[2]}\sin(2x) + \mathfrak{p}_4^{[4]}\sin(4x) \big)  
   +\cO(\e^5)
   \end{aligned}
\end{equation}
with coefficients 
\begin{align} \label{coeffnotipfbis}
 &  
 \mathfrak{p}_2^{[2]}:=\frac{3+4\ch^4+\ch^8}{8\ch^8}\, ,  \\
 &  \mathfrak{p}_3^{[1]} := \frac{4 \ch^{14}+2 \ch^{12}-17 \ch^{10}-14 \ch^8+10 \ch^6+10 \ch^4-15 \ch^2-12}{16 \ch^{10} (\ch^2+1)} \, ,\\
 & \label{frakp3}
 \mathfrak{p}_3^{[3]} := \frac{9+41\ch^4+43\ch^8 +3\ch^{12}}{64\ch^{14}} \, , \\ \label{coefffrakp2}
&\mathfrak{p}_4^{[2]}:= -\frac1{256 \ch^{20} (\ch^2+1)} \big(8 \ch^{24}-57 \ch^{22}-37 \ch^{20}+199 \ch^{18}+175 \ch^{16}+238 \ch^{14}\\ \notag 
&\qquad\qquad +190 \ch^{12}-130 \ch^{10}-178 \ch^8+171 \ch^6+135 \ch^4+27 \ch^2+27 \big) \, , \\
&\mathfrak{p}_4^{[4]}:=\frac{\ch^{24}+44 \ch^{20}+557 \ch^{16}+2528 \ch^{12}+3595 \ch^8+1332 \ch^4+135}{512 \ch^{20} (\ch^4+5)} \, . \label{coefffrakp4}
\end{align}
The real constant $  \ttf_\e $ in \eqref{cLepsilon} 
has the  Taylor expansion 
\begin{equation}\label{expfe0}
   \ttf_\e = \e^2 \tf_2 + \e^4 \tf_4 + \cO(\e^5) 
 \end{equation} 
 with coefficients 
 \begin{align}\label{expfe}
&\ttf_2:=\frac{\ch^4-3}{4\ch^2} \, ,\\ \notag
&\ttf_4:= \frac1{64 \ch^{14} (\ch^2+1)}\big(-4 \ch^{22}-8 \ch^{20}+5 \ch^{18}+23 \ch^{16}+40 \ch^{14}+22 \ch^{12}-78 \ch^{10}-72 \ch^8+72 \ch^6+54 \ch^4-27 \ch^2-27\big) \, .
\end{align}
\end{lem}

\begin{proof}
We expand
\begin{equation}\label{coeffnotipf1}
\begin{aligned}
& \mathfrak{p}(x)  = \e\mathfrak{p}_1(x) + \e^2 \mathfrak{p}_2(x) + \e^3 \mathfrak{p}_3(x)+ \e^4 \mathfrak{p}_4(x) +\cO(\e^5)\, , \\
& \ttf_\e = \e^2 \ttf_2 + \e^3 \ttf_3+ \e^4 \ttf_4 + \cO(\e^5) \, ,
\end{aligned}
\end{equation}
where, by \cite[formula (2.15)]{BMV3}, $ \mathfrak{p}_1(x)$, $ \mathfrak{p}_2(x) $
 are given in \eqref{coeffnotipfbis} and $\ttf_2  $ in \eqref{expfe}. 

Let us denote derivatives w.r.t $x$ with a prime $ ' $. 
By \eqref{def:ttf}, \eqref{exp:Sto}  and \eqref{coeffnotipf1}, we get
$$
\tf_3 =\frac1{2\pi} \int_\bT  \big( \eta_3(x)+\eta_2'(x)\mathfrak{p}_1(x) + \eta_1'(x)\mathfrak{p}_2(x)+\frac12 \eta_1''(x)\mathfrak{p}_1^2(x)\big) \de x = 0
$$
as stated in the expansion \eqref{expfe0}. 
In view of \eqref{fixedpoint} and \eqref{expfe} we have
\begin{equation}\label{pfrak3}
\begin{aligned}
\mathfrak{p}_3(x) &= - \im \frac{\sgn(D)}{\tanh(\tth |D|)} \Big( \eta_3(x) + \eta_2'(x) \mathfrak{p}_1(x) + \eta_1'(x) \mathfrak{p}_2(x) + \frac12 \eta_1''(x) \mathfrak{p}_1^2(x) \Big) \\
 &\qquad + \ttf_2 \frac{\pa_x}{\tanh^2(\tth |D|)} \big(1-\tanh^2(\tth |D|) \big)\eta_1(x) \, .
\end{aligned}
\end{equation}
In view of \eqref{pfrak3}, \eqref{exp:Sto}, \eqref{coeffnotipf}-\eqref{coeffnotipfbis} and \eqref{Tartagliate}, we have
\begin{equation} 
\begin{aligned}
\mathfrak{p}_3(x) &= \frac{16\eta_3^{[1]}\ch^8 - 16 \eta_2^{[2]} \ch^6 - 3 -6\ch^4 -\ch^8-16\ttf_2 \ch^6+16 \ttf_2 \ch^{10} {16\ch^{10}}}\sin(x) \\ 
&\qquad + \frac{(1+3\ch^4)(8\eta_3^{[3]}\ch^4 + 8 \eta_2^{[2]} \ch^2 + 4\ch^4\kp_2^{[2]}+1)}{8\ch^6(3+\ch^4)}\sin(3x)\, ,
\end{aligned}
\end{equation}
which, by \eqref{expcoef}-\eqref{expcoefbis}, is \eqref{frakp3}.
By \eqref{def:ttf}, \eqref{exp:Sto}, \eqref{coeffnotipfbis} and \eqref{frakp3}, we have
\begin{align}
\tf_4 &= \frac{1}{2\pi} \int_\bT \big( \eta_4(x) + \eta_3'(x) \mathfrak{p}_1(x) +\eta_2'(x) \mathfrak{p}_2(x) + \eta_1'(x) \mathfrak{p}_3(x)  \\  
&\quad +\frac12 \eta_2''(x) \mathfrak{p}_1^2(x)+ \eta_1''(x) \mathfrak{p}_2(x)\mathfrak{p}_1(x)+\frac16 \eta_1'''(x)\mathfrak{p}_1^3(x) \big) \de x \notag \\ \notag
&= \eta_4^{[0]} - \frac{\eta_3^{[1]}}{2\ch^2} - \mathfrak{p}_2^{[2]}\eta_2^{[2]} - \frac12 \mathfrak{p}_3^{[1]} +\frac{\eta_2^{[2]}}{2\ch^4} - \frac{\mathfrak{p}_2^{[2]}}{4\ch^2} + \frac{1}{16\ch^6} \, ,
\end{align}
which gives the fourth-order coefficient in \eqref{expfe}. Finally, by \eqref{fixedpoint} and \eqref{expfe},
\begin{align}
\mathfrak{p}_4(x) &= -\im \frac{\sgn(D)}{\tanh(\tth |D|)} \big( \eta_4(x) + \eta_3'(x) \mathfrak{p}_1(x) +\eta_2'(x) \mathfrak{p}_2(x) + \eta_1'(x) \mathfrak{p}_3(x) \notag \\  
&\qquad +\frac12 \eta_2''(x) \mathfrak{p}_1^2(x)+ \eta_1''(x) \mathfrak{p}_2(x)\mathfrak{p}_1(x)+\frac16 \eta_1'''(x)\mathfrak{p}_1^3(x) \big) \notag \\
&\qquad + \tf_2 \frac{\pa_x}{\tanh^2(\tth|D|)}\big( 1-\tanh^2(\tth|D|) \big) \big(\eta_2(x) + \eta_1'(x) \mathfrak{p}_1(x)\big) \notag \\
&=  \frac{1+\ch^4}{2\ch^2} \big(\frac12 \mathfrak{p}_3^{[1]}-\frac12 \mathfrak{p}_3^{[3]}-\frac{1}{12\ch^6}-\frac{ \eta_2^{[2]}}{\ch^4}+\frac{ \eta_3^{[1]}}{2\ch^2}-\frac{3 \eta_3^{[3]}}{2\ch^2}+ \eta_4^{[2]} \big) \sin(2x) \notag \\ \notag
&\qquad + \frac{1+6\ch^4+\ch^8}{4\ch^2(1+\ch^4)} \big(\frac{\mathfrak{p}_2^{[2]}}{4 \ch^2}+\frac{3\eta_3^{[3]}}{2\ch^2}+\mathfrak{p}_2^{[2]} \eta_2^{[2]}+\frac12 \mathfrak{p}_3^{[3]}+\frac{1}{48 \ch^6}+\frac{\eta_2^{[2]}}{2 \ch^4}+\eta_4^{[4]} \big) \sin(4x) \\
&\qquad  - \ttf_2 \frac{(1-\ch^4)^2(1+2\ch^2\eta_2^{[2]})}{4\ch^6} \sin(2x) \, ,\notag
\end{align}
which, by \eqref{expcoef}-\eqref{expcoefbis}, gives
\eqref{coeffnotipf} with the coefficients computed in  \eqref{coefffrakp2}-\eqref{coefffrakp4}.
\end{proof}

The second preparatory result is given by the following

\begin{lem}[Expansion of $B(x)$ and $V(x)$]\label{lemBV}
The functions $ B(x) $ and $ V(x) $ in \eqref{espVB} admit the following 
Taylor expansion
\begin{equation}\label{espVB2}
\begin{aligned}
 B(x)&= \e B_1(x)+ \e^2 B_2(x) + \e^3 B_3(x)+ \e^4 B_4(x)+\cO(\e^5)\, ,  \\ 
 V(x)&= \e V_1(x) + \e^2 V_2(x)+\e^3 V_3(x)+\e^4 V_4(x) + \cO(\e^5)
 \, , 
 \end{aligned}
\end{equation}
where 
\begin{equation}\label{espVB12}
\begin{alignedat}{3}
 B_1(x) &= \ch \sin(x) \, , \quad &&B_2(x) = B_2^{[2]}\sin(2x) \, ,\quad &B_2^{[2]}:=\frac{3-2\ch^4}{2\ch^5}\, ,
  \\ V_1(x) &=\ch^{-1} \cos(x) \, , \quad &&V_2(x) =\frac\ch2  +  V_2^{[2]}\cos(2x)\, ,\quad &V_2^{[2]}:=\frac{3-\ch^8}{4\ch^7}\, ,
 \end{alignedat}
 \end{equation} 
 and
 \begin{equation}\label{espBV34}
 \begin{alignedat}{2}
 &B_3(x) = B_3^{[1]}\sin(x)+ B_3^{[3]}\sin(3x)\, ,\quad &&B_4(x) = B_4^{[2]}\sin(2x)+ B_4^{[4]}\sin(4x)\, , \\
 &V_3(x) = V_3^{[1]}\cos(x)+ V_3^{[3]}\cos(3x)\, ,\quad &&V_4(x) = V_4^{[0]} +  V_4^{[2]}\cos(2x)+ V_4^{[4]}\cos(4x)\, ,
 \end{alignedat}
 \end{equation}
 with
 \begin{subequations}\label{coeffsBV34}
\begin{align}\notag
 &B_3^{[1]}:=\frac{6+3\ch^2-8\ch^4-8\ch^6+6\ch^8+3\ch^{10}-4\ch^{12}-2\ch^{14}}{16\ch^7(1+\ch^2)}\, ,\ \ B_3^{[3]}:= \frac{81-99\ch^4+43\ch^8-\ch^{12}}{64\ch^{11}}\, ,\\ \label{coeffsB3V3}
 &V_3^{[1]}:=\frac{2\ch^{12}-15\ch^{8}-12\ch^6+24\ch^4+24\ch^2-3}{16\ch^7 (1+\ch^{2})}\, ,\ \  V_3^{[3]}:= \frac{21\ch^{12}-39\ch^8+15\ch^4+27}{64\ch^{13}}\, ,
 \end{align}
 and
\begin{equation}\label{coeffsB4V4}
\begin{aligned}
V_4^{[0]}:&=\frac{-2 \ch^{18}-6 \ch^{16}+3 \ch^{14}+9 \ch^{12}-33 \ch^6-27 \ch^4+36 \ch^2+36}{32 \ch^{11} (\ch^2+1)} \, ,\\
B_4^{[2]}:&=\frac1{192 \ch^{17} (\ch^2+1)} \big(-24 \ch^{22}+24 \ch^{20}+354 \ch^{18}+210 \ch^{16}\\ 
&-943 \ch^{14}-835 \ch^{12}+927 \ch^{10}+855 \ch^8-81 \ch^6+27 \ch^4-81 \ch^2-81\big) \, ,\\
V_4^{[2]}:&= \frac1{384 \ch^{19} (\ch^2+1)} \big( 12 \ch^{26}+36 \ch^{24}-9 \ch^{22}-45 \ch^{20}+357 \ch^{18}+285 \ch^{16}\\
&-1060 \ch^{14}-988 \ch^{12}+1584 \ch^{10}+1584 \ch^8-243 \ch^6-135 \ch^4-81 \ch^2-81\big)  \, , \\
B_4^{[4]}:&=\frac{6 \ch^{20}-47 \ch^{16}-100 \ch^{12}+522 \ch^8-594 \ch^4+405}{96 \ch^{17} (\ch^4+5)}\, , \\
V_4^{[4]}:&= \frac{9 \ch^{24}-96 \ch^{20}-377 \ch^{16}+1484 \ch^{12}-1413 \ch^8+756 \ch^4+405}{384 \ch^{19} (\ch^4+5)} \, .
\end{aligned}
\end{equation}
\end{subequations}
\end{lem}
\begin{proof}
The first jets of  \eqref{espVB2} in \eqref{espVB12} were computed in \cite[(B.12-B.13)]{BMV3}. \\
On the other hand, in view of \eqref{espVB} and \eqref{exp:Sto}, \eqref{allcoefStokes}, the third order terms are
\begin{subequations}\label{littVB34}
\begin{align}
 B_3(x) &= -\ch \eta_3'(x) + \ch (\eta_1'(x))^3 + \psi_1'(x) \eta_2'(x) + (\psi_2'(x)-c_2)\eta_1'(x) \label{espB3}   \\   \notag
 &= \big(\ch \eta_3^{[1]} - \frac34 \ch - \frac{1}{\ch} \eta_2^{[2]} + \psi_2^{[2]} + c_2  \big)\sin(x) + \big( 3\ch \eta_3^{[3]} + \frac14 \ch - \frac{1}{\ch} \eta_2^{[2]} -\psi_2^{[2]} \big)\sin(3x) 
 \end{align}
 and 
 \begin{align}
 \label{espV3} 
 V_3(x) &= \psi_3'(x) - B_1(x) \eta_2'(x) - B_2(x) \eta_1'(x)\\ \notag
 &= \big(\psi_3^{[1]} + \ch \eta_2^{[2]} +\frac12 B_2^{[2]} \big)\cos(x) + \big( 3\psi_3^{[3]} -   \ch \eta_2^{[2]} -\frac12 B_2^{[2]} \big) \cos(3x) \, .
\end{align}
The fourth order terms are given by
\begin{align}  \notag
B_4(x) &= \psi_3'(x) \eta_1'(x) + \big(\psi_2'(x)-c_2+3\ch (\eta_1'(x))^2\big) \eta_2'(x) + \psi_1'(x) \eta_3'(x) - \psi_1'(x) (\eta_1'(x))^3 -\ch \eta_4'(x)  \, ,   \\ \notag
&= \big(\frac32 \psi_3^{[3]}-\frac12 \psi_3^{[1]}+2c_2 \eta_2^{[2]}-3\ch \eta_2^{[2]} - \frac1{2\ch} \eta_3^{[1]} - \frac3{2\ch} \eta_3^{[3]} + \frac1{4\ch}+2\ch\eta_4^{[2]} \big) \sin(2x)\\ \label{espB4V4}
&\!\!\!\!\!\!\!\!\!\!\!\!\!\!\!\!\,\,\begin{aligned}
&\quad\ \big(-\frac32 \psi_3^{[3]}- 2\eta_2^{[2]}\psi_2^{[2]}+\frac32 \ch \eta_2^{[2]} -\frac3{2\ch} \eta_3^{[3]} - \frac1{8\ch} +4\ch \eta_4^{[4]} \big) \sin(4x) \, ,\\ 
V_4(x)&= \psi_4'(x) -B_1(x)\eta_3'(x) - B_2(x) \eta_2'(x) -B_3(x) \eta_1'(x)
\end{aligned} \\ \notag
&=\frac12 \ch \eta_3^{[1]}+B_2^{[2]}\eta_2^{[2]}+\frac12 B_3^{[1]}+ \big(4\psi_4^{[4]}-\frac32 \ch\eta_3^{[3]} -B_2^{[2]} \eta_2^{[2]} -\frac12 B_3^{[3]} \big) \cos(4x) \\ \notag
&\quad +\big(2\psi_4^{[2]} -\frac12 \ch \eta_3^{[1]}+\frac32 \ch \eta_3^{[3]}-\frac12 B_3^{[1]} +\frac12 B_3^{[3]}\big)\cos(2x) \, .
\end{align}
\end{subequations}
From \eqref{littVB34} we obtain \eqref{espBV34} with the coefficients in \eqref{coeffsBV34}.
\end{proof}

We now provide the  fourth order expansion of the functions $p_\e (x) $ and $a_\e (x) $
in \eqref{cLepsilon}.

\begin{prop}\label{propaepe}
The functions $p_\e (x) $ and $a_\e (x) $ in \eqref{cLepsilon} have a Taylor expansion 
\begin{equation}\label{SN1}
\begin{aligned}
p_\e (x)  
&= \e p_1 (x) + \e^2 p_2 (x)+ \e^3 p_3(x) + \e^4 p_4(x) + \cO(\e^5) \, , \\
a_\e (x)  
&= \e a_1(x) +\e^2 a_2 (x) + \e^3 a_3(x)  + \e^4 a_4(x)  + \cO(\e^5) \, ,
\end{aligned} 
\end{equation}
with 
\begin{subequations}\label{apexp}
\begin{align}
\label{pino1fd}
 &p_1(x) = p_1^{[1]}\cos (x)\, , \quad p_1^{[1]} :=  - 2 \ch^{-1}\, ,\\ \notag
&p_2(x) = 
p_2^{[0]}+p_2^{[2]}\cos(2x)\, , \quad  
 p_2^{[0]} := 
 \frac{9+12 \ch^4+ 5\ch^8-2 \ch^{12}}{16 \ch^7}\, , \quad
  p_2^{[2]}:= - \frac{3+\ch^4}{2\ch^7}\, , 
 \\ \label{pino3}
  &p_3(x) = 
p_3^{[1]}\cos(x)+p_3^{[3]}\cos(3x)\, , \quad p_3^{[3]} := -\frac{\ch^{12}+17 \ch^8+51 \ch^4+27}{32\ch^{13}} \, , \\ \notag
 &p_3^{[1]} :=\frac{-2 \ch^{14}+14 \ch^{10}+11 \ch^8-10 \ch^6-10 \ch^4+24 \ch^2+21}{8 \ch^9 (\ch^2+1)} \, , 
\\ \label{pino4}
 &p_4(x) = p_4^{[0]}+p_4^{[2]}\cos(2x) + p_4^{[4]}\cos(4x) \, , \\ \notag
 &p_4^{[0]} :=\frac1{1024 \ch^{19} (\ch^2+1)} \big( 56 \ch^{30}+88 \ch^{28}-208 \ch^{26}-336 \ch^{24}+441 \ch^{22}+369 \ch^{20}-995 \ch^{18}\\ \notag
 &\qquad-899 \ch^{16}-630 \ch^{14}-294 \ch^{12}+1026 \ch^{10}+1314 \ch^8-27 \ch^6+189 \ch^4+81 \ch^2+81\big) \, ,\\ \notag
 &p_4^{[2]} :=  \frac1{64 \ch^{19} (\ch^2+1)}\big(-12 \ch^{22}-4 \ch^{20}-19 \ch^{18}-7 \ch^{16}+350 \ch^{14}\\ \notag
 &\qquad +314 \ch^{12}-256 \ch^{10}-268 \ch^8+198 \ch^6+162 \ch^4+27 \ch^2+27\big) \, ,  \\ \notag
 &p_4^{[4]} :=\frac{-\ch^{20}-39 \ch^{16}-366 \ch^{12}-850 \ch^8-657 \ch^4-135}{64 \ch^{19} (\ch^4+5)}\, ,
  \end{align}
  \end{subequations}
  and
  \begin{subequations} \label{expa04}
  \begin{equation}
 \begin{aligned} \label{aino2fd}
& a_1(x) 
= a_1^{[1]} \cos (x)\, , \quad
a_1^{[1]}:= - ( \ch^2 + \ch^{-2})\, , \\
& a_2(x)  = a_2^{[0]}+a_2^{[2]}\cos(2x)\, ,\quad   a_2^{[0]}:=\frac32 + \frac1{2\ch^4}\, , 
\quad a_2^{[2]} := 
 \frac{9\ch^8-14\ch^4-3}{4\ch^8}\, , 
 \end{aligned}
 \end{equation}
 \begin{align}
  &
  \begin{aligned}\label{aino3}
  & a_3(x) = 
a_3^{[1]}\cos(x)+a_3^{[3]}\cos(3x)\, ,  \quad a_3^{[3]}:= \frac{-\ch^{16}-98 \ch^{12}+252 \ch^8-318 \ch^4-27}{64\ch^{14}}\, ,\\
 & a_3^{[1]} := \frac{4 \ch^{18}+6 \ch^{16}-11 \ch^{14}-12 \ch^{12}-45 \ch^{10}-48 \ch^8+93 \ch^6+90 \ch^4+27 \ch^2+24}{16 \ch^{10} (\ch^2+1)}
\, , 
\end{aligned} \\  
& a_4(x) = a_4^{[0]}+a_4^{[2]}\cos(2x) + a_4^{[4]}\cos(4x) \, , \label{aino4} \\ 
&\begin{aligned}  \notag 
&a_4^{[0]}:=\frac{-12 \ch^{20}-31 \ch^{18}-17 \ch^{16}+40 \ch^{14}+46 \ch^{12}-150 \ch^{10}-132 \ch^8+84 \ch^6+90 \ch^4+9 \ch^2+9}{32 \ch^{16} (\ch^2+1)}\,, \\ 
&a_4^{[2]}:=\frac1{128 \ch^{20} (\ch^2+1)} \big(-72 \ch^{24}-431 \ch^{22}-211 \ch^{20}+1767 \ch^{18}\\ \notag
&\qquad +1623 \ch^{16}-2142 \ch^{14}-2070 \ch^{12}+1022 \ch^{10}+854 \ch^8+333 \ch^6+297 \ch^4+27 \ch^2+27\big)\, , \\ 
&a_4^{[4]}:= \frac{9 \ch^{24}+238 \ch^{20}-233 \ch^{16}-1676 \ch^{12}+743 \ch^8-3042 \ch^4-135}{128 \ch^{20} (\ch^4+5)}\, .
\end{aligned}
\end{align} 
\end{subequations}
\end{prop}

\begin{proof}
The first two jets $ p_1 (x), p_2(x) $  in \eqref{pino1fd}  and $ a_1 (x), a_2 (x) $ 
 in \eqref{aino2fd}  of the expansion \eqref{SN1} have been computed 
in \cite[Lemma 2.2]{BMV3}.  Let us compute the third order terms. 
In view of \eqref{def:pa}, \eqref{espVB2}  and \eqref{coeffnotipf}, we have
\begin{subequations}
\begin{align} \label{pino3preespanso}
p_3(x) &= \ch \big(-\mathfrak{p}_3'(x) +2\mathfrak{p}'_1(x)\mathfrak{p}'_2(x)- (\mathfrak{p}_1'(x))^3 \big)  + V_1(x) \big(\mathfrak{p}_2'(x) - (\mathfrak{p}_1'(x))^2 \big)  \\
& \qquad  - \big(c_2 - V_2(x) - V_1'(x) \mathfrak{p}_1(x) \big)\mathfrak{p}_1'(x) \notag \\
&\qquad  - V_3(x) - V_2'(x) \mathfrak{p}_1(x) - V_1'(x) \mathfrak{p}_2(x) -\frac12 V_1''(x)\mathfrak{p}_1^{{2}}(x) \notag \\ \notag
&= \big(-\ch \kp_3^{[1]} +\frac7{2\ch}\kp_2^{[2]}-\frac{13}{8} \ch^{-5} -\frac{c_2}{\ch^2}+\frac1{2\ch} +\frac{3V_2^{[2]}}{2\ch^2}-V_3^{[1]} \big)\cos(x)\\ \notag
&\qquad +\big(-3\ch \kp_3^{[3]} + \frac{5}{2\ch}\kp_2^{[2]} -\frac3{8\ch^5} - \frac{V_2^{[2]}}{2\ch^2}-V_3^{[3]}\big)\cos(3x)\, ,
\\[2mm] \label{aino3preespanso}
a_3(x) &= 
-\kp_3'(x)+2\kp_1'(x)\kp_2'(x) - (\kp_1'(x))^3 \\ \notag
&\qquad - \ch \big(B_3'(x) + B_2''(x) \kp_1(x) + B_1''(x) \kp_2(x) +\frac12 B_1'''(x) \kp_1^2(x) \big) \\ \notag
&\qquad -p_1(x) \big(B_2'(x) + B_1''(x) \kp_1(x) \big) - p_2(x) B_1'(x)\, , \\ \notag
&= \big(-\kp_3^{[1]}+2\ch^{-2} \kp_2^{[2]} -\frac34 \ch^{-6} - \ch B_3^{[1]}+2\ch^{-1} B_2^{[2]} +\frac12 \ch^2 \kp_2^{[2]}  \\ \notag
&\qquad+\frac18 \ch^{-2}-p_1^{[1]} B_2^{[2]} +\frac14 \ch^{-1} p_1^{[1]} -\ch p_2^{[0]} -\frac12 \ch p_2^{[2]} \big)\cos(x) \\ \notag
&\qquad +\big(-3\kp_3^{[3]}+2\ch^{-2} \kp_2^{[2]} -\frac14 \ch^{-6} -3\ch B_3^{[3]} -2\ch^{-1}B_2^{[2]}\\ \notag
&\qquad-\frac12 \ch^2 \kp_2^{[2]} -\frac18 \ch^{-2} -p_1^{[1]} B_2^{[2]} -\frac14 \ch^{-1} p_1^{[1]} -\frac12 \ch p_2^{[2]} \big) \cos(3x)\, ,
\end{align}
\end{subequations}
and
\begin{subequations}
\begin{align} \label{pino4preespanso}
p_4(x) &= c_4 - V_4(x) - V_3'(x)\mathfrak{p}_1(x) - V_2'(x) \mathfrak{p}_2(x) -V_1'(x)\mathfrak{p}_3(x)\\
&\qquad	-\frac12 V_2''(x)\mathfrak{p}^2_1(x) - V_1''(x) \mathfrak{p}_2(x)\mathfrak{p}_1(x) - \frac16 V_1'''(x) \mathfrak{p}_1^3(x) \notag \\
&\qquad +\mathfrak{p}_1'(x) \big( V_3(x) + V_2'(x) \mathfrak{p}_1(x) + V_1'(x) \mathfrak{p}_2(x) + \frac12 V_1''(x) \mathfrak{p}_1^2(x)  \big)  \notag \\
&\qquad + \big( (\mathfrak{p}_1'(x))^2 -\mathfrak{p}_2'(x) \big) \big( c_2 -V_2(x) - V_1'(x) \mathfrak{p}_1(x)\big) \notag \\
&\qquad + V_1(x) \big( \mathfrak{p}_3'(x) - 2 \mathfrak{p}_2'(x)\mathfrak{p}_1'(x) + (\mathfrak{p}_1'(x))^3\big) \notag  \\
&\qquad +\ch \big( -\mathfrak{p}_4'(x)+ 2\mathfrak{p}_3'(x)\mathfrak{p}_1'(x) + (\mathfrak{p}_2'(x))^2 -3\mathfrak{p}_2'(x) (\mathfrak{p}_1'(x))^2  +  (\mathfrak{p}_1'(x))^4  \big)   \, , \notag \\  \notag
&=\frac{3}{4}\ch^{-7}+\frac12 c_2 \ch^{-4}-\frac54\ch^{-4} V_2^{[2]}-2\ch^{-3}\kp_2^{[2]}-\frac14 \ch^{-3}+\ch^{-2}V_3^{[1]}+2 \ch  (\kp_2^{[2]})^2\\ \notag
&\  +2\ch^{-1}\kp_3^{[1]}+2 \kp_2^{[2]} V_2^{[2]}+c_4-V_4^{[0]} +\big(\frac{13}{12}\ch^{-7}+\frac12 \ch^{-4}c_2+\frac12 \ch^{-4}V_2^{[2]}-6\ch^{-3}\kp_2^{[2]} \\ \notag
&\ -\frac14 \ch^{-3}+2\ch^{-2}V_3^{[3]}+\ch \kp_2^{[2]}+\ch^{-1}\kp_3^{[1]}+5\ch^{-1}\kp_3^{[3]}-2\ch \kp_4^{[2]}-2 \kp_2^{[2]} c_2-V_4^{[2]} \big) \cos(2x)  \\ \notag
&\ +\big(\frac16 \ch^{-7}-\frac14 \ch^{-4}V_2^{[2]}-2\ch^{-3}\kp_2^{[2]}-\ch^{-2}V_3^{[3]}+2 \ch (\kp_2^{[2]})^2+4 \ch^{-1}\kp_3^{[3]}-4 \ch \kp_4^{[4]}-V_4^{[4]} \big) \cos(4x)
\\[2mm]
\label{aino4preespanso}
a_4(x)&=-\kp_4'(x) + (\kp_2'(x))^2+2\kp_1'(x)\kp_3'(x) - 3(\kp_1'(x))^2\kp_2'(x) + (\kp_1'(x))^4 \\ \notag
&\qquad-\ch \big(B_4'(x) + B_3''(x) \kp_1(x) + B_2''(x) \kp_2(x) + B_1''(x) \kp_3(x) \\ \notag
&\qquad + \frac12 B_2'''(x) \kp_1^2(x) + B_1'''(x) \kp_1(x) \kp_2(x) + \frac16 B_1^{\scriptscriptstyle{\mathit{IV}}}(x) \mathfrak{p}_1^3(x) \big) \\ \notag
&\qquad -p_1(x) \big(B_3'(x) + B_2''(x) \kp_1(x) + B_1''(x) \kp_2(x) +\frac12 B_1'''(x) \kp_1^2(x) \big) \\ \notag
&\qquad -p_2(x) \big(B_2'(x) + B_1''(x) \kp_1(x) \big) - p_3(x) B_1'(x)\\ \notag
&=\frac38 \ch^{-8}-\frac1{16}\ch^{-4}-\frac32 \ch^{-4}\kp_2^{[2]}+ \frac1{16}\ch^{-3}p_1^{[1]}-\ch^{-3}B_2^{[2]}+\ch^{-2}B_2^{[2]}p_1^{[1]}+ \ch^{-2}\kp_3^{[1]}\\ \notag
&\ +\frac12 \ch^2 \kp_3^{[1]}+2\ch \kp_2^{[2]}B_2^{[2]}+\frac14 \ch \kp_2^{[2]}p_1^{[1]}-\frac12 \ch p_3^{[1]}+\frac12 \ch^{-1} B_3^{[1]}+\frac12 \ch^{-1}p_2^{[0]}-\frac14 \ch^{-1}p_2^{[2]}\\ \notag
&\ + 2 (\kp_2^{[2]})^2+\frac14 \kp_2^{[2]}- B_2^{[2]} p_2^{[2]}-\frac12 B_3^{[1]} p_1^{[1]} + \big(\frac12 \ch^{-8}-3\ch^{-4}\kp_2^{[2]}+\frac1{12}\ch^{-4}+2\ch^{-3}B_2^{[2]}\\ \notag
&\  -\frac12 \ch^2 \kp_3^{[1]}+ \ch^{-2}\kp_3^{[1]}+\frac12 \ch^2 \kp_3^{[3]}+3\ch^{-2}\kp_3^{[3]}-\frac12 \ch^{-1}B_3^{[1]}+\frac92 \ch^{-1}B_3^{[3]}-2\ch B_4^{[2]}-\frac12 \ch^{-1}p_2^{[0]}\\ \notag
&\ +\frac12 \ch^{-1} p_2^{[2]}-\frac12\ch p_3^{[1]}-\frac12 \ch p_3^{[3]}-2B_2^{[2]}p_2^{[0]}-\frac12 B_3^{[1]}p_1^{[1]}-\frac32 B_3^{[3]}p_1^{[1]}-2\kp_4^{[2]} \big)\cos(2x)\\ \notag
&+ \big(\frac18 \ch^{-8}-\frac32 \ch^{-4}\kp_2^{[2]}-\frac1{48} \ch^{-4}-\ch^{-3}B_2^{[2]}-\frac1{16} \ch^{-3} p_1^{[1]}-\ch^{-2}B_2^{[2]}p_1^{[1]}-\frac12 \ch^2 \kp_3^{[3]}\\ \notag
&\ +3\ch^{-2}\kp_3^{[3]}-2\ch \kp_2^{[2]} B_2^{[2]}-\frac14 \ch \kp_2^{[2]}p_1^{[1]}-\frac92 \ch^{-1}B_3^{[3]}-4\ch B_4^{[4]}-\frac14 \ch^{-1} p_2^{[2]} \\ \notag
&\ -\frac12 \ch p_3^{[3]}+2 (\kp_2^{[2]})^2 -\frac14 \kp_2^{[2]}-B_2^{[2]}p_2^{[2]}-\frac32 B_3^{[3]}p_1^{[1]}-4 \kp_4^{[4]}\big)\cos(4x)  
\, .
\end{align}
\end{subequations}
The expansions of $p_3(x)$, $a_3(x)$, $p_4(x)$ and $a_4(x)$ in \eqref{pino1fd}-\eqref{aino2fd} descend from \eqref{pino3preespanso}, \eqref{aino3preespanso}, \eqref{pino4preespanso} and \eqref{aino4preespanso} respectively, in view of 
\eqref{coefffrakp4}, \eqref{frakp3} \eqref{coeffnotipfbis},
 \eqref{espVB12},  \eqref{expc2}, \eqref{espV3}, \eqref{espB3} and \eqref{espB4V4}-\eqref{coeffsB4V4}. 
 \end{proof}
 
\begin{rmk}
The functions $V$ and $B$ in \eqref{espVB} coincide respectively with the horizontal and vertical derivative of the velocity potential $\Phi$ in \cite[formula (12)]{Fenton} after the rescaling procedure outlined in Remark \ref{check-controllo}. 
\end{rmk}

\section{Expansion of the projector $ P_{\mu,\e} $ }\label{conticinoleggeroleggero}

In this Appendix we prove Lemmata \ref{lem:secondjetsP} and \ref{lem:P03acts}.

We denote, for any $ k \in \bN $, 
\begin{equation}
\begin{aligned}
\label{fksigma}
& f_k^+:=\vet{\ch^{1/2}\cos(kx)}{\ch^{-1/2}\sin(kx)} \, ,
\quad  f_k^- :=\vet{-\ch^{1/2}\sin(kx)}{\ch^{-1/2}\cos(kx)} \, , \\
&  f_{-k}^+ :=\vet{\ch^{1/2}\cos(kx)}{-\ch^{-1/2}\sin(kx)}\, ,
\quad 
f_{-k}^- :=\vet{\ch^{1/2}\sin(kx)}{\ch^{-1/2}\cos(kx)}  \, ,
\end{aligned}  
\end{equation}
and we define for any $k\in \bZ $ the spaces
\begin{equation}\label{cWk}
\cW_k:= \text{span}\left\{f_k^+, \ f_k^-, f_{-k}^+, \ f_{-k}^- \right\} \, ,\quad \cW_k^\sigma := \text{span}_\bR \{f_k^\sigma,\ f_{-k}^\sigma\}\, ,\quad \sigma=\pm\, .
\end{equation}
We have the following 
\begin{lem}
 The jets of the operator $\cB_{\mu,\e}$ in \eqref{Bsani} act on the spaces in \eqref{cWk} as follows\footnote{the sum is direct if $j\leq k$, otherwise some spaces may  overlap.}
\begin{equation} \label{Bjkrev}
\cB_{\ell,j} \cW_k^\sigma = \underbrace{\im^\ell \cW_{k-j}^{(-1)^{\ell} \sigma}+_\bR \im^\ell \cW_{k-j+2}^{(-1)^{\ell} \sigma} +_\bR \dots  +_\bR \im^\ell \cW_{k+j}^{(-1)^{\ell} \sigma} }_{j+1 \ terms}\, ,\quad f_0^- \notin \cB_{0,j} \cW_k^-\, ,
\end{equation}
with $\ell,j=0,\dots,4 $, while the operator $\cJ$ in \eqref{PoissonTensor} acts as 
$\cJ  \cW_k^\pm =  \cW_k^\mp $. 
\end{lem}

\begin{proof}
The first formula in \eqref{Bjkrev} follows by \eqref{Bsani}-\eqref{Pi_odd}. 
Let us prove the second statement by contradiction supposing 
that there exists $g\in \cW_k^-$ such that $\cB_{0,j} g = f_0^- $. Then
$1 =\big(f_0^- , f_0^- \big) = \big(\cB_{0,j} g , f_0^- \big) = \big( g , \cB_{0,j} f_0^- \big)  
= 0 $, by \eqref{Bsani}, which is a contradiction.
\end{proof}
We now include an extended version of \cite[Lemma A.2]{BMV3}. 
\begin{lem}\label{lem:VUW}
The space $ H^1(\bT) $ decomposes as 
$
H^1(\bT) =  \cV_{0,0} \oplus \cU \oplus {\cW_{H^1}}
${, with  $\cW_{H^1}=\overline{\bigoplus\limits_{k\geq 2}^{\textcolor{white}{,}} \cW_k}^{_{\footnotesize \begin{matrix} H^1 \end{matrix} }}$}
where the subspaces $\cV_{0,0}, \cU $ and $ \cW_k $, defined below, are 
invariant  under   $\sL_{0,0} $ and  the following properties hold:
\begin{itemize}
\item[(i)] $ \cV_{0,0} = \text{span} \{ f^+_1, f^-_1, f^+_0, f^-_0\}$  is the 
generalized kernel of $\sL_{0,0}$. For any $ \lambda \neq 0 $ the operator 
$ \sL_{0,0}-\lambda :  \cV_{0,0} \to \cV_{0,0} $ is invertible and  
\begin{subequations}\label{primainversione}
 \begin{align}\label{primainversione1}
& (\sL_{0,0}-\lambda)^{-1}f_1^+ = -\frac1\lambda f_1^+ \, ,
\  
(\sL_{0,0}-\lambda)^{-1}f_1^- = -\frac1\lambda f_1^-,
\    (\sL_{0,0}-\lambda)^{-1}f_0^- = -\frac1\lambda f_0^- \, ,  \\
& \label{primainversione2}
(\sL_{0,0}-\lambda)^{-1}f_0^+ = -\frac1\lambda f_0^+ + \frac{1}{\lambda^2} f_0^- \, .
\end{align} 
\item[(ii)] $\cU := \text{span}\left\{ f_{-1}^+, f_{-1}^-  \right\}$.   For any 
$ \lambda \neq \pm 2 \im $ the operator 
$ \sL_{0,0}-\lambda :  \cU \to \cU $ is invertible and
\begin{equation}
\label{primainversione3}
\begin{aligned}
&  (\sL_{0,0}-\lambda)^{-1} f_{-1}^+ =
   \frac{1}{\lambda^2+4 \ch^2}\left(-\lambda f_{-1}^+ + 2 \ch f_{-1}^-\right),  \\
&  (\sL_{0,0}-\lambda)^{-1} f_{-1}^- 
  = \frac{1}{\lambda^2+4\ch^2}
  \left(-2 \ch f_{-1}^+ - \lambda f_{-1}^-\right) \, .
  \end{aligned}
\end{equation}
\item[{(iii)}]  Each
subspace $\cW_k$ in \eqref{cWk} is  invariant under $ \sL_{0,0} $.  
For any
$|\lambda| < 
\delta(\tth)$ small enough and any natural $k\geq 2$, the operator 
$ {\sL_{0,0}-\lambda :  \cW_k\to \cW_k }$ 
is invertible and for any $f \in \cW_k$ and any natural number $N$
\begin{align}
\label{primainversione4}
 (\sL_{0,0}-\lambda)^{-1} f  =  \sL_{0,0}^{\, \inv} f + \lambda \big( \sL_{0,0}^{\, \inv}\big)^2 f + \dots +  \lambda^{N-1} \big( \sL_{0,0}^{\, \inv}\big)^{N} f + \lambda^N \varphi_{f,N}(\lambda, x) \, ,
\end{align}
\end{subequations}
for some analytic  function  $\lambda \mapsto \varphi_{f,N}(\lambda, \cdot) \in 
\cW_k $, where $\sL_{0,0}^{\, \inv} : \cW_k\to \cW_k$ is
\begin{equation}\label{L0inv} {\footnotesize
\sL_{0,0}^{\, \inv}:= \begin{matrix} \big(\ch^2 \pa_x^2 + |D|\tanh(\tth |D|)\big)^{-1} \end{matrix} \begingroup 
\setlength\arraycolsep{2pt}\begin{bmatrix} \ch \partial_x & - |D|\tanh(\tth |D|) \\ 1 & \ch \partial_x\end{bmatrix} \endgroup  , \quad \sL_{0,0}^{\, \inv} \cW^\pm_k = \cW^\mp_k \, . }
\end{equation}
\end{itemize}
\end{lem}

\begin{rmk}
We will use in the sequel the following 
decomposition formula, for any, 
\begin{equation}\label{inoutkernel}
\begin{aligned}
&{\footnotesize \begin{matrix} \vet{\mathtt{a} \cos(x) }{\mathtt{b} \sin(x)} \end{matrix} } = \frac12 (\ta \ch^{-\frac12}+\tb \ch^{\frac12}) f_1^+ + \frac12 (\ta \ch^{-\frac12}-\tb \ch^{\frac12}) f_{-1}^+ \, ,\\ 
&{\footnotesize \begin{matrix} \vet{\mathtt{a} \sin(x) }{\mathtt{b} \cos(x)} \end{matrix}} = \frac12(\tb \ch^{\frac12}-\ta \ch^{-\frac12}) f_1^- + \frac12 (\tb \ch^{\frac12}+\ta \ch^{-\frac12})  f_{-1}^- \, ,
 \end{aligned}\quad \forall \ta\, ,\tb\in\bC\, .
\end{equation}
\end{rmk}

\noindent {\bf Notation.} We denote by $\cO(\lambda)$ 
an analytic function having a zero of order $1$ at $\lambda=0$ and
$\cO_Z (\lambda^m)$ 
an analytic function with valued in a subspace $ Z $
having a zero of order $ m $ at $\lambda=0$. 
 We denote with $\cO(\lambda^{-1}:\lambda) $ any function having a Laurent series at $\lambda=0$ of the form
$ \sum_{\substack{j\in\bZ \setminus \{0\}}} a_j \lambda^j $. We denote by $ f_{{\cal W}_k}$
a function in $ {\cal W}_k $. 

If $h(\lambda) = h_0 +  \cO(\lambda^{-1}:\lambda)$, $h_0\in \bC$, then, by the residue theorem,
\begin{equation}\label{residuethm}
\frac1{2\pi\im} \oint_\Gamma \frac{h(\lambda)}{\lambda} \de\lambda = h_0\, .
\end{equation}

We prepend to the proof of Lemma \ref{lem:secondjetsP} a list of results given by straight-forward computations.

\begin{lem}[Action of $(\sL_{0,0}- \lambda)^{-1} \cJ  \cB_{0,1}$ on $\cV_{0,0}$, $\cU$ and $\cW_2$]
\label{lem:action1}
One has 
\begin{equation}\label{risolventeintermedio}
\begin{aligned} 
&(\sL_{0,0}- \lambda)^{-1} \cJ  \cB_{0,1} f_1^+= \frac{\zeta_1^+}{\lambda} f_0^- + \mathrm{A}_2^+
 +\lambda \mathrm B_2^- + \lambda^2 f_{\cW_2^+} + 
\lambda^3 f_{\cW_2^-} +  
 \cO_{\cW_2}(\lambda^4)\, , \\ 
&(\sL_{0,0}- \lambda)^{-1} \cJ  \cB_{0,1} f_1^-= \mathrm{A}_2^- + \lambda  \mathrm{B}_2^+ +  \lambda^2f_{\cW_2^-} + 
\lambda^3 f_{\cW_2^+} + 
\cO_{\cW_2}(\lambda^4)  \, ,  \\
  &(\sL_{0,0}- \lambda)^{-1} \cJ  \cB_{0,1} f_0^+=
  \frac{ \zeta_0^+}{\lambda} f_1^- + \frac{\alpha_0^+}{ \lambda^2+4\ch^2} f_{-1}^+
   + \lambda  \frac{\beta_0^+ }{\lambda^2+4\ch^2} f_{-1}^- \, , \\
    &(\sL_{0,0}- \lambda)^{-1} \cJ  \cB_{0,1} f_0^-= 0  \, ,  \\
     &  (\sL_{0,0}-\lambda)^{-1}\cJ \cB_{0,1} f^+_{-1}
 =  \frac{\zeta_{-1}^+ }{ \lambda } f_0^- +\mathrm{A}_{-2}^+  + \lambda f_{\cW_2^-}+\cO_{\cW_2} (\lambda^2)
 \, , \\
 & (\sL_{0,0}-\lambda)^{-1}\cJ \cB_{0,1} f^-_{-1}  =
 f_{\cW_2^-} + \cO_{\cW_2}(\lambda) \, , 
\end{aligned}
\end{equation}
where  $\zeta^+_1, \zeta^+_0$, $\alpha_0^+, \beta_0^+, \zeta_{-1}^+$ are real numbers, 
 and 
\begin{align}
\notag
&\mathrm{A}_2^+(x) := {\footnotesize \sepvet{\dfrac{-a_1^{[1]} \ch + (2 + \ch^4) p_1^{[1]}}{ 2 \ch^{9/2}} \cos(2 x)}{-\dfrac{(\ch^4+1) (a_1^{[1]} \ch-2 p_1^{[1]})}{4 \ch^{11/2}}  \sin (2 x)}}\,,\ \mathrm B_2^-(x) := {\footnotesize \sepvet{\dfrac{(\ch^4+1) \big((\ch^4+4) p_1^{[1]}-2 a_1^{[1]} \ch\big)}{4 \ch^{19/2}} \sin (2 x)}{\dfrac{(\ch^4+1)  \big(a_1^{[1]} \ch (\ch^4+2)-(3 \ch^4+4) p_1^{[1]} \big)}{8 \ch^{21/2}}\cos(2x) } }\, , \\ \notag
&\mathrm A_2^-(x):={\footnotesize \sepvet{\dfrac{ \big(a_1^{[1]} \ch- (\ch^4+2) p_1^{[1]}\big)}{2\ch^{9/2}}\sin(2x) }{ -\dfrac{(\ch^4+1)(a_1^{[1]}\ch-2 p_1^{[1]})}{4 \ch^{11/2}}  \cos (2 x) }} \, ,\ \ \mathrm{B}_2^+ (x):=   {\footnotesize \sepvet{\dfrac{(\ch^4+1) \big((\ch^4+4) p_1^{[1]}-2 a_1^{[1]} \ch\big)}{4 \ch^{19/2}}\cos (2 x)}{-\dfrac{(\ch^4+1) \big(a_1^{[1]} \ch (\ch^4+2)-(3 \ch^4+4) p_1^{[1]}\big)}{8 \ch^{21/2}}\sin(2x)}} \, , \\  \label{risolventeintermedio.1}
 & \mathrm{A}_{-2}^+(x):= {\footnotesize
 \sepvet{\dfrac{ (\ch^3 p_1^{[1]}-a_1^{[1]} )}{2 \ch^{7/2}} \cos(2x) }{-\dfrac{a_1^{[1]} (\ch^4+1) }{4 \ch^{9/2}}\sin(2x)}}
 \, .
\end{align}
Moreover 
\begin{align}\notag
& (\sL_{0,0}-\lambda)^{-1}\cJ \cB_{0,1}  \mathrm A_2^+ =
 \frac{\zeta_2^+}{\lambda} f_1^- + \frac{\alpha_2^+}{\lambda^2+4\ch^2} f_{-1}^+ + 
 \mathrm A_3^+ +\frac{\lambda \beta_2^+}{\lambda^2 + 4 \ch^2} f_{-1}^- +\lambda f_{\cW_3^-} + \cO_{ \cW_3}(\lambda^2)  \, ,\\ \notag
& (\sL_{0,0}-\lambda)^{-1}\cJ \cB_{0,1}  \mathrm B_2^- = \frac{\zeta_3^+}{\lambda} f_1^+ +  \frac{\alpha_3^+}{\lambda^2+4\ch^2} f_{-1}^- + f_{\cW_3^-} +  
\frac{\lambda \beta_3^+}{\lambda^2 + 4 \ch^2} f_{-1}^+ 
+
 \cO_{ \cW_3}(\lambda)  \,  ,\\ \notag
& (\sL_{0,0}-\lambda)^{-1}\cJ \cB_{0,1}  \mathrm A_2^- = 
\frac{\zeta_2^-}{\lambda} f_1^+ + \frac{\alpha_2^-}{\lambda^2+4\ch^2} f_{-1}^- + f_{\cW_3^-} +  \frac{\lambda \beta_2^-  }{\lambda^2 + 4\ch^2} f_{-1}^+
 + \lambda f_{\cW_3^+} 
+ \cO_{\cW_3}(\lambda^2)
\, ,\\
& (\sL_{0,0}-\lambda)^{-1}\cJ \cB_{0,1}  \mathrm B_2^+ =
 \frac{\zeta_3^-}{\lambda} f_1^-   
 + \frac{\alpha_3^-}{\lambda^2+4\ch^2} f_{-1}^+ + f_{\cW_3^+}  
 + \frac{\lambda \beta_3^-}{\lambda^2 + 4\ch^2} f_{-1}^- 
 + \cO_{\cW_3} (\lambda)\, ,
\end{align}
where $\zeta_2^- $, $\beta_2^\pm$,  $\alpha_3^\pm$  and $\beta_3^{\pm} $ are real numbers
 and
\begin{align}\label{2104_1446}
& \zeta_2^+ := -\frac{(a_1^{[1]})^2 \ch^2-2 a_1^{[1]} (\ch^4+2) \ch p_1^{[1]}+(3 \ch^4+4) (p_1^{[1]})^2}{8 \ch^5}   \, , 
\\  \notag
&\alpha_2^+ := -\alpha_2^- := -\frac{(a_1^{[1]})^2 \ch-2 a_1^{[1]} (\ch^4+1) p_1^{[1]}+\ch^3 (p_1^{[1]})^2}{4 \ch^3 } \, , 
\\ \notag
&\zeta_3^+ : = \zeta_3^- := \frac{(\ch^4+1) (a_1^{[1]} \ch-2 p_1^{[1]}) \big( (\ch^4+2) p_1^{[1]}-a_1^{[1]} \ch \big)}{8 \ch^{10}}\, ,
  \\ \notag
& \mathrm A_3^+(x) := 
 \sepvet{\dfrac{  (a_1^{[1]})^2 (\ch^4+3) \ch^2-2\ch p_1^{[1]} a_1^{[1]} (\ch^8+9 \ch^4+6) + (11 \ch^8+29 \ch^4+12 ) (p_1^{[1]})^2 }{32 \ch^{19/2}}\cos (3 x)}{ \dfrac{(3 \ch^4+1) \big((a_1^{[1]})^2 \ch^2-2 \ch p_1^{[1]} a_1^{[1]} (\ch^4+2) + (3 \ch^4+4) (p_1^{[1]})^2\big)}{32 \ch^{21/2}} \sin (3 x)} 
 \, .
\end{align}
\end{lem}

\begin{proof}
Use the operator $\cB_{0,1}$ in \eqref{B1sano}, Lemma \ref{lem:VUW} and that
\begin{equation*}
\frac1{2\tanh(2\tth) - 4\ch^2} = - \frac{1+\ch^4}{4\ch^6}\, ,\quad \frac1{3\tanh(3\tth) - 9\ch^2} = -\frac{1+3\ch^4}{24\ch^6}\, ,
\end{equation*}
which comes from the identities in \eqref{Tartagliate}.
\end{proof}

\begin{lem}[Action of $(\sL_{0,0}-\lambda)^{-1} \cJ\cB_{1,0}$ on $\cV_{0,0} $ and $\cU$]
\label{lem:action10} 
One has 
\begin{equation}\label{murisolventeinterno}
\begin{aligned}
&(\sL_{0,0}-\lambda)^{-1} \cJ\cB_{1,0}f_1^+ = - \im \frac{\mu_\tth}{ \lambda} f_1^+ + \im \frac{ 2\ch \mu_\tth}{\lambda^2 + 4 \ch^2} f_{-1}^- - \im  \frac{\lambda \mu_\tth}{\lambda^2 + 4 \ch^2} f_{-1}^+   \, , \\ 
&(\sL_{0,0}-\lambda)^{-1}\cJ\cB_{1,0}f_1^-= - \im \frac{ \mu_\tth}{\lambda} f_1^- + \im \frac{ 2 \ch \mu_\tth}{\lambda^2 + 4 \ch^2}  f_{-1}^+ 
+\im \lambda \frac{ \mu_\tth}{\lambda^2 + 4 \ch^2}  f_{-1}^-   \, , \\
& (\sL_{0,0}-\lambda)^{-1}\cJ\cB_{1,0}f_{0}^\pm=0\, , \\
&(\sL_{0,0}-\lambda)^{-1} \cJ\cB_{1,0}f_{-1}^+ =  \im \frac{\mu_\tth}{ \lambda} f_1^+ 
- \im  \frac{2\ch  \mu_\tth}{\lambda^2 + 4 \ch^2} f_{-1}^- 
+ \im \frac{ \lambda \mu_\tth}{\lambda^2 + 4 \ch^2} f_{-1}^+ 
  \, , \\ 
&(\sL_{0,0}-\lambda)^{-1}\cJ\cB_{1,0}f_{-1}^-=  -\im \frac{ \mu_\tth}{\lambda} f_1^- 
+ \im \frac{ 2 \ch \mu_\tth}{\lambda^2 + 4 \ch^2}  f_{-1}^+ 
+ \im \lambda \frac{ \mu_\tth}{\lambda^2 + 4 \ch^2}  f_{-1}^- \, , 
\end{aligned}
\end{equation}
with 
\begin{equation}\label{muh}
\mu_\tth:=\dfrac{(\ch^4-1) \tth-\ch^2}{2 \ch} \, . 
\end{equation}
\end{lem}
\begin{proof} We apply the operator
 ${\cal B}_{1,0} = \ell_{1,0}(|D|) \Pi_{\mathtt{s}}$ in  \eqref{B1sano}
 to the vectors  in $ \cV_{0,0} $  and use 
\eqref{inoutkernel} and \eqref{primainversione1}-\eqref{primainversione3}. 
\end{proof}

\begin{lem}[Action of $ (\sL_{0,0}  -\lambda)^{-1} \cJ\cB_{0,2}$ on $\cV_{0,0} $]
\label{lem:action02}
One has  
\begin{align}\notag
&  (\sL_{0,0}  -\lambda)^{-1} \cJ {\cal B}_{0,2} f_1^+ = 
\frac{\tau_1^+}{\lambda} f_1^- + \frac{\ell_{1}^+}{\lambda^2+4\ch^2} f_{-1}^+  + \mathrm L_3^+ 
 + \frac{\lambda m_1^+}{\lambda^2+4\ch^2} f_{-1}^- +  O_{\cW_3}(\lambda)  \, , \\ \label{1904_1424} 
&  (\sL_{0,0}  -\lambda)^{-1} \cJ{\cal B}_{0,2} f_1^- =  
\frac{\tau_1^-}{\lambda} f_1^+ + \frac{\ell_{1}^-}{\lambda^2+4\ch^2} f_{-1}^-  + f_{\cW_3^-}   + \cO_{\cU\oplus \cW_3}(\lambda)   \, , 
 \\  \notag
&   (\sL_{0,0}  -\lambda)^{-1} \cJ\cB_{0,2} f_{0}^+ = 
\frac{\tau_0^+}{\lambda} f_0^- + f_{\cW_2^+}+ \cO_{\cW_2}(\lambda)\, , \quad    (\sL_{0,0}  -\lambda)^{-1} \cJ\cB_{0,2} f_{0}^- = 0  \, ,
 \end{align}
 where $m_1^+$, $\tau_1^-$, $\tau_0^+$, are real numbers,
    and 
 \begin{align}  \notag
& \tau_1^+ :=\frac1{4 \ch} \big(2 a_2^{[0]} \ch^2+a_2^{[2]} \ch^2+2 \ttf_2(1-\ch^4)-4 \ch p_2^{[0]}-2 \ch p_2^{[2]} \big) \, , \\ \notag
  &\ell_{1}^+ :=\frac12 \big(\ch^2 (2 a_2^{[0]}+a_2^{[2]})-2 \ttf_2 (1-\ch^4) \big)\, , \quad \ell_{1}^- := \frac12 \big( \ch^2 (-2a_2^{[0]}+a_2^{[2]})+2 \ttf_2 (1-\ch^4) \big) \, , \\ \label{1904_1446} 
&  \mathrm L_3^+(x) := 
 \sepvet{-\dfrac{ \big(a_2^{[2]}\ch (\ch^4+3)-2 (5 \ch^4+3) p_2^{[2]})}{16 \ch^{9/2}} \cos (3 x)}{-\dfrac{(3 \ch^4+1) (a_2^{[2]} \ch-2 p_2^{[2]})}{16 \ch^{11/2}}\sin (3 x)} \,  . 
 \end{align}
\end{lem}

\begin{proof}
We apply the operators $\cB_{0,2}$ in \eqref{B2sano} and $\cJ$ in \eqref{PoissonTensor}   to the vectors  in $ \cV_{0,0} $. Then we use \eqref{inoutkernel} and Lemma \ref{lem:VUW} to obtain \eqref{1904_1424}-\eqref{1904_1446}.
\end{proof}

\begin{lem}[Action of $(\sL_{0,0}- \lambda)^{-1}\cJ \cB_{1,1}$ on $\cV_{0,0} $ and $\cU$]
\label{lem:action11}
One has 
\begin{equation}
\label{P11acts}
\begin{aligned}
& (\sL_{0,0}- \lambda)^{-1}\cJ \cB_{1,1}  f_1^+  
 =  
- \frac{\im \ch^{-\frac12}}{\lambda^2} f_0^-
+\frac{\im \ch^{-\frac12}}{\lambda} f_0^+
  + \im  f_{\cW_2^-} + \cO_\cW(\lambda)  \, ,
\\
&  (\sL_{0,0}- \lambda)^{-1}\cJ \cB_{1,1} f_1^-   = 
 \frac{\im \ch^{-\frac32}}{\lambda} f_0^- 
 +\im \mathrm Q_{2}^++ \lambda  \im f_{\cW_2^-} + 
  \cO_{\cW_2}(\lambda^2) \, ,
\\
&  (\sL_{0,0}- \lambda)^{-1}\cJ \cB_{1,1} f_0^+ =  \frac{ \im \ch^{-\frac32} }{\lambda} f_1^+ -\frac{2\im \ch^{-\frac12}}{\lambda^2+4\ch^2} f_{-1}^- + \lambda \frac{\im \ch^{-\frac32}}{\lambda^2+4\ch^2} f_{-1}^+ \, , \\
&  (\sL_{0,0}- \lambda)^{-1}\cJ \cB_{1,1} f_0^- = 
\frac{ \im \ch^{-\frac12}}{\lambda} f_1^- +   \frac{2\im \ch^{\frac12}}{\lambda^2+4\ch^2}f_{-1}^+
+ \lambda\frac{\im \ch^{-\frac12}}{\lambda^2+4\ch^2} f_{-1}^- \, ,
\end{aligned} 
\end{equation}
where 
\begin{align}
\label{1804_1855}
&
\mathrm Q_{2}^+(x)
:=   \sepvet{\dfrac{ (\ch^4+3) p_1^{[1]} }{4 \ch^{9/2}} \cos(2x)}{ \dfrac{3(\ch^4+1) p_1^{[1]} }{8 \ch^{11/2}} \sin(2x) } \, .
\end{align}
\end{lem}
\begin{proof}
We have  $\cB_{1,1} = - \im p_1(x) \cJ$  by \eqref{B2sano}, with $p_1(x) = p_1^{[1]} \cos(x)$ in \eqref{pino1fd} and $p_1^{[1]} = -2\ch^{-1}$. Use also Lemma \ref{lem:VUW}.
\end{proof}

\begin{lem}[Action of $ (\sL_{0,0}- \lambda)^{-1} \cJ\cB_{2,0}$ on $f_0^-$]
\label{lem:action20}
 One has
\begin{equation}\label{1804_1819}
 (\sL_{0,0}- \lambda)^{-1} \cJ\cB_{2,0} f_0^-  = 
 \frac{\tth}{\lambda^2} f_0^- -\frac{\tth}{\lambda} f_0^+ \, .
\end{equation}
\end{lem}

\begin{proof} We apply ${\cal B}_{2,0} = \ell_{2,0}(|D|) \Pi_{\mathtt {ev}}$ in \eqref{B2sano}
and 
\eqref{primainversione2}. 
\end{proof}

\begin{lem}[Action of $ ((\sL_{0,0}  - \lambda)^{-1}  \cJ {\cal B}_{0,1})^2$ on $\cV_{0,0}$] 
\label{lem:B012}
One has
\begin{align}  
&\begin{aligned}\label{2104_1749}
&   [(\sL_{0,0}  - \lambda)^{-1}  \cJ \cB_{0,1} ]^2 f_1^+  
 = 
\frac{\zeta_2^+ }{\lambda}  f_1^-   + \frac{\alpha_2^+}{\lambda^2+4\ch^2} f_{-1}^+ 
+\zeta_3^+ f_1^+ + \mathrm A_2^+ \\
& \qquad \qquad \qquad \qquad \qquad 
\qquad        +\lambda \big( f_{\cW_1^-}
 +f_{\cW_3^-}
  \big) +
\lambda^2  f_{\cW_1^+} 
  + \cO_{\cW_3}(\lambda^2)   + \cO(\lambda^3) \, ,
 \\   
&  [(\sL_{0,0}  - \lambda)^{-1}  \cJ \cB_{0,1} ]^2 f_1^-   = 
\frac{\zeta_2^-}{\lambda} f_1^+ 
 + \frac{\alpha_2^-}{\lambda^2+4\ch^2} f_{-1}^- 
+\zeta_3^- f_1^- + f_{\cW_3^-}   +\lambda 
 \big(  f_{\cW_1^+ }
 + f_{\cW_3^+}  \big) \\
 & \qquad \qquad \qquad \qquad 
 \qquad \qquad + \lambda^2  f_{\cW_1^-} + \cO_{\cW_3} (\lambda^2)  + \cO(\lambda^3)\, ,
\end{aligned} 
 \\     \notag
&  [(\sL_{0,0}- \lambda)^{-1} \cJ \cB_{0,1}]^2  f_0^+  = 
\zeta_0^+ \mathrm B_2^+ + \frac{\alpha_0^+}{4 \ch^2} \mathrm{A}_{-2}^+
 + \cO(\lambda^{-1}:\lambda) \, ,\ \  [(\sL_{0,0}- \lambda)^{-1} \cJ \cB_{0,1}]^2  f_0^- = 0\, .
\end{align}
\end{lem}
\begin{proof}
Apply twice Lemma \ref{lem:action1} and use that
\begin{align*}
& \lambda^2 \, (\sL_{0,0}  - \lambda)^{-1}  \cJ \cB_{0,1} \, f_{\cW_2^\pm} 
= \lambda \alpha_5^\pm  f_1^\mp +  
\lambda^2 f_{\cW_1^\pm}  
+ \cO_{\cW_3}(\lambda^2) + \cO(\lambda^3) \ , \\
& \lambda^3 \, (\sL_{0,0}  - \lambda)^{-1}  \cJ \cB_{0,1} \ f_{\cW_2^\mp}  = \lambda^2  \alpha_6^\pm f_1^\pm  + \cO(\lambda^3)  \ , \\
&  (\sL_{0,0}  - \lambda)^{-1}  \cJ \cB_{0,1} \cO_{\cW_2}(\lambda^4) = \cO(\lambda^3)
\end{align*}
where $\alpha_5^\pm$, $\alpha_6^\pm$ are real numbers.
 \end{proof}

We further list a series of identities to exploit later.

By applying first Lemma \ref{lem:action10} and then Lemma \ref{lem:action1} we get 
\begin{align} \notag
&(\sL_{0,0}-\lambda)^{-1} \cJ \cB_{0,1} (\sL_{0,0}-\lambda)^{-1} \cJ \cB_{1,0}f_1^+  = 
- \im \mu_\tth \Big[
\frac{ \zeta_1^+}{\lambda^2} f_0^-  
+ \frac{1}{\lambda}\mathrm A_2^+ + f_{\cW_2^-} + \alpha_7 f_0^- + \cO_{\cW_2}(\lambda) + \cO_{\cW_0}(\lambda^2)  \Big]  \, , 
 \\ \notag
& {(\sL_{0,0}-\lambda)^{-1}\cJ\cB_{0,1} (\sL_{0,0}-\lambda)^{-1} \cJ\cB_{1,0}f_1^- 
 = -\im \mu_\tth 
 \Big[ 
 \frac{1}{\lambda} \Big( \mathrm{A}_2^- - \frac{2\ch \zeta_{-1}^+}{\lambda^2+ 4 \ch^2} f_0^- \Big) + \mathrm{J}_2^+ + \lambda f_{\cW_2^-}+\cO_{\cW_2}(\lambda^2)}\Big]  \, ,\\
  \label{1904_1129}
&(\sL_{0,0}-\lambda)^{-1}\cJ\cB_{0,1}(\sL_{0,0}-\lambda)^{-1} \cJ\cB_{1,0}f_0^\pm  = 0\, ,
\end{align}
where $\alpha_7$ is a real number and 
\begin{equation}\label{JK}
\mathrm J_2^+(x):= \mathrm B_2^+ (x) - \frac{1}{2\ch} \mathrm{A}_{-2}^+ (x)\stackrel{\eqref{risolventeintermedio.1}}{=} \sepvet{\dfrac{(5 \ch^4+4) p_1^{[1]}-a_1^{[1]} \ch (\ch^4+2)}{4 \ch^{19/2}}\cos(2x)}{\dfrac{(\ch^4+1) \big((3 \ch^4+4) p_1^{[1]}-2 a_1^{[1]}\ch\big) }{8 \ch^{21/2}}\sin (2 x)} \, .
\end{equation}
By applying first Lemma \ref{lem:action1} and then Lemma \ref{lem:action10}, and since $\cJ \cB_{1,0} \cW_k \subseteq \cW_k$, we get 
\begin{equation}\label{conti2}
\begin{aligned}
&(\sL_{0,0}-\lambda)^{-1}\cJ\cB_{1,0} (\sL_{0,0}-\lambda)^{-1} \cJ\cB_{0,1}f_1^+
=   \im f_{\cW_2^-}+\cO_{\cW_2}(\lambda) \, ,\\ 
&(\sL_{0,0}-\lambda)^{-1}\cJ\cB_{1,0} (\sL_{0,0}-\lambda)^{-1} \cJ\cB_{0,1}f_1^-
= \im \mathrm S_2^+  +\im 
\lambda f_{\cW_2^-} + \cO_{\cW_2}(\lambda^2) \, , \\ 
 &(\sL_{0,0}-\lambda)^{-1} \cJ\cB_{1,0} (\sL_{0,0}- \lambda)^{-1} \cJ  \cB_{0,1} f_0^+ = \im f_{\cW_1^-}
  +  \cO(\lambda^{-1}:\lambda)  \, ,\\
 &(\sL_{0,0}-\lambda)^{-1} \cJ\cB_{1,0} (\sL_{0,0}- \lambda)^{-1} \cJ  \cB_{0,1} f_0^-  = 0\, , 
 \end{aligned}
 \end{equation}
 where, using also \eqref{L0inv}, \eqref{B1sano}--\eqref{ell10}, \eqref{risolventeintermedio.1}
 \begin{equation}\label{S1-}
\im \mathrm S_2^+ (x) := (\sL_{0,0})^{\, \inv} \cJ \cB_{1,0} 
\mathrm{A}_2^-  = \im  \sepvet{
 \dfrac{ (\ch^8 \tth+\ch^6-2 \ch^4 \tth+\ch^2+\tth) (a_1^{[1]} \ch-2 p_1^{[1]})}{4 \ch^{21/2}} \cos (2 x) }{
 \dfrac{ (\ch^8 \tth+\ch^6-2 \ch^4 \tth+\ch^2+\tth)  (a_1^{[1]} \ch-2 p_1^{[1]})}{8 \ch^{23/2}} \sin (2 x) } 
 \, .
 \end{equation}
 
 We are now in  position to prove  Lemmata \ref{lem:secondjetsP} and \ref{lem:P03acts}.
\\[1mm]
\noindent{\it Proof of Lemma \ref{lem:secondjetsP}.} The proof is divided in three parts, one for each group of formulas in \eqref{secondjetsP}.

\noindent {\bf Computation of $P_{0,2} f_j^\sigma$}. 
Since $P_{0,\e}f_0^-=f_0^-$ (cfr. \cite[Lemma A.4]{BMV3}) we have 
\begin{equation}\label{P02f0m}
P_{0,2} f_0^- = 0\, . 
\end{equation}
On the other hand, for $f_j^\sigma\in\{f_1^+, f_1^-,f_0^+ \}$,
in view of \eqref{Psani} we have, by \eqref{primainversione1}-\eqref{primainversione2} and since $\cB_{0,1} f_0^-=\cB_{0,2} f_0^- =0$,
 \begin{align} \label{IeII}
P_{0,2} f_j^\sigma  &= 
 -\frac{1}{2\pi \im} \oint_{\Gamma} \frac{(\sL_{0,0}- \lambda)^{-1}}{\lambda} \cJ \cB_{0,2}  f_j^\sigma \de \lambda \\ \notag
& +\frac{1}{2\pi \im} 
 \oint_{\Gamma} \frac{(\sL_{0,0}- \lambda)^{-1}}{\lambda} \cJ \cB_{0,1} (\sL_{0,0}- \lambda)^{-1} \cJ  \cB_{0,1} f_j^\sigma \de \lambda 
 =: \mathrm{I}_j^\sigma + \mathrm{I\!I}_j^\sigma \, .
\end{align}
In case $f_j^\sigma=f_0^+$ one readily sees, in view of \eqref{1904_1424} for $\mathrm{I}_0^+$ and \eqref{2104_1749} for $\mathrm{I\!I}_0^+$, that $P_{0,2}f_0^+\in \cW_2^+ $ which implies the second statement in \eqref{secondjetsPbis}.
We now compute the remaining four terms. 		

First   by Lemma \ref{lem:action02} and the residue theorem 
\begin{equation}
\mathrm{I}_1^+ =  -\frac{\ell_1^+}{{4\ch^2}} f_{-1}^+ - \mathrm L_3^+  \, , 
\qquad 
 \mathrm{I}_1^- =  -\frac{\ell_1^-}{{4\ch^2}} f_{-1}^-  + f_{\cW_3^-}  \, .
\end{equation}
Then 
by \eqref{2104_1749} and the residue theorem
\begin{equation}
\mathrm{I\!I}_1^+  = \frac{\alpha_2^+}{4\ch^2} f_{-1}^+ + \zeta_3^+ f_1^+ + 
\mathrm{A}_3^+ \,  , 
\qquad
\mathrm{I\!I}_1^-  = \frac{\alpha_2^-}{4\ch^2} f_{-1}^- + \zeta_3^- f_1^- + f_{\cW_{3}^-} \ . 
\end{equation}
In conclusion we have formulae \eqref{secondjetsPbis}, \eqref{secondjetsPeps} with 
\begin{equation*}
 P_{0,2} f_1^+ =  \underbrace{\frac{\alpha_2^+ - \ell_1^+  }{4\ch^2} }_{\ku_{0,2}^+} f_{-1}^+ + \underbrace{\zeta_3^+}_{\kn_{0,2}} f_1^+ 
 \underbrace{ - \mathrm{L}_3^+ + \mathrm{A}_3^+}_{\scriptsize \vet{\ka_{0,2} \cos(3x)}{\kb_{0,2} \sin(3x)}} , 
 \quad
  P_{0,2} f_1^- = \underbrace{ \frac{\alpha_2^- - \ell_1^- }{4\ch^2} }_{\ku_{0,2}^-} f_{-1}^- + \underbrace{\zeta_3^-}_{\kn_{0,2}} f_1^-    + f_{\cW_{3}^-}\, , 
\end{equation*}
and we obtain the explicit expression \eqref{u02-}  of the coefficients given by
\begin{align*}
&\kn_{0,2}:= \zeta_3^+\, ,\ \ \ku_{0,2}^+ := \frac{\alpha_2^+-\ell_1^+}{4\ch^2}\, ,\ \ \ku_{0,2}^- := \frac{\alpha_2^--\ell_1^-}{4\ch^2}\, ,\ \
{\footnotesize \vet{\ka_{0,2}\cos(3x)}{\kb_{0,2}\sin(3x)}} := \mathrm A_3^+(x)  -\mathrm L_3^+(x) \, ,
\end{align*}
with $\zeta_3^+$, $\alpha_2^\pm $, $\mathrm A_3^+$ in  \eqref{2104_1446} and $\ell_1^\pm$, $\mathrm L_3^+$ in  \eqref{1904_1446}. 

\noindent  {\bf Computation of $P_{2,0} f_0^-$}. Since $P_{\mu,0}f_0^-=f_0^-$ (cfr. \cite[Lemma A.5]{BMV3}) we have 
\begin{equation}
P_{2,0} f_0^- = 0\, . 
\end{equation}

\noindent {\bf Computation of $P_{1,1} f_j^\sigma$}. 
In case $f_j^\sigma\in\{f_1^+,f_1^-,f_0^- \}$  by \eqref{Psani}, \eqref{hP} and \eqref{primainversione1}, we have
\begin{subequations}
\begin{align} \notag
P_{1,1} f_j^\sigma &= 
 - \frac{1}{2\pi \im} \oint_{\Gamma} \frac{(\sL_{0,0}- \lambda)^{-1}}{\lambda} \cJ \cB_{1,1}  f_j^\sigma \de \lambda + \frac{1}{2\pi \im} 
 \oint_{\Gamma} \frac{(\sL_{0,0}- \lambda)^{-1}}{\lambda} \cJ \cB_{0,1} (\sL_{0,0}- \lambda)^{-1} \cJ  \cB_{1,0}  f_j^\sigma \de \lambda 
 \\ 
&  + \frac{1}{2\pi \im} 
 \oint_{\Gamma} \frac{(\sL_{0,0}- \lambda)^{-1}}{\lambda} \cJ \cB_{1,0} (\sL_{0,0}- \lambda)^{-1} \cJ  \cB_{0,1}  f_j^\sigma \de \lambda =: \mathrm{I\!I\!I}_j^\sigma +\mathrm{I\!V}_j^\sigma +\mathrm{V}_j^\sigma  \, , \label{P11.fj}
 \end{align} 
 whereas, by \eqref{primainversione2}, \eqref{murisolventeinterno}, \eqref{risolventeintermedio}
\begin{align} \notag
P_{1,1} f_0^+ 
 & =
-  \frac{1}{2\pi \im} \oint_{\Gamma} \frac{(\sL_{0,0}- \lambda)^{-1}}{\lambda} \cJ \cB_{1,1}  f_0^+ \de \lambda
 + \frac{1}{2\pi \im} \oint_{\Gamma} \frac{(\sL_{0,0}- \lambda)^{-1}}{\lambda^2} \cJ \cB_{1,1}  f_0^- \de \lambda\\
&  +\frac{1}{2\pi \im} 
 \oint_{\Gamma} \frac{(\sL_{0,0}- \lambda)^{-1}}{\lambda} \cJ \cB_{1,0} (\sL_{0,0}- \lambda)^{-1} \cJ  \cB_{0,1}  f_0^+ \de \lambda  =: \mathrm{I\!I\!I}_0^+ + \mathrm{I\!V}_0^+  +\mathrm{V}_0^+\, .
 \end{align}  
 \end{subequations}
When $f_j^\sigma=f_1^+$ one readily sees, in view of \eqref{P11acts} for $\mathrm{I\!I\!I}_1^+$, 
\eqref{1904_1129} for $\mathrm{I\!V}_1^+$ and \eqref{conti2} for $\mathrm{V}_1^+$, that $P_{1,1} f_1^+ \in \im \cW_0^- \oplus_\bR \im  \cW_2^- $ as stated in \eqref{secondjetsPmix}. Similarly, when $f_j^\sigma=f_0^+$ one has, in view of \eqref{P11acts} for $\mathrm{I\!I\!I}_0^+$ and  $\mathrm{I\!V}_0^+$ and \eqref{conti2} for $\mathrm{V}_0^+$, that $P_{1,1} f_0^+ \in  \im  \cW_1^- $ as stated in \eqref{secondjetsPmix}. 
 
 We now compute the remaining terms. 
 By Lemma \ref{lem:action11} and the residue theorem
 \begin{equation}
 \mathrm{I\!I\!I}_1^- = - {\im \mathrm{Q}_2^+} \, ,  \quad \mathrm{I\!I\!I}_0^-  = -\frac{\im }{2 \ch^{3/2}} f_{-1}^+  \, .
 \end{equation}
   By \eqref{1904_1129} we have
\begin{equation}
\mathrm{I\!V}_1^- = -\im \mu_\tth  {\mathrm J_2^+} , \quad
\mathrm{I\!V}_0^- = 0 \, .
\end{equation}
 By \eqref{conti2} we have 
\begin{equation}
\mathrm{V}_1^- = {\im \mathrm S_2^+} 
\, ,  
\quad
\mathrm{V}_0^- = 0\, .
\end{equation}
In conclusion we have formulae \eqref{secondjetsPmix} with 
\begin{equation}
 P_{1,1} f_1^- =  \underbrace{ - {\im \mathrm Q_2^+}
+ {\im \mathrm S_2^+} 
- \im \mu_\tth { \mathrm J_2^+} }_{\im {\scriptsize \vet{\ka_{1,1} \cos(2x)}{\kb_{1,1} \sin(2x)}}}
 \, ,
\quad  
{P_{1,1} f_0^- = -\frac{\im }{2 \ch^{3/2}} f_{-1}^+ }\, ,
\end{equation}
and we obtain the explicit expression \eqref{n11a11} of the coefficients given by
$$ \vet{\ka_{1,1} \cos(2x)}{\kb_{1,1} \sin(2x)} :=  -  \mathrm Q_2^+
+\mathrm S_2^+
-  \mu_\tth \mathrm J_2^+ $$
 with $\mathrm S_2^+$ in \eqref{S1-}, $\mu_\tth$ in \eqref{muh},  $\mathrm Q_2^+ $ in \eqref{1804_1855} and $\mathrm J_2^+$ in \eqref{JK}.\\
This concludes the proof of Lemma \ref{lem:secondjetsP}. \qed \medskip

\noindent {\it Proof of Lemma \ref{lem:P03acts}}. 

\noindent{\bf Computation of $P_{0,3} f_j^\sigma $}.   Similarly to \eqref{P02f0m} we have $P_{0,3}f_0^- = 0$ (as stated in  \eqref{P03acts}). Let us now compute
$ P_{0,3} f_1^+ $.
By \eqref{Psani}, \eqref{hP} and \eqref{primainversione1} 
\begin{align}  \notag
P_{0,3} f_1^+ &= 
- \frac{1}{2\pi \im} \oint_{\Gamma} \frac{(\sL_{0,0}- \lambda)^{-1}}{\lambda} \cJ \cB_{0,3}  f_1^+ \de \lambda \\ \notag
&  +\frac{1}{2\pi \im} 
 \oint_{\Gamma} \frac{(\sL_{0,0}- \lambda)^{-1}}{\lambda} \cJ \cB_{0,2} (\sL_{0,0}- \lambda)^{-1} \cJ  \cB_{0,1} f_1^+ \de \lambda \\  \notag
 & +\frac{1}{2\pi \im} 
 \oint_{\Gamma} \frac{(\sL_{0,0}- \lambda)^{-1}}{\lambda} \cJ \cB_{0,1} (\sL_{0,0}- \lambda)^{-1} \cJ  \cB_{0,2} f_1^+ \de \lambda  \\ \notag
 & - \frac{1}{2\pi \im} 
 \oint_{\Gamma} \frac{(\sL_{0,0}- \lambda)^{-1}}{\lambda} \cJ \cB_{0,1} (\sL_{0,0}- \lambda)^{-1} \cJ  \cB_{0,1} (\sL_{0,0}- \lambda)^{-1} \cJ  \cB_{0,1}  f_1^+ \de \lambda  \\
 & =: \mathrm{V\!I} + \mathrm{V\!I\!I} + \mathrm{V\!I\!I\!I} + \mathrm{I\!X} \, . \label{P03formal}
 \end{align}
  We now compute these four terms. 
  
 \noindent $\mathrm{V\!I} \big)$ 
By \eqref{B3sano} we have 
${\cal B}_{0,3} = \begin{bmatrix} a_3(x) & -p_3(x)\pa_x \\ \pa_x\circ p_3(x) & 0 \end{bmatrix}$ with $p_3(x), a_3(x)$ in \eqref{pino3}-\eqref{aino3}. Then 
$   {\cal B}_{0,3} f_0^-  = 0 $ whereas   $\cJ {\cal B}_{0,3} f_1^+  = \alpha_{10}  f_0^- + {\mathrm{W}_{2}^-}    +  f_{\cW_4^-}$ where  $\alpha_{10} \in \bR$ and 
\begin{equation}
 {\mathrm{W}_{2}^-} (x) := 
\sepvet{
- \ch^{\frac12} (p_3^{[1]} +  p_3^{[3]}) \, \sin(2x)}{
\dfrac{- \ch (a_3^{[1]} + a_3^{[3]}) + (p_3^{[1]} + p_3^{[3]})}{2 \ch^{\frac12}} \cos(2x) } \ . 
\end{equation}
Hence by  \eqref{primainversione4} we get 
\begin{equation*}
(\sL_{0,0}- \lambda)^{-1}\cJ {\cal B}_{0,3} f_1^+  = 
 \sL_{0,0}^{\, \inv} ({\mathrm{W}_{2}^-} +  f_{\cW_4^-}) + \cO(\lambda^{-1}:\lambda) 
 \, , 
\end{equation*}
and, by \eqref{residuethm}, the term $\mathrm{V\!I}$ in \eqref{P03formal} is
\begin{equation}\label{VIfinal}
 \mathrm{V\!I} = - \sL_{0,0}^{\, \inv} {\mathrm{W}_{2}^-} 
 + \widetilde f_{\cW_4^+}   \, ,
\end{equation}
 with
\begin{equation}\label{L0b3}
 \sL_{0,0}^{\, \inv} {\mathrm{W}_{2}^-} (x) = - \sepvet{\kappa_1 \cos(2x)}{\varsigma_1  \sin(2x) }\, ,\quad \begin{aligned} \kappa_1 :&=\dfrac{\ch(a_3^{[1]} +a_3^{[3]} ) - (\ch^4+2) (p_3^{[1]}+p_3^{[3]})}{2 \ch^{9/2}} \, ,\\ \varsigma_1:&=\dfrac{(\ch^4+1) (\ch (a_3^{[1]}+a_3^{[3]})-2 (p_3^{[1]}+p_3^{[3]}))}{4 \ch^{11/2}} \, . \end{aligned} 
 \end{equation}
  \noindent $\mathrm{V\!I\!I} \big)$  
By Lemma \ref{lem:action1}  and since $\cB_{0,2} f_0^- = 0$, we get
\begin{align}\label{VII}
&(\sL_{0,0}- \lambda)^{-1} \cJ \cB_{0,2} (\sL_{0,0}- \lambda)^{-1} \cJ  \cB_{0,1} f_1^+ = 
 (\sL_{0,0}- \lambda)^{-1} \cJ \cB_{0,2}\mathrm{A}_2^+\\
 \notag
& \qquad  +\lambda (\sL_{0,0}- \lambda)^{-1} \cJ \cB_{0,2} f_{\cW_2^-} +\lambda^2 (\sL_{0,0}- \lambda)^{-1} \cJ \cB_{0,2} f_{\cW_2^+} +\cO(\lambda) \, .
\end{align}
Applying the operators $\cB_{0,2} $ in \eqref{B2sano} and $\cJ$ in \eqref{PoissonTensor} to the vector $\mathrm A_2^+$ in \eqref{risolventeintermedio.1} one obtains 
\begin{subequations}\label{VIIaux}
\begin{equation}
\cJ \cB_{0,2} \mathrm A_2^+  =  {\mathrm{X}_2^-}+f_{\cW_0^-}  + f_{\cW_4^-} \, , 
\end{equation}
where
\begin{equation}\label{A02p}
{\mathrm{X}_2^-}(x) := 
\sepvet{-\dfrac{(\ch^4-1)^2 \ttf_2 (a_1^{[1]} \ch-2 p_1^{[1]})+\ch(\ch^4+1)p_2^{[0]} \big( (\ch^4+2 ) p_1^{[1]}-a_1^{[1]} \ch \big)}{\ch^{11/2} (\ch^4+1)}\sin (2 x)}{-\dfrac{a_2^{[0]} \ch \big((\ch^4+2) p_1^{[1]}-a_1^{[1]} \ch\big)+(\ch^4+1) p_2^{[0]} (a_1^{[1]} \ch-2 p_1^{[1]})}{2 \ch^{11/2}} \cos (2 x)}  \ .
\end{equation}
On the other hand, in view of \eqref{Bjkrev} and Lemma \ref{lem:VUW}, one has
\begin{equation}
\lambda (\sL_{0,0}- \lambda)^{-1} \cJ \cB_{0,2} f_{\cW_2^-} = \cO(\lambda) \, ,\quad \lambda^2 (\sL_{0,0}- \lambda)^{-1} \cJ \cB_{0,2} f_{\cW_2^+}  = \cO(\lambda)\, .
\end{equation}
\end{subequations}
Then  Lemmata \ref{lem:action02},  \ref{lem:VUW}, \eqref{VII} and \eqref{VIIaux} give 
  $$
  \begin{aligned}
  &(\sL_{0,0}-\lambda)^{-1} \cJ \cB_{0,2} (\sL_{0,0}- \lambda)^{-1} \cJ  \cB_{0,1} f_1^+ =
    \sL_{0,0}^{\, \inv} ({\mathrm{X}_2^-} + f_{\cW_4^-})  + \cO(\lambda^{-1}:\lambda)\, ,
  \end{aligned}
   $$
and, by \eqref{residuethm}, the term $\mathrm{V\!I\!I}$ in \eqref{P03formal} is
  \begin{align}\label{VIIfinal}
   \mathrm{V\!I\!I} &= \sL_{0,0}^{\, \inv}{\mathrm{X}_2^-}  + \widetilde f_{\cW_4^+}\, ,
  \end{align}
  with, by \eqref{L0inv} and \eqref{A02p}, 
  \begin{equation}\label{LbotA02}
  \sL_{0,0}^{\,\inv} {\mathrm{X}_2^-}(x) = \vet{\kappa_2 \cos(2x)}{\varsigma_2 \sin(2x) }
  \end{equation} 
  with 
  $$ 
  \begin{aligned}
&{\footnotesize \kappa_2 :=}  {\scriptsize \begin{matrix} \dfrac{a_1^{[1]} \ch \big(a_2^{[0]} \ch^2+(\ch^4-1)^2 \ttf_2-2 (\ch^4+1) \ch p_2^{[0]}\big)+p_1^{[1]} \big(-a_2^{[0]} (\ch^4+2) \ch^2-2 (\ch^4-1)^2 \ttf_2+(\ch^8+5 \ch^4+4) \ch p_2^{[0]}\big)}{2 \ch^{21/2}} \end{matrix}} \, , \\ 
&{\footnotesize \varsigma_2} :=  {\scriptsize \begin{matrix}\dfrac{a_1^{[1]} \ch \big(a_2^{[0]} \ch^2(\ch^4+1)+(\ch^4-1)^2 \ttf_{2}-\ch (\ch^8+3 \ch^4+2) p_2^{[0]}\big)+p_1^{[1]} \big(-a_2^{[0]} (\ch^8+3 \ch^4+2) \ch^2-2 (\ch^4-1)^2 \ttf_{2}+ (3 \ch^8+7 \ch^4+4) \ch p_2^{[0]}\big)}{4 \ch^{23/2}} \end{matrix}.}
\end{aligned}
  $$
 \noindent $\mathrm{V\!I\!I\!I}\big)$ 
By Lemma \ref{lem:action02}, \eqref{1904_1446} and Lemma \ref{lem:VUW}
we get
\begin{align}\label{VIIIstart}
(\sL_{0,0}&-\lambda)^{-1} \cJ \cB_{0,1} (\sL_{0,0}-\lambda)^{-1} \cJ \cB_{0,2}  f_1^+ =\\ \notag
&\frac{\tau_1^+}{\lambda} (\sL_{0,0}-\lambda)^{-1} \cJ \cB_{0,1} f_1^- + \frac{\ell_{1}^+}{\lambda^2+4\text{$\ch^2$}}  (\sL_{0,0}-\lambda)^{-1} \cJ \text{$\cB_{0,1}$}  f_{-1}^+ + \sL_{0,0}^{\, \inv}\cJ  \cB_{0,1} {\mathrm L_3^+} \\  \notag  
 + &\frac{\lambda m_1^+}{\lambda^2+4\ch^2} (\sL_{0,0}-\lambda)^{-1} \cJ \cB_{0,1} f_{-1}^-  +  (\sL_{0,0}-\lambda)^{-1} \cJ \cB_{0,1} \cO_{\cW_3}(\lambda) + \cO(\lambda)  \, .
\end{align}
Applying the operators $\cB_{0,1} $ in \eqref{B1sano} and $\cJ$ in \eqref{PoissonTensor} to the vector {$\mathrm L_3^+$}  in \eqref{1904_1446} one obtains 
\begin{subequations}\label{VIIIaux}
\begin{equation}
\cJ \cB_{0,1} {\mathrm L_3^+}  = {\mathrm{Y}_2^-} + f_{\cW_4^-}\, ,  
\end{equation}
where 
\begin{equation}\label{A01p}
{\mathrm{Y}_2^-}(x) := \sepvet{\dfrac{p_1^{[1]} \big(a_2^{[2]} \ch (3 + \ch^4) - 2 (3 + 5 \ch^4) p_2^{[2]}\big)}{16 \ch^{9/2}}\sin(2x)}{-\dfrac{3 (1 + 3 \ch^4) p_1^{[1]} (a_2^{[2]} \ch - 2 p_2^{[2]}) + 
 a_1^{[1]} \ch \big(-a_2^{[2]} \ch (3 + \ch^4) + 2 (3 + 5 \ch^4) p_2^{[2]}\big)}{32 \ch^{11/2}}\cos(2x)}\, .
\end{equation}
 In view of \eqref{Bjkrev} and Lemma \ref{lem:VUW}, one obtains by inspection 
 \begin{equation}
 \frac{\lambda m_1^+}{\lambda^2+4\ch^2} (\sL_{0,0}-\lambda)^{-1} \cJ \cB_{0,1} f_{-1}^- = \cO(\lambda)\, ,\quad  (\sL_{0,0}-\lambda)^{-1} \cJ \cB_{0,1} \cO_{\cW_3}(\lambda)  =\cO(\lambda)\, .
 \end{equation}
 \end{subequations}
Then \eqref{VIIIstart},  Lemmata \ref{lem:action1}, \ref{lem:VUW} and \eqref{VIIIaux} give
\begin{equation*}
{\footnotesize \begin{matrix} (\sL_{0,0}-\lambda)^{-1} \cJ \cB_{0,1} (\sL_{0,0}-\lambda)^{-1} \cJ \cB_{0,2}  f_1^+ \end{matrix} }= \tau_1^+ {\mathrm B_2^+} + \frac{\ell_{1}^+}{4\text{$\ch^2$}}  {\mathrm A_{-2}^+} +\sL_{0,0}^{\, \inv}( {\mathrm{Y}_2^-} +f_{\cW_4^-} )   + \cO(\lambda^{-1}:\lambda)  \, ,
\end{equation*}
and, by \eqref{residuethm}, the term $\mathrm{V\!I\!I\!I}$ in \eqref{P03formal} is
 \begin{equation}\label{VIIIfinal}
\mathrm{V\!I\!I\!I} =  \tau_1^+ { \mathrm{B}_2^+ }
+ \frac{\ell_1^+}{4\ch^2} {\mathrm{A}_{-2}^+} +
\sL_{0,0}^{\, \inv} {\mathrm{Y}_2^-} + \widetilde f_{\cW_4^+}  
\, ,
\end{equation}
where
\begin{equation}\label{LbotA03n}
\sL_{0,0}^{\, \inv} {\mathrm{Y}_2^-}(x) = \sepvet{\kappa_3 \cos(2x)}{\varsigma_3 \sin(2x)} \, .
\end{equation}
 with
 \begin{align*} \kappa_3 :&= \dfrac{a_1^{[1]} \ch \big(a_2^{[2]} \ch (\ch^4+3)-2 (5 \ch^4+3)p_2^{[2]}\big)-a_2^{[2]} \ch (\ch^8+13 \ch^4+6) p_1^{[1]}+2 (\ch^4+3) (5 \ch^4+2) p_1^{[1]} p_2^{[2]}}{32 \ch^{19/2}} \, ,\\
 \varsigma_3 :&=  \dfrac{(\ch^4+1)\big[ a_1^{[1]} \ch\big( a_2^{[2]} \ch (\ch^4+3)-2(5 \ch^4+3) p_2^{[2]} \big)-2 a_2^{[2]} \ch (5 \ch^4+3) p_1^{[1]}+4 (7 \ch^4+3) p_1^{[1]} p_2^{[2]}\big]}{64 \ch^{21/2}}\, .
 \end{align*}
 
\noindent $\mathrm{I\!X} \big)$ By \eqref{2104_1749} we get 
\begin{subequations}\label{IXstart}
\begin{align}
&[(\sL_{0,0}-\lambda)^{-1} \cJ \cB_{0,1}]^3  f_1^+    
  = \frac{\zeta_2^+ }{\lambda} (\sL_{0,0}-\lambda)^{-1} \cJ \cB_{0,1} f_1^-    
\\ \notag 
&+ \frac{\alpha_2^+}{\lambda^2+4\ch^2} (\sL_{0,0}-\lambda)^{-1} \cJ \cB_{0,1} f_{-1}^+ +\zeta_3^+ (\sL_{0,0}-\lambda)^{-1} \cJ \cB_{0,1} f_1^+ 
+(\sL_{0,0}-\lambda)^{-1} \cJ \cB_{0,1} {\mathrm A_3^+} \\ 
&\begin{aligned}
\label{2605_1113}
&+\lambda (\sL_{0,0}-\lambda)^{-1} \cJ \cB_{0,1} \left(  f_{\cW_1^-}  + f_{\cW_3^-} \right)
+
\lambda^2  (\sL_{0,0}-\lambda)^{-1} \cJ \cB_{0,1}  f_{\cW_1^+} \\
& 
+ (\sL_{0,0}-\lambda)^{-1} \cJ \cB_{0,1}  \cO_{\cW_3}(\lambda^2) 
+  (\sL_{0,0}-\lambda)^{-1} \cJ \cB_{0,1}  \cO(\lambda^3)
 \, .\end{aligned}
\end{align}
\end{subequations}
In view of \eqref{Bjkrev} and Lemma \ref{lem:VUW} the terms in the two lines in \eqref{2605_1113} are
 \begin{align*}
& \lambda (\sL_{0,0}-\lambda)^{-1} \cJ \cB_{0,1}\left(   f_{\cW_1^-}
+ f_{\cW_3^-} \right) 
 = \cO(\lambda)\, ,\ \  \lambda^2 (\sL_{0,0}-\lambda)^{-1} \cJ \cB_{0,1}  f_{\cW_1^+} =\cO(\lambda)\, ,\  \\
&  \ (\sL_{0,0}-\lambda)^{-1} \cJ \cB_{0,1}  \cO_{\cW_3}(\lambda^2) =\cO(\lambda^2)\,  , 
\quad
 (\sL_{0,0}-\lambda)^{-1} \cJ \cB_{0,1}  \cO(\lambda^3) = \cO(\lambda)\, .
  \end{align*}
The remaining terms in \eqref{IXstart}, again by Lemma \ref{lem:action1}, are
\begin{align}\label{2605_1127}
&[(\sL_{0,0}-\lambda)^{-1} \cJ \cB_{0,1}]^3  f_1^+    
  =  \zeta_2^+ { \mathrm B_2^+} + \frac{\alpha_2^+}{4\text{$\ch^2$}} 
  {\mathrm A_{-2}^+  } 
+\zeta_3^+\mathrm A_2^+
 + \sL_{0,0}^{\, \inv} \cJ \cB_{0,1} { \mathrm A_3^+ }
  + \cO(\lambda^{-1}:\lambda)\, .
\end{align}
By
applying the operators $\cB_{0,1} $ in \eqref{B1sano} and $\cJ$ in \eqref{PoissonTensor} to the vector $ {\mathrm A_3^+}$  in \eqref{2104_1446} one gets
\begin{equation}\label{2605_1128}
\cJ \cB_{0,1} { \mathrm A_3^+}  =  {\mathrm{Z}_2^-}+ f_{\cW_4^-}  
\end{equation}
where 
\begin{equation}\label{A04p}
{\mathrm{Z}_2^-}(x):= 
- \sepvet{
\frac{p_1^{[1]} \left((a_1^{[1]})^2 \left(\ch^4+3\right) \ch^2
-2 a_1^{[1]}p_1^{[1]}
\left(\ch^8+9 \ch^4+6\right) \ch +
(p_1^{[1]})^2
\left(11 \ch^8+29 \ch^4+12\right) 
\right)}{32 \ch^{19/2}}\sin (2 x)}{\frac{ (a_1^{[1]}
\ch-p_1^{[1]})
 \left( ( a_1^{[1]})^2 \left(\ch^4+3\right) \ch^2
 -2 a_1^{[1]}  p_1^{[1]} \left(\ch^8+13 \ch^4+6\right) \ch
  +3(p_1^{[1]})^2 \left(9 \ch^8+15 \ch^4+4\right) \right)}{64 \ch^{21/2}} \cos (2 x)}\, . 
\end{equation}
Finally, by \eqref{2605_1127} and \eqref{2605_1128}, we  get that  the term $\mathrm{I\!X}$ in \eqref{P03formal} is
\begin{equation}\label{IXfinal}
 \mathrm{I\!X} = - \zeta_2^+ {\mathrm B_2^+} - \frac{\alpha_2^+}{4\text{$\ch^2$}} {\mathrm A_{-2}^+}   
-\zeta_3^+\mathrm A_2^+
 - \sL_{0,0}^{\, \inv} \left( {\mathrm{Z}_2^-}  + f_{\cW_4^-} \right) \, ,
 \end{equation}
 where
 \begin{equation}\label{LbotA04n}
 \sL_{0,0}^{\, \inv}  {\mathrm{Z}_2^-}
  = 
 \sepvet{ \kappa_4 \cos (2 x)}{ \varsigma_4 \sin (2 x)}
 \end{equation}
 where
 \begin{equation*}
 \begin{aligned}
&  \kappa_4 := 
 \frac1{64 \ch^{29/2}} \big( -(a_1^{[1]})^3 (\ch^4+3) \ch^3
 +(a_1^{[1]})^2 p_1^{[1]}
 (3 \ch^8+31 \ch^4+18)
  \ch^2   \\ 
  &\qquad  -a_1^{[1]}(p_1^{[1]})^2
   (2 \ch^{12}
  +49 \ch^8+101 \ch^4+36) 
  \ch +  (p_1^{[1]})^3
  (11 \ch^{12}+67 \ch^8+86 \ch^4+24)
   \big)  \\
 &   \varsigma_4 :=
 \frac{(\ch^4+1)}{{128 \ch^{31/2}} }   \big(-(a_1^{[1]})^3 (\ch^4+3) \ch^3+2 (a_1^{[1]})^2 p_1^{[1]}
 (\ch^8+14 \ch^4+9) \ch^2\\
 &\qquad-
 a_1^{[1]}(p_1^{[1]})^2 (31 \ch^8+89 \ch^4+36) \ch +2
 (p_1^{[1]})^3 (19 \ch^8+37 \ch^4+12) \big)
 \end{aligned}
 \end{equation*}
 
In conclusion, by \eqref{P03formal}, \eqref{VIfinal}, \eqref{VIIfinal}, \eqref{VIIIfinal} and \eqref{IXfinal} we deduce that
$$
P_{0,3} f_1^+ = 
 \sL_{0,0}^{\, \inv}\left( - {\mathrm W_2^-}
+ {\mathrm X_2^-} + {\mathrm Y_2^-}
  -  {\mathrm Z_2^-}\right)
+(\tau_1^+- \zeta_2^+) { \mathrm{B}_2^+} 
+ \frac{\ell_1^+- \alpha_2^+}{4\ch^2} { \mathrm{A}_{-2}^+} 
-\zeta_3^+\mathrm A_2^+   +  \widetilde f_{\cW_4^+} 
$$
which
proves the expansion of 
$ P_{0,3} f_1^+$  in \eqref{P03acts} with
\begin{align*}
&\vet{\ka_{0,3}\cos(2x)}{\kb_{0,3}\sin(2x)} := \vet{(\kappa_1 +\kappa_2 + \kappa_3 - \kappa_4 )\cos(2x) }{(\varsigma_1 +\varsigma_2 + \varsigma_3 - \varsigma_4 )\sin(2x)}\\
&\qquad + (\tau_1^+ -\zeta_2^+) \mathrm B_2^+(x) + \frac{\ell_1^+- \alpha_2^+}{4\ch^2} \mathrm{A}_{-2}^+(x) 
-\zeta_3^+\mathrm A_2^+(x)  \, ,
\end{align*}  
with $\kappa_i$, $\varsigma_i$, $i=1,\dots,4$ in \eqref{L0b3}, 
\eqref{LbotA02}, \eqref{LbotA03n},
\eqref{LbotA04n}, $\mathrm B_2^+ $, $\mathrm{A}_{-2}^+$ $\mathrm A_2^+$ in  \eqref{risolventeintermedio.1}, $\zeta_2^+$, $\alpha_2^+$, $\zeta_3^+$  in \eqref{2104_1446} and $\ell_1^+$, $\tau_1^+$ in
\eqref{1904_1446},
resulting in the coefficients $\ka_{0,3}$ and $\kb_{0,3}$ in \eqref{n03}.
  
\noindent{\bf Computation of $P_{1,2} f_0^- $.}  By \eqref{Psani}, \eqref{hP} and the fact 
that 
 $\cB_{1,0} f_0^- =  \cB_{0,1}f_0^- =   \cB_{0,2}f_0^- =  0$, the  term  $ P_{1,2} f_0^- $ reduces to
 \begin{align}\notag
 P_{1,2}f_0^- &= 
-\frac{1}{2\pi \im} \oint_{\Gamma} \frac{(\sL_{0,0}- \lambda)^{-1}}{\lambda} \cJ \cB_{1,2} f_0^- \de \lambda \\ 
 & \ +\frac{1}{2\pi \im} 
 \oint_{\Gamma} \frac{(\sL_{0,0}- \lambda)^{-1}}{\lambda} \cJ \cB_{0,1} (\sL_{0,0}- \lambda)^{-1} \cJ  \cB_{1,1}f_0^- \de \lambda  \label{P12formal} =: \mathrm{X} + \mathrm{X\!I}\, .
 \end{align}
 We now compute the two terms. 
 
\noindent $\mathrm{X}\big)$ By \eqref{cBacts}  we have
$\cJ \cB_{1,2} f_0^- =\im  a_3 f_0^- + {\im\mathrm{W}_{-2}^-} $
with 
$ 
{\mathrm{W}_{-2}^-} 
(x) := p_2^{[2]} {\footnotesize \vet{0}{ \cos (2 x)}} \, 
$
and, by \eqref{primainversione},
\begin{equation} \label{Xfinal}
\mathrm{X} \stackrel{\eqref{residuethm} }{=}
-\im \sL_{0,0}^{\, \inv}
 {\mathrm{W}_{-2}^-}\, ,\quad \sL_{0,0}^{\, \inv}
{\mathrm{W}_{-2}^-}(x) \stackrel{\eqref{L0inv}}{=}
 p_2^{[2]} \vet{\ch^{-4}\cos (2 x)}{\dfrac{1+\ch^4}{2\ch^5}\sin (2 x)}\, .
\end{equation}

\noindent $\mathrm{X\!I}\big)$ By Lemmata \ref{lem:action11}, \ref{lem:action1} one has
$$(\sL_{0,0}-\lambda)^{-1} \cJ \cB_{0,1} (\sL_{0,0}- \lambda)^{-1} \cJ  \cB_{1,1} f_0^-  = \im \ch^{-\frac12} {\mathrm{B}_2^+ }+ \im \frac{1}{2 \ch^{3/2}} { \mathrm{A}_{-2}^+}  +\cO(\lambda^{-1}: \lambda)\, ,$$
and therefore 
\begin{equation}  \label{XIfinal}
\mathrm{X\!I} = \im \ch^{-\frac12} {\mathrm{B}_2^+ }+ \im \frac{1}{2 \ch^{3/2}} {\mathrm{A}_{-2}^+} \, .
\end{equation}
 In conclusion, by \eqref{P12formal}, \eqref{Xfinal} and \eqref{XIfinal}
 \begin{equation}
 P_{1,2} f_0^- = -\im (\sL_{0,0}^{\, \inv}) {\mathrm{W}_{-2}^-} + \im\ch^{-\frac12} 
 {\mathrm{B}_2^+} +\im \frac{1}{2 \ch^{3/2}}{\mathrm{A}_{-2}^+} \, ,
  \end{equation}
  which, in view of \eqref{Xfinal}, 
proves the expansion of 
$  P_{1,2} f_0^-  $  in \eqref{P03acts} with
\begin{align*}
\vet{\ka_{1,2}\cos(2x)}{\kb_{1,2} \sin(2x) }:= \vet{-p_2^{[2]} \ch^{-4} \cos(2x)}{- p_2^{[2]}\frac{1+\ch^4}{2\ch^5}\sin(2x) }+ \ch^{-\frac12} 
 \mathrm{B}_2^+(x) + \frac{1}{2 \ch^{3/2}}\mathrm{A}_{-2}^+(x)\, ,
\end{align*}
with $\mathrm{B}_2^+$ and $\mathrm{A}_{-2}^+$ in \eqref{risolventeintermedio.1},
resulting in the coefficients $\ka_{1,2}$ and $\kb_{1,2}$ given in \eqref{n03}.

\noindent{\bf Computation of $P_{2,1} f_0^- $.}  By \eqref{Psani} and the fact 
that 
 $\cB_{1,0} f_0^- =  \cB_{0,1}f_0^- $ and $ \cB_{2,1} =  0$ the term  $ P_{2,1} f_0^- $ reduces to
 \begin{align}\notag
 P_{2,1} f_0^- &= 
 \frac{1}{2\pi \im} 
 \oint_{\Gamma} \frac{(\sL_{0,0}- \lambda)^{-1}}{\lambda} \cJ \cB_{0,1} (\sL_{0,0}- \lambda)^{-1} \cJ  \cB_{2,0} f_0^- \de \lambda  \\ \notag
 &+\frac{1}{2\pi \im} 
 \oint_{\Gamma} \frac{(\sL_{0,0}- \lambda)^{-1}}{\lambda} \cJ \cB_{1,0} (\sL_{0,0}- \lambda)^{-1} \cJ  \cB_{1,1} f_0^- \de \lambda   \, .
 \end{align}
By repeated use of \eqref{Bjkrev} and Lemma \ref{lem:VUW} one  finds that $P_{2,1} f_0^- \in \cW_1^- $ as stated in \eqref{P03acts}.

 \begin{footnotesize}

 \end{footnotesize}
\end{document}